%% file: aswdfinal.tex
\documentclass[11pt]{amsart}

\usepackage{amsmath, amsthm, amsfonts, amssymb, amscd}
\usepackage[mathscr]{eucal}
\usepackage[dvips]{graphicx}
\usepackage[dvips]{color}

\input xy
\xyoption{all}

\newcommand{\iotaA}{i}

\newcommand{\Z}{\mathbf Z}
\newcommand{\Q}{\mathbf Q}
\newcommand{\C}{\mathbf C}
\newcommand{\R}{\mathbf R}
\newcommand{\PP}{\mathbf P}
\newcommand{\PSL}{\mathrm {PSL}}

\newcommand{\SL}{{\mathrm{SL} }}

\newcommand{\msr}[1]{\mathscr{#1}}
\newcommand{\mbf}[1]{\mathbf{#1}}

\newcommand{\mr}[1]{\mathrm{#1}}

\newcommand{\HH}{\mbf{H}}

\newcommand{\fq}{\mathbf{F}_q}

\DeclareMathOperator{\Gal}{Gal}
\DeclareMathOperator{\Frob}{Frob}
\DeclareMathOperator{\Tr}{Tr}

\newcommand{\abcd}[4]{\left(
        \begin{smallmatrix}#1&#2\\#3&#4\end{smallmatrix}\right)}

\theoremstyle{plain}
\newtheorem{thm}{Theorem}[section]

\newtheorem{Def}[thm]{Definition}

\numberwithin{thm}{subsection}

\newcommand{\comment}[1]{}

\begin{document}

\title[Modular forms, noncongruence subgroups, ASwD
relations]
{Modular forms on noncongruence subgroups and Atkin-Swinnerton-Dyer
relations}

\thanks {This research was carried out as part of
an REU summer program at LSU, supported by the
National Science Foundation grant DMS-0353722 and a
Louisiana Board of Regents Enhancement grant,
LEQSF (2002-2004)-ENH-TR-17.
The last author was partially supported by grants
LEQSF (2004-2007)-RD-A-16 and NSF award DMS-0501318.
}

\author[Fang, Hoffman, Linowitz, Rupinski, Verrill]
{Liqun Fang, J. William Hoffman, Benjamin Linowitz,
        Andrew Rupinski, Helena Verrill}
\address[Fang, Hoffman and Verrill]
      {Mathematics Department \\Louisiana State University \\
 Baton Rouge 70803 \\Louisiana}
\address[Linowitz]{Mathematics Department \\Dartmouth College\\
Hanover 03755\\New Hampshire}
\address[Rupinski]{Mathematics Department \\ University of Pennsylvania \\
Philadelphia 19104 \\Pennsylvania}

\date{\today}
\email[Fang]{liqun@math.lsu.edu}
\email[Hoffman]{hoffman@math.lsu.edu}
\email[Linowitz]{Benjamin.D.Linowitz@dartmouth.edu}
\email[Rupinski]{rupinski@math.upenn.edu}
\email[Verrill]{verrill@math.lsu.edu}

\begin{abstract}
We give new examples of   noncongruence subgroups $\Gamma\subset
\SL_2(\Z)$ whose space of weight 3 cusp forms $S_3(\Gamma
)$ admits a basis satisfying the Atkin-Swinnerton-Dyer congruence
relations with respect to a weight 3 newform for a certain
congruence subgroup.
\end{abstract}

\maketitle


\section{Introduction}
\label{S:intro}

A finite index subgroup of $\SL_2(\Z)$
is {\it noncongruence} if it
does not contain
$
\Gamma (N)
$
for any $N \ge 1$.
The study of modular forms on such subgroups was initiated
by Atkin and Swinnerton-Dyer who 
discovered experimentally
the congruences now bearing their
names \cite{ASwD}. 
Subsequently, Scholl proved congruences
satisfied by the coefficients of modular forms on noncongruence
subgroups \cite{Sch85i,Sch85ii,Sch87,Sch88,Sch93}.
A refined conjecture has recently been put forward by
Atkin, Li, Long and  Yang \cite{LLY},\cite{ALL},
\cite{LL}. See \cite{LLY2} for a general survey of this.

In this paper we give
new examples of noncongruence subgroups 
having a basis of cuspidal modular forms satisfying the
Atkin-Swinnerton-Dyer (ASwD) congruences.
We only give experimental evidence of our results,
obtained using {\sc Magma}
\cite{magma}, Mathematica,  and PARI \cite{PARI2}.
In a later publication, we will give a detailed 
treatment of one of our examples.

\subsection{Notation}
\label{SS:back1}
We assume familiarity with the action of $\SL_2(\R)$
on the upper half complex plane $\HH$, 
with congruence subgroups such as
$\Gamma_0(N)$, $\Gamma_1(N)$, $\Gamma^0(N)$, $\Gamma^1(N)$,
and with
$M_k (\Gamma )$ and $S_k (\Gamma )$ the
finite-dimensional vector spaces of modular forms and
cusp forms for $\Gamma$, and $S_k (\Gamma _0 (N), \chi
)$ the space of cusp forms with character $\chi : (\Z /
N)^{\ast}\to \C ^{\ast}$. 

It is well known 
(see \cite{shimu} for details)
that $S_{k}(\Gamma _0 (N),\chi )$
has a basis
of Hecke eigenforms, which have $q$-expansions
\[
f(z)=\sum_{n\ge 1} a_n (f)q^n, \quad \mr{where\ \ }q = \exp(2\pi iz),
\]
with
$a_n$ satisfying the relations 
\begin{equation}
\label{eqn:Heckerelation} a_{np} - a_pa_n + \chi(p)p^{k-1}a_{n/p} =
0,\quad a_n = a_n(f)
\end{equation}
for all positive integers $n$ and primes $p\not| N$, taking $a_{n/p}=0$ if
$p\not|n$.

\subsection{Atkin--Swinnerton-Dyer congruences}
\label{SS:back2}

If $\Gamma$ is a  {\it noncongruence} subgroup, then 
 $S_k(\Gamma)$ has no basis of
forms satisfying (\ref{eqn:Heckerelation}).  Instead, it is conjectured
that certain congruences hold, as in the following definition.

\begin{Def}[\cite{LLY}]
\label{def:LLYASWD}
Suppose that the noncongruence subgroup
 $\Gamma$ has cusp width $\mu$ at infinity, and that
$h\in S_k(\Gamma )$ has
an $M$-integral $q^{1/\mu}$-expansion
$h=\sum a_n(h)q^{n/\mu}$ for some $M\in\Z$.
(cf \cite[Proposition 5.2]{Sch85ii}).
Let $f = \sum c_n(f)q^n$ be a
normalized newform of weight $k$, level $N$, character $\chi$.
The forms $h$ and $f$ are said to satisfy the
Atkin-Swinnerton-Dyer congruence relation if, for all primes $p$
not dividing $MN$ and for all $n \ge 1$,
\begin{eqnarray}\label{e:0.8}
 (a_{np}(h) - c_p(f)a_n(h) + \chi(p)p^{k-1}a_{n/p}(h))/(np)^{k-1}
 \end{eqnarray}
is integral at all places dividing $p$.
\end{Def}

\begin{Def}
\label{def:aswdbasis}
We say that $S_k (\Gamma)$ has an ASwD basis if there
is a basis $h_1, ..., h_n$ of  $S_k (\Gamma)$
and normalized newforms $f_1, ..., f_n$ such that 
each pair $(h_i, f_i)$ satisfies the
ASwD congruence relation in Definition~\ref{def:LLYASWD}.
\end{Def}
Note that, in the above definition, the choices of $h_1, ..., h_n$
and of $f_1, ..., f_n$ may depend on the prime number $p$. There are
examples known where the same $h_i$ and  $f_j$ work for every 
prime $p$ (actually all but a finite number of exceptional primes). 
On the other hand, there are examples known where 
the choice of the ASwD basis depends on the value of $p$ modulo 
some modulus $N$ (see examples 2 and 3 in the tables below).

\section{Statement of results}
\label{SS:statement}
\subsection{Tables.}
\label{SS:tables}
For the noncongruence subgroups $\Gamma $ 
considered, there are two main issues addressed:
\begin{enumerate}
 \item Modularity of the $l$-adic Scholl's representation 
attached to the cusp forms of weight 3, $S_3 (\Gamma )$.
\item Giving a basis of $S_3 (\Gamma )$ that satisfies ASwD
congruences.
\end{enumerate}
In our cases the dimension of  $S_3 (\Gamma )$ is 2 so the 
$l$-adic representation is 4 dimensional. We find that this 
4-dimensional representation breaks up into two 2-dimensional 
pieces, each of which is isomorphic to the 2-dimensional 
representations that Deligne constructed for 
Hecke eigenforms
$f$ on congruence subgroups. Thus, each $S_3 (\Gamma )$
should be associated to a pair $f_1, f_2$ of Hecke eigenforms
on {\it congruence} subgroups. In the 
examples, these are one and the same form, or conjugate forms
or base extensions of one form to a quadratic extension of
$\Q$.

In Tables~\ref{tab:summary1},
\ref{tab:summary2}, 
\ref{tab:summary3},
\ref{tab:summary4}, 
we define modular forms
$h_1,h_2,f$, where $h_1$ and $h_2$ span $S_3(\Gamma)$ for
the noncongruence subgroup $\Gamma$
given in Definition~\ref{def:8groups}, and $f$ is a weight $3$ Hecke
eigenform for some congruence subgroup.  
For each group we give a basis $(h_1,h_2)$
of $S_3(\Gamma)$, in some cases depending on the prime $p$,
and a newform $f$ with
$(h_i,f)$ satisfying the ASwD congruence relation.
Most forms are given  in terms of the Dedekind eta function,
\begin{equation}
\eta(z)= q^{1/24} \prod_{n = 1} ^{\infty} (1 - q^{n}), 
\;\; \text{ where }\; q = e^{2\pi i z}. 
\label{eqn:Ded_eta_def}
\end{equation}
Our experiments support the following:
\begin{thm}
\label{T:modularity}
Let $\rho$ be the $l$-adic representation constructed by Scholl
for $S_3(\Gamma )$ for an appropriate choice of $\Q$-model
of the curve $X_{\Gamma}$.  For the $L$-function of the 
corresponding representations we have 
\[
L(s, \rho )  = L(s, f) L(s, f) \quad \mr{for\  1a,\  1b}, 
\]
\[
L(s, \rho )  = L(s, f) L(s, \overline{f}) \quad \mr{for \ 3a,\  3b, \ 4a,\  4b}. 
\] 
\end{thm} 
In an earlier version of this paper a complete proof for cases 1a and 1b was given. 
We do not reproduce it here as it is very similar to other published examples. 
The $L$-function for examples 2a, 2b exhibits new and interesting features and will be discussed in a future work.

\subsection{The examples.} 
\label{SS:examples}
All the noncongruence subgroups $\Gamma$
discussed in this paper are of index three inside a congruence subgroup
$G$ which itself is one of the index 12 genus 0 subgroups considered by Beauville. Each of these gives rise to a family if elliptic curves 
$E_{G}\to X_{G} =
(G\backslash \HH)^* \cong \mbf{P}^1 (\C) $ with 
ramification over the four cusps of $G$. For each of these, 
we select two of the cusps of $G$ to construct
a subgroup $\Gamma$ such that the corresponding covering 
\[
X_{\Gamma}\cong \mbf{P}^1 (\C)  \longrightarrow 
X_{G} \cong \mbf{P}^1 (\C) 
\]
branches only over the two chosen cusps. We describe these coverings
in the form $r^3 = m(t)$, where $r$ (resp. $t$) is a generator 
of the function field of $X_{\Gamma}$ (resp. $X_{G}$), i.e., a Hauptmodul,
which exists since these curves have genus 0. See table \ref{tab:coveringmaps}. 
We have also considered arithmetic twists of a given covering
gotten by varying some of the constants in the expression of $m(t)$. This 
leads to different models of Scholl's $l$-adic representation attached to
$S_3 (\Gamma)$, i.e., representations of $\mr{Gal} (\overline{\Q}/\Q)$
that become isomorphic as representations of  $\mr{Gal} (\overline{\Q}/K)$
for a finite extension $K/\Q$. It is an important point that, 
in contrast to the case of classical modular curves for congruence
subgroups, there are no {\it canonical} models defined over a number
field. Scholl's construction of his $l$-adic representations depends 
on a choice of a model. Moreover, this choice is subject to a number of 
hypotheses: generally that there should be a model defined over
$\Q$, and a cusp which is $\Q$-rational. This cusp is used for the 
expansions of modular forms whose coefficients satisfy ASwD congruences. 

The $l$-adic representations that Scholl constructs that are associated to 
$S_k (\Gamma)$ for noncongruence subgroups $\Gamma $ have very different 
properties from the corresponding representations constructed by Deligne
for congruence $\Gamma$. The main point is that in the congruence case, 
the Hecke algebra acts and commutes with the Galois action so that 
the $2d$-dimensional representation ($d = \dim S_k(\Gamma)$) splits 
into $2$-dimensional $\lambda$-adic representations. This is no longer the 
case in general for noncongruence subgroups. It is the case in our examples
that the $4$-dimensional representations attached to $S_3 (\Gamma)$ factor
into $2$-dimensional pieces. Geometrically this is due to the presence 
of extra symmetries given by involutions and/or isogenies of our 
elliptic surfaces.  

\subsection{Outline.}
\label{SS:outline}
In section \ref{S:desc} we define the congruence and noncongruence 
subgroups we will be working with. Section \ref{S:constructing} gives the method
we use to construct the noncongruence forms $h_1$, $h_2$. Section \ref{S:traces}
explains how we computed the traces of Frobenius elements in the 
$l$-adic Scholl's representation attached to our group $\Gamma$. The main point
is to count the number of rational points over $\mbf{F}_p$ and
$\mbf{F}_{p^2}$ of the elliptic modular surface $E_{\Gamma}$. 
In section \ref{S:II}. we discuss involutions and isogenies 
of these elliptic surfaces. 
Finally in section

\ref{S:aswddata} we provide the experimental evidence for the ASwD congruences.

\begin{table}
\begin{tabular}{|l|}
\hline
{\bf 1a.}
Basis of $S_3(\Gamma_{24.6.1^6})$:\\
$
h_1(z)=
\displaystyle{\sqrt[3]{\frac{\eta(z)^4\eta(4z)^{20}}{\eta(2z)^6}}}=
 q - \frac{4}{3}
q^2 + \frac{8}{9}q^3 - 
\frac{176}{81}q^4 + \cdots
$
\\
$
h_2(z)=
\displaystyle{
\sqrt[3]{\frac{\eta(4z)^{16}\eta(2z)^6}{\eta(z)^4}}}
=
 q + \frac{4}{3}q^2 + \frac{8}{9}q^3 + \frac{176}{81}q^4 + \cdots
$
\\
Associated newform in $S_3 (\Gamma _0 (48), \chi),$ where
$\chi(\Frob_p)=\left(\frac{-3}{p}\right)\left(\frac{-4}{p}\right)$:
\\
$
f(z) =\displaystyle{
\frac{\eta(4z)^9\eta(12z)^9}
{\eta(2z)^3\eta(6z)^3\eta(8z)^3\eta(24z)^3}}
=
q + 3q^3 - 2q^7 + 9q^9 
- 22q^{13} 
+\ldots
$
\\
The ASwD  basis is $h_1, h_2$.\\
\hline
\hline
{\bf 1b.}
Basis of $S_3(\Gamma_{8^3.2^3.3^3})$:
\\
$
h_1(z) =\displaystyle{
\sqrt[3]{
\frac{\eta(2\tau)^{20}\eta(8\tau)^{4}}{\eta(4\tau)^{6}}}}
=
q^{1/3} - \frac{20}{3}{q^{4/3}} + \frac{128}{9}{q^{7/3}} 
- \frac{400}{81}{q^{10/3}} + \cdots$
\\
$
h_2(z) =
\displaystyle{
\sqrt[3]{\frac{\eta(2\tau)^{16}\eta(4\tau)^{6}}{\eta(8\tau)^{4}}}}
=
q^{2/3} - \frac{16}{3}q^{14/3} + \frac{38}{9}q^{26/3} + \frac{1696}{81}q^{38/3}
 + \cdots
$\\
The associated newform is a twist $f\otimes \chi$ of the $f$ in case 1a.\\
The ASwD basis is $h_1, h_2$.\\
\hline
\end{tabular}
\caption{Modular forms for noncongruence subgroups, and
associated forms for congruence subgroups.}
\label{tab:summary1}
\end{table}


\begin{table}
\begin{tabular}{|l|}
\hline
{\bf 2a.}
Basis of $S_3(\Gamma_{8^3.6.3.1^3})$:
\\
$
h_1(z)
=
\displaystyle{
\sqrt[3]{\frac{\eta(z)^{4}\eta(2z)^{10}\eta(8z)^{8}}{\eta(4z)^4}}}
=
q - \frac 4 3q^2 - \frac{40}{9}q^3 + \frac{400}{81}q^4 
+ \frac{1454}{243}q^5 + \cdots
$
\\
$
h_2(z)
=
\displaystyle{
\sqrt[3]{\frac{\eta(z)^{8}\eta(4z)^{10}\eta(8z)^{4}}{\eta(2z)^4}}}
=
 q - \frac 8 3q^2 + \frac 8 9q^3 + \frac {32} {81}q^4 
- \frac{82}{243}q^5 + \ldots
$
\\
Newfor
m in $S_3 (\Gamma _0 (432), \chi)$, 
where
$\chi(\Frob_p)=\left(\frac{-4}{p}\right)$:\\
$
f(z)=
f_1(12z) + 6\sqrt{2} f_5(12z) + \sqrt{-3}f_7(12z) + 6\sqrt{-6}f_{11}(12z),
$
\\
where\\
$\begin{array}{lll}
f_1(z)= \frac{\eta(2z)^3\eta(3z)}{\eta(6z)\eta(z)}E_6(z)
&
f_5(z)=
\frac{\eta(z)\eta(2z)^3\eta(3z)^3}{\eta(6z)}&
\\
\rule{0ex}{5ex}
f_7(z)=\frac{\eta(6z)^3\eta(z)}{\eta(2z)\eta(3z)}E_6(z)
&
f_{11}(z)=
\frac{\eta(3z)\eta(z)^3\eta(6z)^3}{\eta(2z)}&
\end{array}$
\\
$\text{ and } E_6(z)=
1 + 12\sum_{n\ge 1}(\sigma(3n) -3\sigma(n))q^n,$
$\text{ where }\sigma(n)=\sum_{d|n} d.$
\\
Atkin Swinnerton-Dyer basis:
\\
$
\begin{array}{lll}
\text{if } p\equiv 1\mod 3   &\text{ basis is }& h_1, h_2\\
\text{if }p\equiv 2\mod 3  &\text{ basis is }& h_1\pm \alpha h_2, \ \ \alpha ^3  = 4.
\end{array}
$
\\
\hline
\hline
{\bf 2b.}
Basis of $S_3(\Gamma_{24.3.2^3.1^3})$:
\\
$
h_1(z) =\displaystyle{
\sqrt[3]{
\frac{\eta(2\tau)^{22}\eta(8\tau)^{8}}{\eta( \tau)^{4}\eta(4\tau)^{8}}}
=
q
 + \frac{4}{3} q^2
 - \frac{40}{9} q^3
 - \frac{400}{81} q^4
 + \frac{1454}{243} q^5
+ \cdots}
$
\\
$
h_2(z) =
\displaystyle{
\sqrt[3]{
\frac{\eta(2\tau)^{20}\eta(4\tau)^{2}\eta(8\tau)^{4}}{\eta( \tau)^{8}}}
=
q
 + \frac{8}{3} q^2
 + \frac{8}{9} q^3
 - \frac{32}{81} q^4
 - \frac{82}{243} q^5
+ \cdots}
$
\\
The associated new form and the ASwD basis\\
are given in exactly the same way as in case 2a.\\
\\
A variant denoted $S_3(\Gamma_{24.3.2^3.1^3 B})$  is discussed in section
\ref{SSS:836}
\\
\hline

\end{tabular}
\caption{Modular forms for noncongruence subgroups, and
associated forms for congruence subgroups.}
\label{tab:summary2}
\end{table}

\begin{table}
\begin{tabular}{|l|}
\hline
{\bf 3a.}
Basis of $S_3(\Gamma_{18.6.3^3.1^3})$
\\
$
h_1(z)
=
\displaystyle{
\sqrt[3]{\frac{\eta(z)^{4}\eta(2z)^7\eta(6z)^{11}}{\eta(3z)^4}}}
=
q - \frac 4 3q^2 - \frac{31}{9}q^3 + \frac{400}{81}q^4 
+ \frac{104}{243}q^5 + \cdots
$
\\
$
h_2(z)
=
\displaystyle{
\sqrt[3]{\frac{\eta(3z)^{4}\eta(6z)^7\eta(2z)^{11}}{\eta(z)^4}}}
=
 q + \frac 4 3q^2 - \frac 7 9q^3 - \frac {112} {81}q^4 
- \frac{616}{243}q^5 + \ldots
$
\\
Newform in $S_3 (\Gamma _0 (243), \chi)$, where
$\chi(\Frob_p)=\left(\frac{-3}{p}\right)$.
\\
$f(z) =
q + 3iq^2 - 5q^4 + 6iq^5 + 11q^7 - 3iq^8 - 18q^{10}+\cdots
$
\\
Atkin Swinnerton-Dyer basis:
\\
$
\begin{array}{lll}
\text{if }p\equiv 1\mod 3   &\text{ basis is }& h_1, h_2\\
\text{if }p\equiv 2\mod 3   &\text{ basis is }& h_1\pm i\sqrt[3]{3}h_2
\end{array}
$
\\
\hline 
\hline 
{\bf 3b.}
Basis of $S_3(\Gamma_{9.6^3.3.2^3})$; $r=q^{1/3}$.
\\
$
h_1(z)
=
\displaystyle{
\sqrt[3]{
\frac{
\eta( \tau)^{7}\eta(2\tau)^{4}\eta(3\tau)^{11}}{\eta(6\tau)^{4}}}
=
r
 - \frac{7}{3} r^4
 - \frac{19}{9} r^7
 + \frac{193}{81} r^{10}
 + \frac{2306}{243} r^{13}+ \cdots}
$
\\
$
h_2(z)
=
\displaystyle{
\sqrt[3]{
\frac{\eta( \tau)^{11}\eta(3\tau)^{7}\eta(6\tau)^{4}}{\eta(2\tau)^{4}}}
=
r^2
 - \frac{11}{3} r^5
 + \frac{23}{9} r^8
 - \frac{13}{81} r^{11}
+\cdots
}
$
\\
The associated new form and the ASwD basis are given in\\
exactly the same way as in case 3a.\\
\hline
\end{tabular}
\caption{Modular forms for noncongruence subgroups, and
associated forms for congruence subgroups.}
\label{tab:summary3}
\end{table}


\begin{table}
\begin{tabular}{|l|}
\hline
{\bf 4a.}
Basis of $S_3(\Gamma_{9.6^4.1^3})$\\
$
h_1(z)
=
\displaystyle{
\sqrt[3]{\frac{\eta(z)^{13}\eta(6z)^{14}}{\eta(2z)^{2}\eta(3z)^7}}}
=
q - \frac {13} 3q^2 + \frac{32}{9}q^3 + \frac{670}{81}q^4 
- \frac{3577}{243}q^5 + \cdots
$
\\
$
h_2(z)
=
\displaystyle{
\sqrt[3]{\frac{\eta(z)^{14}\eta(6z)^{13}}{\eta(2z)^{7}\eta(3z)^2}}}
=
 q - \frac {14} 3q^2 + \frac {56} 9q^3 - \frac {58} {81}q^4 
+ \frac{266}{243}q^5 + \ldots
$
\\
Associated newform in $S_3 (\Gamma _0 (486), \chi)$,
where $\chi(\Frob_p)=\left(\frac{-3}{p}\right)$.
\\
$
f(z)=q - \sqrt{-2}q^{2} - 2q^{4} + 3\sqrt{-2}q^{5} - 7q^{7} + 2\sqrt{-2}q^{8} +
 6q^{10} - 3\sqrt{-2}q^{11} + 5q^{13}
$
\\
Atkin Swinnerton-Dyer basis:
\\
$
\begin{array}{lll}
\text{if }p\equiv 1\mod 3   &\text{ basis is }& h_1, h_2\\
\text{if }p\equiv 2\mod 3   &\text{ basis is }& h_1\pm \sqrt{-2}\sqrt[3]{3}h_2
\end{array}
$
\\
\hline
\hline
{\bf 4b.}
Basis of $S_3(\Gamma_{18.3^4.2^3})$; $r=q^{1/3}$:
\\
$
h_1(z)
=
\displaystyle{
\sqrt[3]{
\frac{\eta(2\tau)^{13}\eta(3\tau)^{14}}{\eta(6\tau)^{7}\eta( \tau)^{2}}}
= r + \frac{2}{3}r^4 - \frac{28}{9}r^7
- \frac{482}{81}r^{10} - \frac{736}{243}r^{13}
+\cdots
}
$
\\
$
h_2(z)
=
\displaystyle{
\sqrt[3]{
\frac{\eta(2\tau)^{14}\eta(3\tau)^{13}}{\eta(6\tau)^{2}\eta( \tau)^{7}}} 
= r^2 + 
\frac{7}{3}r^5 + \frac{14}{9}r^8
- \frac{148}{81}r^{11}
- \frac{1708}{243}r^{14}
+\cdots}
$
\\
The associated newform is the same as in case 4a.
\\
Atkin Swinnerton-Dyer basis:
$
\begin{array}{lll}
\text{if }p\equiv 1\mod 3   &\text{ basis is }& h_1, h_2\\
\text{if }p\equiv 2\mod 3   &\text{ basis is }& h_1\pm \sqrt{-2}\sqrt[3]{3}h_2
\end{array}
$
\\
\hline
\end{tabular}
\caption{Modular forms for noncongruence subgroups, and
associated forms for congruence subgroups. }
\label{tab:summary4}
\end{table}

\clearpage

\section{Description of  the noncongruence subgroups}
\label{S:desc}

\subsection{Beauville's families}
\label{SS:beauv}
We start with certain index $12$ genus $0$ torsion free congruence
subgroups of $\SL_2(\Z)$, listed
 in Table~\ref{table:dataforbeauvillecurves} \cite{Seb01}.
Figure~\ref{fig:index12fundomains}
shows corresponding fundamental domains
and generating matrices.

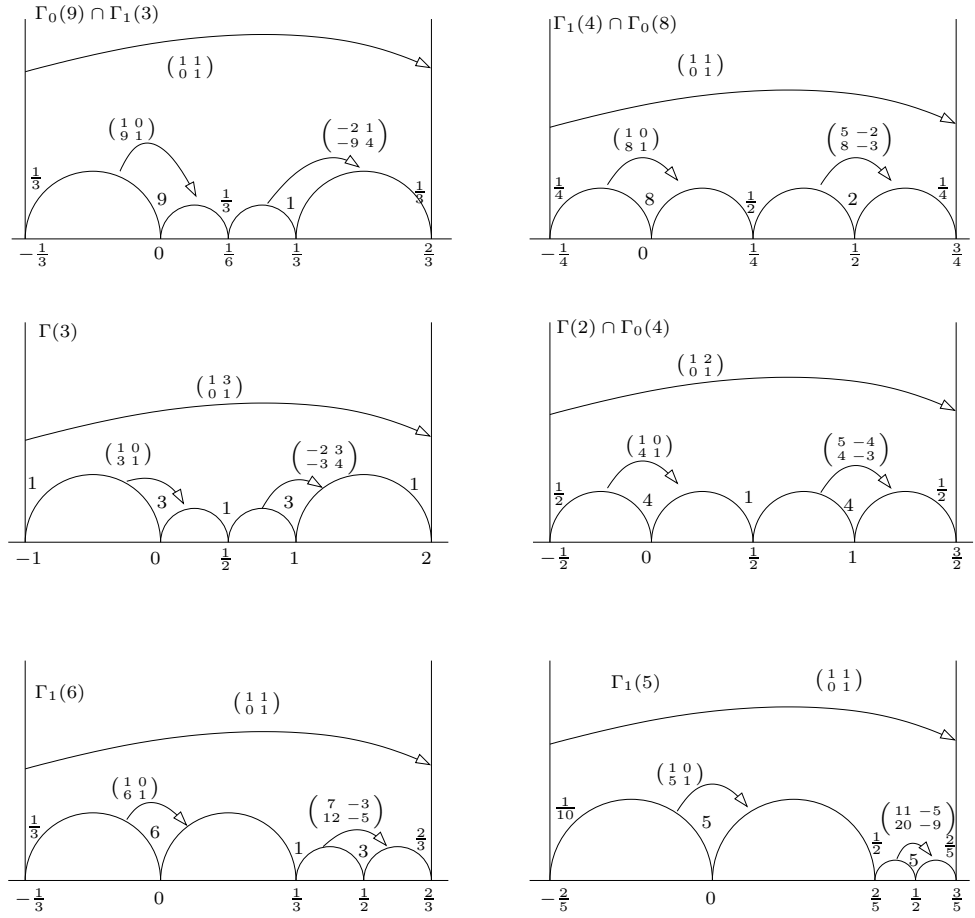
\begin{figure}
\label{fig:index12fundomains}
\input{beauvillecurves.pstex_t}
\caption{Fundamental domains for torsion free
index $24$ congruence subgroups
in $\SL_2(\Z)$. }
\end{figure}

Table~\ref{table:dataforbeauvillecurves}
gives equations
for the associated families of elliptic curves
\cite{Beau82}.
Table~\ref{tab:weierstrassforbeauville}
gives the $a_1,\dots,a_5$
of the Weierstrass form
$y^2 + a_1xy + a_3y = x^3 + a_2x^2 + a_4x + a_6.$
The hauptmodul $t(\tau)$ listed
in the table is such that $j(E_{t(\tau)})=j(\tau)$.

\begin{table}
$
\renewcommand\arraystretch{1.5}
\begin{array}{ccc}
\text{group}&
\text{elliptic family} & 
j-\text{ invariant} \\
\hline
\Gamma(3) &    (x^3+y^3+z^3)=txyz        
&
\frac{t^3(t^3+216)^3}{(t^3-27)^3}
\\
\Gamma(2)\cap\Gamma_1(4) &  x(x^2+z^2+2zy)=tz(x^2-y^2)  
&
\frac{(t^4 - t^2 + 1)^3}{ t^4 (t - 1)^2 (t + 1)^2}
\\
\Gamma^1(5) &  x(x-z)(y-z)t=y(y-x)z    
&
-\frac{(t^4 + 12t^3 + 14t^2 - 12t + 1)^3}
{ t^5 (t^2 + 11t - 1)}
\\
\Gamma_1(6) & (xy + yx + zx)(x+y+z)=txyz 
&
\frac{(3t - 1)^3 (3t^3 - 3t^2 + 9t - 1)^3}{ (t - 1)^3 t^6 (9t - 1)}
\\
\Gamma_0(8)\cap\Gamma_1(4) &  (x+y)(xy + z^2)t=4xyz   
&
-16\frac{(t^4 -16 t^2 + 16)^3}{ t^8(t+1)(t-1)}
\\
\Gamma_0(9)\cap\Gamma_1(3) &  (x^2y + y^2z+z^2x)=txyz 
&
\frac{t^3(t^3-24)^3}{t^3 - 27}
\\
\hline
\end{array}
$
\caption{Data for Beauville's elliptic surfaces.}
\label{table:dataforbeauvillecurves}
\end{table}

\begin{table}
$
\renewcommand\arraystretch{1.5}
\begin{array}{|c|ccccc|c|}
\hline
\text{level}&
 \multicolumn{5}{c|}{\text{Coefficients of Weierstrass form}} &
t \text{ as a }\\
 & a_1 & a_2 & a_3 & a_4 & a_6&\text{Hauptmodul}\\
\hline
3
&
0
&
t^2
&
0
&
-72t
&
-8(4t^2 + 27)
&
\frac{\eta{\left(\frac{1}{3}\tau\right)^{3}}}{\eta(3\tau)^3}
+3
\rule{0ex}{4ex}
\\
4
&
0
&4 + 4t^2
&0
&16t^2
&0
&
\frac{1}{2}\frac{\eta(\tau)^{12}} 
{\eta(2\tau)^{8}\eta\left(\frac{1}{2}
\tau\right)^{4}}
\rule{0ex}{4ex}
\\
5
&t+1 & t & t &0&0
&
q^{\frac 1 5}
{\displaystyle{\prod_{\stackrel{n=0}{e=1,-1}}^\infty}} \left(
\frac{\left(1-q^{n+e\frac 1 5}\right)}
{\left(1-q^{n+e\frac 2 5}\right)}\right)^5
\rule{0ex}{4ex}
\\
6
&t+1  
&t-t^2
&t-t^2
&0    
&0
&
\frac{1}{9}
\frac{\eta(6\tau)^{4}\eta(\tau)^{8}}  
{\eta(3\tau)^{8}\eta(2\tau)^{4}}
\rule{0ex}{4ex}
\\
8
&4
&t^2
&4t^2
&0
&0
&
\frac{\eta(z)^8\eta(4z)^4}{\eta(2z)^{12}}
\\
9
&0
&t^2
&0
&8t
&16
&
\rule{0ex}{4ex}
27\frac{\eta(9\tau)^3}{\eta(\tau)^{3}} + 3
\\
\hline
\end{array}
$
\caption{Weierstrass equations for Beauville's elliptic families.}
\label{tab:weierstrassforbeauville}
\end{table}

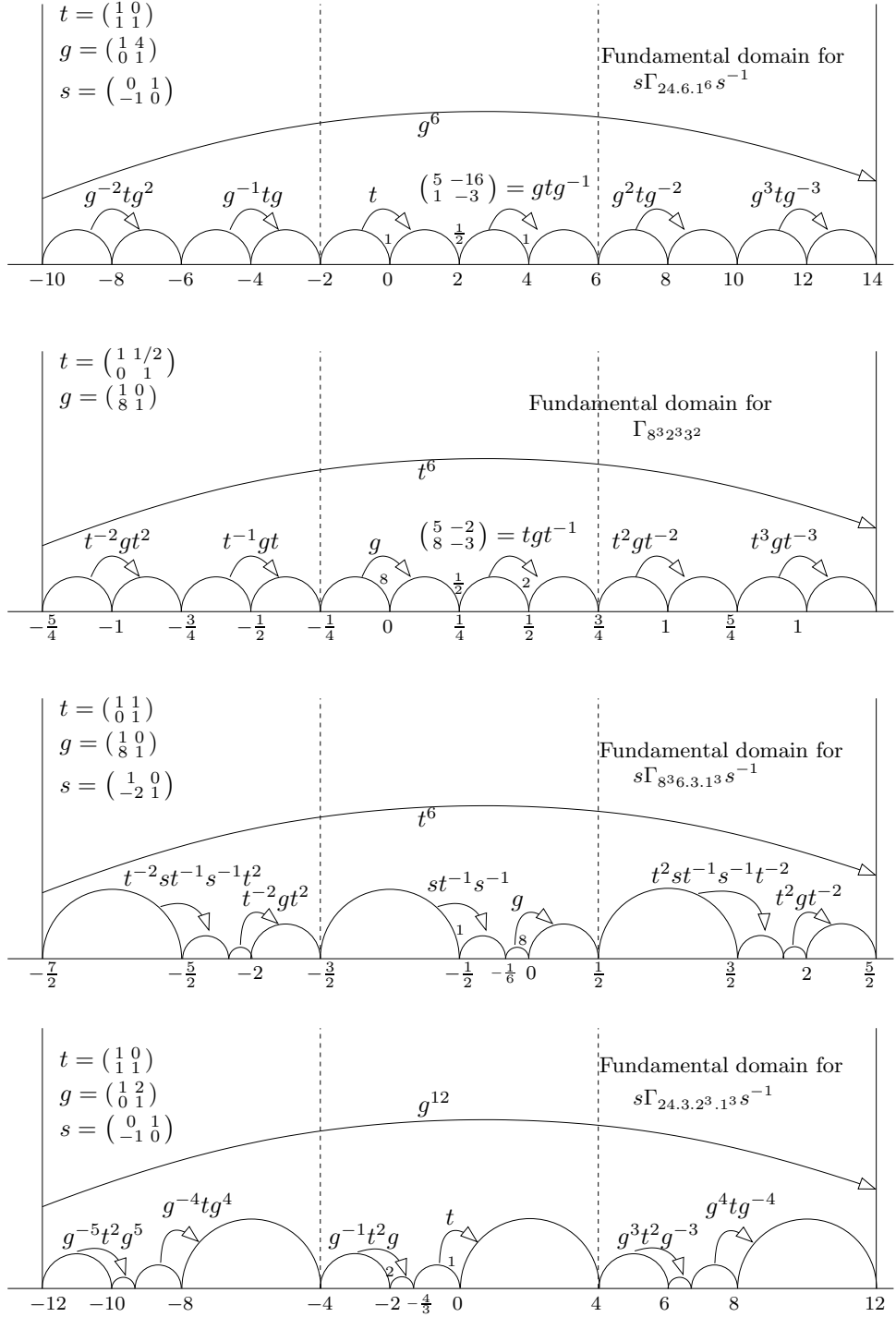
\begin{figure}
$$\input{level8subgroups.pstex_t}$$
\caption{Fundamental domains for conjugates of some
index $3$ subgroups of $\Gamma_0(8)\cap\Gamma_1(4)$.}
\label{level8subgroups}
\end{figure}

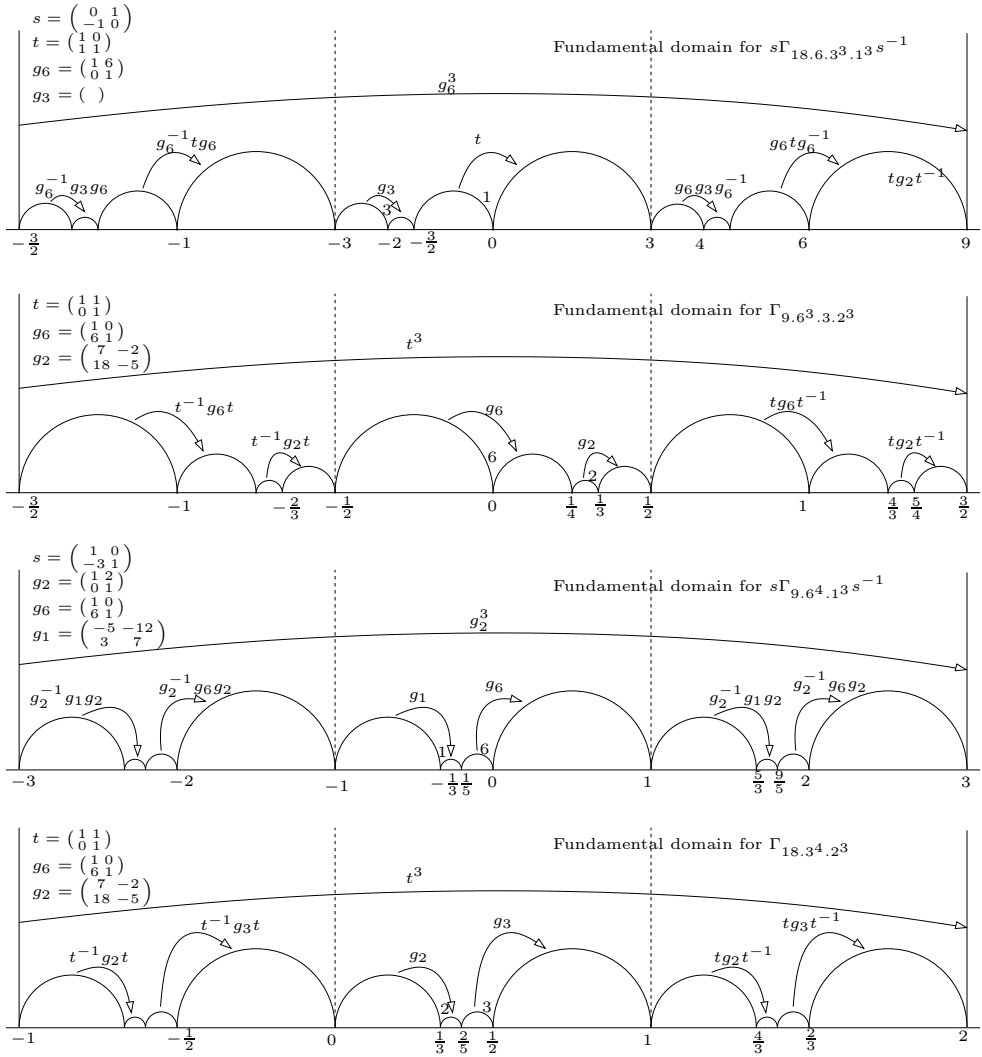
\begin{figure}
$$\input{level6subgps.pstex_t}$$
\caption{Fundamental domains for conjugates of some
index $3$ subgroups of $\Gamma_1(6)$.}
\label{level6subgroups}
\end{figure}

\clearpage

\subsection{The noncongruence subgroups}
\label{SS:subgrp}

We will work with certain index $3$ normal subgroups of 
$\Gamma_1(6)$ and $\Gamma_0(8)\cap\Gamma_1(4)$.
The case $\Gamma_1(5)$ has been studied in \cite{LLY}.
The fundamental domain of $\Gamma$
is a union of three copies of a fundamental domain 
for $G$, corresponding to the three cosets of $\Gamma$ in $G$.
From the fundamental domains, shown in
Figures~\ref{level8subgroups} and \ref{level6subgroups},
we obtain generators
and cusp widths \cite{kulkarni},
allowing us to make the following definition.

\begin{table}
$$
\begin{array}{|l|cccc|}
\hline
\multicolumn{5}{|c|}{\text{cusps and subgroups of $\Gamma_0(8)
\cap\Gamma_1(4)$}}\\
\hline
\raisebox{-1ex}{\rule{0ex}{3.5ex}}
\text{cusp }\tau \hspace{2cm}
& \infty & 0 & \frac{1}{2} & \frac{1}{4}\\
\hline
\text{width}
\rule{0ex}{3ex} & 1 & 8 & 2 & 1\\
\hline
\hline
\text{subgroup} &
\multicolumn{4}{|c|}{\text{ramified cusps}}\\
&\multicolumn{4}{|c|}{\text{indicated by \checkmark}}\\
\hline
\Gamma_{24.6.1^6} & &\checkmark&\checkmark&\\
\Gamma_{8^3.2^3.3^2}   &\checkmark &&&\checkmark \\
\Gamma_{8^3.6.3.1^3}    & & &\checkmark &\checkmark \\
\Gamma_{24.3.2^3.1^3}  &&\checkmark & &\checkmark 
\raisebox{-1ex}{\rule{0ex}{3.5ex}}
\\
\hline
\end{array}
\hspace{1ex}
\begin{array}{|l|cccc|}
\hline
\multicolumn{5}{|c|}{\text{cusps and subgroups of $\Gamma_1(6)$}}\\
\hline
\raisebox{-1ex}{\rule{0ex}{3.5ex}}
\text{cusp }\tau & \infty & 0 & \frac{1}{2} & \frac{1}{3}\\
\hline
\text{width}
\rule{0ex}{3ex} & 1 & 6 & 3 & 2\\
\hline
\hline
\text{subgroup} &
\multicolumn{4}{|c|}{\text{ramified cusps}}\\
&\multicolumn{4}{|c|}{\text{indicated by \checkmark}}\\
\hline
\Gamma_{18.6.3^3.1^3} & &\checkmark&&\checkmark\\
\Gamma_{9.6^3.3.2^3}  &\checkmark & &\checkmark &\\
\Gamma_{9.6^4.1^3}    & & &\checkmark &\checkmark \\
\Gamma_{18.3^4.2^3}   &\checkmark &\checkmark & &
\raisebox{-1ex}{\rule{0ex}{3.5ex}}\\
\hline
\end{array}
$$
\caption{Ramification points of triple covers of 
$X(\Gamma_0(8)\cap\Gamma_1(4))$ 
and $X(\Gamma_1(6))$, with corresponding subgroups.
}
\label{table:choices_of_ramification}
\end{table}

\begin{Def}
\label{def:8groups}
We let $\Gamma_{24.6.1^6}$,
$\Gamma_{8^36.3.1^3}$,
$\Gamma_{24.3.2^3.1^3}$,
$\Gamma_{8^32^33^2}$
be index $3$ genus $0$ subgroups of $\Gamma_0(8)\cap\Gamma_1(4)$,
and 
$\Gamma_{18.6.3^3.1^3},
\Gamma_{9.6^4.1^3},
\Gamma_{9.6^3.3.2^3}$,
$\Gamma_{18.3^4.2^3}$
index $3$ genus $0$ subgroups of $\Gamma_1(6)$,
defined by their generators as follows:
$$
\renewcommand\arraystretch{1.5}
\begin{array}{ll}
\Gamma & \multicolumn{1}{c}{\text{generators}}\\
\hline
\Gamma_{24.6.1^6}&
\left(\begin{smallmatrix}1 & 0\\ 24 & 1\end{smallmatrix}\right),
\left(\begin{smallmatrix}9  & -1\\ 64 & -7\end{smallmatrix}\right),
\left(\begin{smallmatrix}5  & -1\\ 16 & -3\end{smallmatrix}\right),
\left(\begin{smallmatrix}1  &  1\\ 0  &  1\end{smallmatrix}\right),
\left(\begin{smallmatrix}-3 & -1\\ 16 &  5\end{smallmatrix}\right),
\left(\begin{smallmatrix}-7 & -1\\ 64 &  9\end{smallmatrix}\right),
\left(\begin{smallmatrix}-11& -1\\ 144& 13\end{smallmatrix}\right).
\\
\Gamma_{8^32^33^2}&
\left(\begin{smallmatrix}1& 3\\ 0& 1     \end{smallmatrix}\right),
\left(\begin{smallmatrix}-7& -8\\ 8& 9   \end{smallmatrix}\right),
\left(\begin{smallmatrix}-3& -2\\ 8& 5   \end{smallmatrix}\right),
\left(\begin{smallmatrix}1& 0\\ 8& 1     \end{smallmatrix}\right),
\left(\begin{smallmatrix}5& -2\\ 8& -3   \end{smallmatrix}\right),
\left(\begin{smallmatrix}9& -8\\ 8& -7   \end{smallmatrix}\right),
\left(\begin{smallmatrix}13& -18\\ 8& -11\end{smallmatrix}\right).
\\
\Gamma_{8^36.3.1^3}&
\left(\begin{smallmatrix}-11& 6 \\ -24& 13 \end{smallmatrix}\right),
\left(\begin{smallmatrix}41& -25\\ 64& -39 \end{smallmatrix}\right),
\left(\begin{smallmatrix}49& -32\\ 72& -47 \end{smallmatrix}\right),
\left(\begin{smallmatrix}1& 1  \\ 0& 1    \end{smallmatrix}\right),
\left(\begin{smallmatrix}1& 0   \\ 8& 1    \end{smallmatrix}\right),
\left(\begin{smallmatrix}25& -9 \\ 64& -23 \end{smallmatrix}\right),
\left(\begin{smallmatrix}81& -32\\ 200& -79\end{smallmatrix}\right).
\\
\Gamma_{24.3.2^3.1^3}&
\left(\begin{smallmatrix}1& 0  \\ 24& 1   \end{smallmatrix}\right),
\left(\begin{smallmatrix}21& -2\\ 200& -19 \end{smallmatrix}\right),
\left(\begin{smallmatrix}9& -1 \\ 64& -7   \end{smallmatrix}\right),
\left(\begin{smallmatrix}5& -2 \\ 8& -3    \end{smallmatrix}\right),
\left(\begin{smallmatrix}1& 1 \\ 0& 1     \end{smallmatrix}\right),
\left(\begin{smallmatrix}-11&-2\\ 72& 13   \end{smallmatrix}\right),
\left(\begin{smallmatrix}-7& -1\\ 64& 9    \end{smallmatrix}\right).
\\
\hline
\Gamma_{18.6.3^3.1^3}&
\left(\begin{smallmatrix}1& 0\\ 18& 1    \end{smallmatrix}\right),
\left(\begin{smallmatrix}25& -3\\ 192& -23\end{smallmatrix}\right),
\left(\begin{smallmatrix}7& -1\\ 36& -5   \end{smallmatrix}\right),
\left(\begin{smallmatrix}7& -3\\ 12& -5   \end{smallmatrix}\right),
\left(\begin{smallmatrix}1& 1\\ 0& 1     \end{smallmatrix}\right),
\left(\begin{smallmatrix}-11& -3\\ 48& 13 \end{smallmatrix}\right),
\left(\begin{smallmatrix}-5& -1\\ 36& 7   \end{smallmatrix}\right).
\\
\Gamma_{9.6^3.3.2^3}&
\left(\begin{smallmatrix}1& 3\\ 0& 1     \end{smallmatrix}\right),
\left(\begin{smallmatrix}-5& -6\\ 6& 7   \end{smallmatrix}\right),
\left(\begin{smallmatrix}-11& -8\\ 18& 13\end{smallmatrix}\right),
\left(\begin{smallmatrix}1& 0\\ 6& 1     \end{smallmatrix}\right),
\left(\begin{smallmatrix}7& -2\\ 18& -5  \end{smallmatrix}\right),
\left(\begin{smallmatrix}7& -6\\ 6& -5   \end{smallmatrix}\right),
\left(\begin{smallmatrix}25& -32\\ 18& -23\end{smallmatrix}\right).
\\
\Gamma_{9.6^4.1^3}&
\left(\begin{smallmatrix}-17& 6  \\ -54& 19      \end{smallmatrix}\right),
\left(\begin{smallmatrix}127& -49\\ 324& -125\end{smallmatrix}\right),
\left(\begin{smallmatrix}61 & -24\\ 150& -59   \end{smallmatrix}\right),
\left(\begin{smallmatrix}1  & 1  \\ 0  & 1         \end{smallmatrix}\right),
\left(\begin{smallmatrix}1  & 0  \\ 6  & 1         \end{smallmatrix}\right),
\left(\begin{smallmatrix}91 & -25\\ 324& -89  \end{smallmatrix}\right),
\left(\begin{smallmatrix}85 & -24\\ 294& -83  \end{smallmatrix}\right).
\\
\Gamma_{18.3^4.2^3}&
\left(\begin{smallmatrix}1& 3\\ 0& 1     \end{smallmatrix}\right),
\left(\begin{smallmatrix}-11& -8\\ 18& 13\end{smallmatrix}\right),
\left(\begin{smallmatrix}-5& -3\\ 12& 7  \end{smallmatrix}\right),
\left(\begin{smallmatrix}7& -2\\ 18& -5  \end{smallmatrix}\right),
\left(\begin{smallmatrix}7& -3\\ 12& -5  \end{smallmatrix}\right),
\left(\begin{smallmatrix}25& -32\\ 18& -23\end{smallmatrix}\right),
\left(\begin{smallmatrix}19& -27\\ 12& -17\end{smallmatrix}\right).\\
\hline
\end{array}
$$
\label{def:gamma}
\end{Def}

By comparing cusp widths, 
in Tables~\ref{tab:cuspvaluesG8} and
\ref{tab:cuspvaluesforlevel6}, with possible cusp
widths of congruence subgroups
in Table~\ref{table:sebbardata},
we obtain the following result.

\begin{thm}
The groups in Definition~\ref{def:8groups}
 are noncongruence subgroups.
\end{thm}

\begin{table}
$$
\begin{array}{ll}
& 6-6-6-6-3-3-3-3\\
&9-9-9-3-3-1-1-1\\
&9-9-3-3-3-3-3-3\\
&10-10-5-5-2-2-1-1\\
&18-9-2-2-2-1-1-1\\
&27-3-1-1-1-1-1-1\\
\end{array}
$$
\caption{Possible cusp widths of index $36$ genus zero torsion free
subgroups of $\PSL_2(\Z)$, taken from \cite[\S7, Table 2]{Seb01}.}
\label{table:sebbardata}
\end{table}

\subsection{Hauptmoduln and covering maps}

Throughout this paper we fix our choice of identification of
$X(\Gamma_0(8)\cap\Gamma_1(4))$ and $X(\Gamma_1(6))$ with the projective line
$\PP^1$, with parameter $t_8$ and $t_6$ respectively.  
As functions of $z$ in the upperhalf complex plane, $t_8(z)$ and
$t_6(z)$ are given terms of the Dedekind eta function, as listed in the last column of
Table~\ref{tab:weierstrassforbeauville}:
$$t_8(z)=\frac{\eta(z)^8\eta(4z)^4}{\eta(2z)^{12}},\ \text{ and }\
t_6(z)=\frac{1}{9}\frac{\eta(6\tau)^{4}\eta(\tau)^{8}}{\eta(3\tau)^{8}\eta(2\tau)^{4}}.$$
The values of these functions at the cusps are as in 
Table~\ref{tab:valuesatcuspsforhauptmodul}. 
\begin{table}
$$
\begin{array}{|l|cccc|}
\hline
\multicolumn{5}{|c|}{\text{Values of } t_8}\\
\hline
\text{cusp } c & \infty & 0 & \frac{1}{2} &\frac{1}{4}\\
\hline
t_8(c)  & 1 & 0 & \infty & -1\\
\hline
\end{array}
\hspace{1cm}
\begin{array}{|l|cccc|}
\hline
\multicolumn{5}{|c|}{\text{Values of } t_6}\\
\hline
\text{cusp } c & \infty & 0 & \frac{1}{2} &\frac{1}{3}\\
\hline
t_8(c)  & \frac19 & 0 & 1 & \infty\\
\hline
\end{array}
$$
\caption{Values of Hauptmoduln at cusps.}
\label{tab:valuesatcuspsforhauptmodul}
\end{table}

 Since the ramification points of
the covering maps $\Gamma\setminus \HH\rightarrow G\setminus \HH$  are at cusps as  in Table~\ref{table:choices_of_ramification}, 
the covering maps are given in each case by a map
$$r\mapsto r^3=m(t),$$ where the maps $m$ corresponding to each of our
subgroups are as in Table~\ref{tab:coveringmaps}.  

\begin{table}[h]
\renewcommand\arraystretch{1.5}
$$
\begin{array}{|lll|}
\hline
\text{subgroup } & m(t) & m^{-1}(r^3)\\
\hline
\Gamma_{24.6.1^6}     &t&r^3\\
\Gamma_{8^3.2^3.3^2}  &\frac{1+t}{1-t}&\frac{r^3-1}{r^3+1}\\
 \Gamma_{8^3.6.3.1^3}  &\frac{t+1}{4}&4r^3-1\\
\Gamma_{24.3.2^3.1^3} &\frac{2(1+t)}{t}&\frac{2}{r^3-2}\\
\hline
\end{array}
\hspace{1cm}
\begin{array}{|lll|}
\hline
\text{subgroup } & m(t) & m^{-1}(r^3)\\
\hline
\Gamma_{18.6.3^3.1^3} &t/9&9r^3\\
 \Gamma_{9.6^3.3.2^3}  &\frac{1-9t}{3-3t}&\frac{1-3r^3}{9-3r^3}\\
\Gamma_{9.6^4.1^3}   &\frac{8}{3-3t}&1-\frac{8}{3r^3}\\
\Gamma_{18.3^4.2^3}   &\frac{1-9t}{24t}&\frac{1}{24r^3+9}\\
\hline
\end{array}       
$$
\caption{Covering maps corresponding to subgroups of $\Gamma_0(8)\cap\Gamma_1(4)$
and $\Gamma_1(6)$.}
\label{tab:coveringmaps}
\end{table}

\section{Constructing elements of $S_{3}(\Gamma)$}
\label{S:constructing}
\subsection{Dimension}
\label{SS:dim}
For odd $k$,
Shimura \cite[Theorem~2.25]{shimu} gives the following formula 
for $\dim S_k(\Gamma)$ for
a genus $g$ subgroup $\Gamma\notin -I$ of $\SL_2(\Z)$:
\[
\dim S_k (\Gamma) =  (k-1)(g-1) + \frac{1}{2}(k-2)u + \frac{1}{2}(k-1)u'
+\sum_{i=1}^rk\frac{e_i-1}{2e_i}.
\]
The $e_i$ are orders of elliptic points,
$u$ is the number of regular cusps, and $u'$ the number of
irregular cusps.  
Using this formula, we find that
\[
\dim S_3 (\Gamma) =  2,
\]
for $\Gamma$ equal to any of the groups in Definition~\ref{def:8groups}.

\subsection{Method of constructing elements of $S_{3}(\Gamma)$}
\label{SS:meth}
Suppose that $\Gamma$ has index $3$ in
 $G$, one of the groups in
Table~\ref{table:dataforbeauvillecurves}, 
and that the corresponding
covering is ramified at cusps $c_1$ and $c_2$.
Let $t$ be a Hauptmodul for $G$, e.g., as in \cite{CN}.
By a transformation, take $t$ with
$t(c_1)=0$ and $t(c_2)=\infty$.  Then $\sqrt[3]{t}$ is a
Hauptmodul for $\Gamma$.
Let
 $f\in M_3(G)$.  Then 
$\sqrt[3]{t}f\in A_3(\Gamma)$.  If $f$ is zero 
where $t$ has poles, then
$\sqrt[3]{t}f$ and $\sqrt[3]{t^2}f$ are in $S _3(\Gamma)$.
We give modular
forms in terms of the Dedekind eta function, using the data given
by Martin \cite{Martin}.
Explicit details of the forms and their poles and zeros
are given in Tables~\ref{tab:cuspvaluesG8} and \ref{tab:cuspvaluesforlevel6},
and the $q$-expansions are given in
Tables~\ref{tab:qexpansionslevel8} and \ref{tab:level6q-expansions}.
\clearpage

\begin{table}
$
\begin{array}{|l|c|c|c|c|c|}
\hline
 \multicolumn{2}{|r|}{\text{cusps (and widths)} }
&\multicolumn{1}{|c|}{\frac12 (2)} & 
\multicolumn{1}{|c|}{0 (8)} & \multicolumn{1}{|c|}{\infty (1)} & \multicolumn{1}{|c|}
{\frac14 (1)}
 \\
\cline{2-6}
\text{forms for } \Gamma_0(8)\cap\Gamma_1(4)
&\text{weight}
&\multicolumn{4}{|c|}{\parbox{1.2in}{order of vanishing}}\\
\cline{1-6}
\hline
t = \frac{\eta(z)^8\eta(4z)^4}{\eta(2z)^{12}}
=1 - 8q + 32q^2+\cdots&
0           & -1& 1 & \multicolumn{1}{|c|}{0} & \multicolumn{1}{|c|}{0} 
\\
\frac{t+1}{2}
=\frac{\eta(z)^4\eta(4z)^{14}}
{\eta(8z)^4\eta(2z)^{14}}=1 - 4q + 16q^2+\cdots&0
           & -1& 0 & \multicolumn{1}{|c|}{0} & \multicolumn{1}{|c|}{1} 
\\
\frac{t+1}{2t}=
\frac{\eta(4z)^{10}}
{\eta(8z)^4\eta(2z)^2\eta(z)^4}=1 + 4q + 16q^2+\cdots&0
          & 0& -1 & \multicolumn{1}{|c|}{0} & \multicolumn{1}{|c|}{1} 
\\
\frac{4(t+1)}{(1-t)}
=\frac{\eta(4z)^{12}}{\eta(8z)^8\eta(2z)^4}
   &0        & 0& 0 & \multicolumn{1}{|c|}{-1} & \multicolumn{1}{|c|}{1} 
\\
\hline
E_a 
=\frac{\eta(4z)^4\eta(2z)^6}{\eta(z)^4}
  & 3    & 1 & 0 & \multicolumn{1}{|c|}{1} & \multicolumn{1}{|c|}{1} 
\\
\multicolumn{1}{|l|}{E_b=\left(\frac{2t}{t+1}\right)E_a   
=\frac{\eta(2z)^8\eta(8z)^4}{\eta(4z)^{6}}}
   &3  & 1 & 1 & \multicolumn{1}{|c|}{1} & \multicolumn{1}{|c|}{0} 
\\ 
\hline
\end{array}
$
$
\begin{array}{|l|c|ccc|ccc|ccc|ccc|c|}
\hline
& \text{cusps}&
&1/2 &&& 0 && -\frac{1}{8}& \infty  & \frac{1}{8} & 
\frac{-1}{4} &\frac{1}{4} &  \frac{1}{12} \\
\text{forms for} &\text{width}& &6 &&& 24 && 1 & 1 & 1 &1 & 1 & 1 \\
\cline{2-14}
\Gamma_{24.6.1^6} &\text{weight} & 
\multicolumn{12}{|c|}{\text{order of vanishing of form at cusps}} \\
\cline{1-14}
\sqrt[3]{t}
      & 0           && -1&&&  1  && 0 & 0 & 0 & 0 & 0 & 0 \\
E_a   &3            && 3 &&&  0  && 1 & 1 & 1 & 1 & 1 & 1 \\
t^{1/3}E_a   &3     && 2 &&&  1  && 1 & 1 & 1 & 1 & 1 & 1 \\
t^{2/3}E_a   &3     && 1 &&&  2  && 1 & 1 & 1 & 1 & 1 & 1 \\
\hline
\hline
 & \text{cusps}&
&1/2 &&\frac{2}{5}& 0 &\frac{2}{3} &\frac{3}{8}
 & \infty      & \frac{5}{8} &  & 1/4 &  \\
\text{forms for}&\text{width}
&& 6 && 8 & 8 & 8 & 1 &1 & 1 & &3 &\\
\cline{2-14}
\Gamma_{8^36.3.1^3} &\text{weight} & 
\multicolumn{12}{|c|}{\text{order of vanishing of form at cusps}} \\
\cline{1-14}
r_1=\sqrt[3]{\frac{t+1}{2}}
      & 0           && -1 && 0 & 0 & 0  & 0 & 0 & 0 &  & 1 &  \\
E_b   &3            
&& 3 && 1 & 1 & 1 & 1 & 1 & 1 & &0& \\
r_1E_b   &3     && 2 &&   1  & 1 & 1 & 1 & 1 & 1 && 1& \\
r_1^2E_b   &3     && 1  &&  1  & 1 & 1  & 1 & 1 & 1 && 2& \\
\hline
\hline
 & \text{cusps}&
-\frac{1}{6}&\frac{1}{2} &\frac{1}{10}& &0 &
&-\frac{1}{8}& \infty &\frac{1}{8}  &  & 1/4 &\\
\text{forms for}&\text{width}&2& 2 &2&  & 24 &  & 1 &1 & 1 & &3 &\\
\cline{2-14}
\Gamma_{24.3.2^3.1^3} &\text{weight} & 
\multicolumn{12}{|c|}{\text{order of vanishing of form at cusps}} \\
\cline{1-14}
r_2=\sqrt[3]{\frac{(t+1)}{2t}}
      & 0           &0& 0 &0&  & -1 &   & 0 & 0 & 0 &  & 1 &  \\
E_b   &3            
&1& 1 &1 &  & 3 &  &1& 1 & 1 & &0& \\
r_2E_b
&3     &1& 1 &1&     & 2 &  & 1 & 1 & 1 && 1& \\
r_2^2E_b
&3     &1& 1  &1&    & 1 &   & 1 & 1 & 1 && 2& \\
\hline
\hline
 & \text{cusps}&
-\frac12&\frac12 &\frac{3}{2}& -1 &   0 &1& 
& \infty &  & &\frac{1}{4}& \\
\text{forms for}&\text{width}&2& 2 &2& 8 & 8 & 8 &  &3 &  & &3 &\\
\cline{2-14}
\Gamma_{8^32^33^2} &\text{weight} & 
\multicolumn{12}{|c|}{\text{order of vanishing of form at cusps}} \\
\cline{1-14}
r_3=\sqrt[3]{\frac{4(t+1)}{(t-1)}}
             & 0   &0& 0 &0& 0 & 0 & 0 &  & -1&  & & 1 &  \\
E_b        &3    &1& 1 &1&1  & 1 & 1 &  & 3 &  & & 0 & \\
r_3E_b  &3    &1& 1 &1& 1 & 1 & 1 &  & 2 &  & & 1 & \\
r_3^2E_b &3    &1& 1 &1& 1 & 1 & 1 &  & 1 &  & & 2 & \\
\hline
\end{array}
$
\caption{Orders of vanishingat cusps for forms for 
$\Gamma_0(8)\cap\Gamma_1(4)$ and for 
subgroups of 
$\Gamma_0(8)\cap\Gamma_1(4)$.  
}
\label{tab:cuspvaluesG8}
\end{table}

\clearpage
\begin{table}[p]
$
\begin{array}{|l|c|c|c|c|c|}
\hline
\multicolumn{2}{|r|}{\text{cusps (and widths)}} & \infty (1)& 0 (6) & \frac12 (3) 
& \frac13 (2)\\
\cline{2-6}
\text{forms for } \Gamma_1(6)
&\text{weight}&\multicolumn{4}{|c|}{\parbox{1.2in}{order of vanishing}}\\
\hline
\multicolumn{1}{|l|}{a=\frac{\eta(z)\eta(6z)^6}{\eta(2z)^2\eta(3z)^3}=q - q^2 + q^3 + q^4 + \cdots }&1 & 1 & 0 & 0 & 0\\
\multicolumn{1}{|l|}{b=\frac{\eta(2z)\eta(3z)^6}{\eta(z)^2\eta(6z)^3}=1 + 2q + 4q^2 + 2q^3 + \cdots}&1 & 0 & 0 & 0 & 1\\
\multicolumn{1}{|l|}{c=\frac{\eta(3z)\eta(2z)^6}{\eta(6z)^2\eta(z)^3}=1 + 3q + 3q^2 + 3q^3 + \cdots}&1 & 0 & 0 & 1 & 0\\
\multicolumn{1}{|l|}{d=\frac{\eta(6z)\eta(z)^6}{\eta(3z)^2\eta(2z)^3}=1 - 6q + 12q^2 - 6q^3\cdots}&1 & 0 & 1 & 0 & 0\\
\hline
\multicolumn{1}{|l|}{r_0 = b/d= 1 + 8q + 40q^2 + 152q^3 + \cdots }&0& 0  & -1 & 0 & 1\\
\multicolumn{1}{|l|}{r_1 = b/c=8\frac{r_0}{(9r_0-1)}   = 1 - q + 4q^2 + \cdots   }&0& 0  &  0 &-1 & 1\\
\multicolumn{1}{|l|}{r_2 = a/c=\frac{(r_0-1)}{(9r_0-1)}= q - 4q^2 + 10q^3 \cdots         }&0& 1  &  0 &-1 & 0\\
\multicolumn{1}{|l|}{r_3 = a/d=\frac{1}{8}(r_0-1)       = q + 5q^2 + 19q^3 \cdots         }&0& 1  & -1 & 0 & 0\\
\hline
\multicolumn{1}{|l|}{acd=q - 4q^2 + q^3 + 16q^4 +\cdots         }&3&
 1 & 1 & 1 & 0\\
\multicolumn{1}{|l|}{bcd= 1 - q - 5q^2 - q^3 + 11q^4 +\cdots   }&3&
 0 & 1 & 1 & 1\\
\hline
\end{array}
$
$
\begin{array}{|l|c|ccc|ccc|ccc|ccc|c|}
\hline
& \text{cusps}& \frac{1}{6} &\infty  & -\frac{1}{6} 
&& 0 &&  \frac{1}{8} & \frac{1}{2}   &
 -\frac{1}{4} 
 &  &\frac{1}{3}& \\
\text{forms for} &\text{width}& 1&1 &1&& 18 &&  3& 3 &3  & & 6 & \\
\cline{2-14}
\Gamma_{18.6.3^3.1^3}
&\text{weight} & 
\multicolumn{12}{|c|}{\text{order of vanishing of form at cusps}} \\
\cline{1-14}
\sqrt[3]{b/d} & 0           &0&0&0 &&-1&& 0&0&0 &&1&\\
acd           &3            & 1&1&1 && 3&&  1&1&1 &&0&\\
(\sqrt[3]{b/d})    acd  &3  & 1&1&1 && 2&&  1&1&1 &&1&\\
(\sqrt[3]{b/d})^2acd  &3    & 1&1&1 && 1&&  1&1&1 &&2&\\
\hline
\hline
& \text{cusps}&\frac{5}{18}&\infty  &\frac{7}{18}&\frac{2}{5}& 0 &
\frac{2}{7}& & \frac{1}{2}   & &  &\frac{1}{3}& \\
\text{forms for} &\text{width}& 1&1 &1&6& 6 &6&  & 9 &  & & 6 & \\
\cline{2-14}
\Gamma_{9.6^4.1^3}
&\text{weight} & 
\multicolumn{12}{|c|}{\text{order of vanishing of form at cusps}} \\
\cline{1-14}
\sqrt[3]{b/c} & 0           & 0&0  &0&0&  0 &0&&-1 &&& 1&\\
acd           &3            & 1&1&1 &1& 1&1&  &3& &&0&\\
(\sqrt[3]{b/c})acd &3         & 1&1&1 &1&1&1&  &2& &&1&\\
(\sqrt[3]{b/c})^2acd &3       & 1&1&1 &1&1&1&  &1& &&2&\\
\hline
\hline
& \text{cusps}&&\infty  &&-1& 0 &1& & \frac{1}{2}   & & 
-\frac{2}{3}  &\frac{1}{3}&  \frac{4}{3} \\
\text{forms for} &\text{width}& &3 &&6& 6 &6&  & 9 &  & 2& 2 &2 \\
\cline{2-14}
\Gamma_{9.6^3.3.2^3}
&\text{weight} & 
\multicolumn{12}{|c|}{\text{order of vanishing of form at cusps}} \\
\cline{1-14}
\sqrt[3]{a/c}
      & 0           & &1  &&0&  0 &0&&-1 &&0& 0&0\\
bcd    &3            & &0& &1& 1&1&  &3& &1&1&1\\
(\sqrt[3]{a/c})bcd &3       & &1& &1& 1&1&  &2& &1&1&1\\
(\sqrt[3]{a/c})^2bcd &3     & &2& &1& 1&1&  &1& &1&1&1\\
\hline
\hline
& \text{cusps}&&\infty  &&& 0 && -\frac{1}{2}& \frac{1}{2}& \frac{3}{2}  & 
  -\frac{2}{3} &\frac{1}{3}&  \frac{4}{3} \\
\text{forms for} &\text{width}& &3 &&& 18 &&  3& 3 &3  & 2& 2 &2 \\
\cline{2-14}
\Gamma_{18.3^4.2^3}
&\text{weight} & 
\multicolumn{12}{|c|}{\text{order of vanishing of form at cusps}} \\
\cline{1-14}
\sqrt[3]{a/d} & 0   & &1  &&& -1 &&0& 0 &0&0& 0&0\\
bcd           &3    & &0& && 3&&  1&1&1 &1&1&1\\
(\sqrt[3]{a/d})bcd  &3   & &1& && 2&&  1&1&1 &1&1&1\\
(\sqrt[3]{a/d})^2bcd  &3 & &2& && 1&&  1&1&1 &1&1&1\\
\hline
\end{array}
$
\caption{Orders of vanishing at cusps for forms for subgroups of 
$\Gamma_1(6)$. 
}
\label{tab:cuspvaluesforlevel6}
\end{table}

\begin{table}[p]
\renewcommand\arraystretch{1.5}
$
\begin{array}{|ll|}
\hline
\Gamma_{24.6.1^6}& \\
\sqrt[3]{\eta( \tau)^{-4}\eta(2\tau)^{6}\eta(4\tau)^{16}} &=q + \frac{4}{3}q^2 + \frac{8}{9}q^3 + \frac{176}{81}q^4 - \frac{850}{243}q^5 - \frac{3488}{729}q^6 - \frac{5968}{6561}q^7 + \cdots\\
\sqrt[3]{\eta( \tau)^{4}\eta(2\tau)^{-6}\eta(4\tau)^{20}} &=
q - \frac{4}{3}q^2 + \frac{8}{9}q^3 - \frac{176}{81}q^4 - \frac{850}{243}q^5 + \frac{3488}{729}q^6 - \frac{5968}{6561}q^7 + \cdots
\\            
\hline
\hline
\Gamma_{8^36.3.1^3}& \\
\sqrt[3]{\eta( \tau)^{4}\eta(2\tau)^{10}\eta(4\tau)^{-4}\eta(8\tau)^{8}}
&=q - \frac{4}{3}q^2 - \frac{40}{9}q^3 + \frac{400}{81}q^4 + \frac{1454}{243}q^5 - \frac{1888}{729}q^6 - \frac{13168}{6561}q^7 + \cdots
\\
\sqrt[3]{\eta( \tau)^{8}\eta(2\tau)^{-4}\eta(4\tau)^{10}\eta(8\tau)^{4}}
&=q - \frac{8}{3}q^2 + \frac{8}{9}q^3 + \frac{32}{81}q^4 - \frac{82}{243}q^5 + \frac{5440}{729}q^6 - \frac{24400}{6561}q^7 + \cdots
\\
\hline
\hline
\Gamma_{24.3.2^3.1^3}&\\
\sqrt[3]{\eta( \tau)^{-4}\eta(2\tau)^{22}\eta(4\tau)^{-8}\eta(8\tau)^{8}}
&=
q + \frac{4}{3}q^2 - \frac{40}{9}q^3 - \frac{400}{81}q^4 + \frac{1454}{243}q^5 + \frac{1888}{729}q^6 - \frac{13168}{6561}q^7 + \cdots
\\
\sqrt[3]{\eta( \tau)^{-8}\eta(2\tau)^{20}\eta(4\tau)^{2}\eta(8\tau)^{4}}
&=q + \frac{8}{3}q^2 + \frac{8}{9}q^3 - \frac{32}{81}q^4 - \frac{82}{243}q^5 - \frac{5440}{729}q^6 - \frac{24400}{6561}q^7 + \cdots
\\
\hline
\hline
\Gamma_{8^32^33^2}&\\
\sqrt[3]{\eta(2\tau)^{20}\eta(4\tau)^{-6}\eta(8\tau)^{4}} &=
q^{2/3} - \frac{20}{3}q^{8/3} + \frac{128}{9}q^{14/3} - \frac{400}{81}q^{20/3} + \cdots\\
\sqrt[3]{\eta(2\tau)^{16}\eta(4\tau)^{6}\eta(8\tau)^{-4}} &=
q^{1/3} - \frac{16}{3}q^{7/3} + \frac{38}{9}q^{13/3} + \frac{1696}{81}q^{19/3}
 + \cdots\\
\hline
\end{array}
$
\caption{$q$-expansions of basis of forms for $S_3(\Gamma)$ for 
four subgroups of $\Gamma_0(8)\cap\Gamma_1(4)$
}
\label{tab:qexpansionslevel8}
\end{table}
\begin{table}[p]
\renewcommand\arraystretch{1.5}
$$
\begin{array}{|ll|}
\hline
{\Gamma_{18.6.3^3.1^3}}&\\
ab^{1/3}cd^{2/3}& =\sqrt[3]{\eta( \tau)^{4}\eta(2\tau)^{7}\eta(3\tau)^{-4}\eta(6\tau)^{11}}   
= q - \frac{4}{3}q^2 - \frac{31}{9}q^3 + \frac{400}{81}q^4 + \frac{104}{243}q^5+\cdots\\
ab^{2/3}cd^{1/3}& =\sqrt[3]{\eta( \tau)^{-4}\eta(2\tau)^{11}\eta(3\tau)^{4}\eta(6\tau)^{7}}   
= q + \frac{4}{3}q^2 - \frac{7}{9}q^3 - \frac{112}{81}q^4 - \frac{616}{243}q^5 +\cdots\\ 
\hline
\hline{\Gamma_{9.6^4.1^3}}&\\ 
ab^{1/3}c^{2/3}d& =\sqrt[3]{\eta( \tau)^{13}\eta(2\tau)^{-2}\eta(3\tau)^{-7}\eta(6\tau)^{14}} 
= q - \frac{13}{3}q^2 + \frac{32}{9}q^3 + \frac{670}{81}q^4 - \frac{3577}{243}q^5+\cdots\\
ab^{2/3}c^{1/3}d& =\sqrt[3]{\eta( \tau)^{14}\eta(2\tau)^{-7}\eta(3\tau)^{-2}\eta(6\tau)^{13}} 
= q - \frac{14}{3}q^2 + \frac{56}{9}q^3 - \frac{58}{81}q^4 + \frac{266}{243}q^5 +\cdots\\ 
\hline 
\hline
{\Gamma_{9.6^3.3.2^3}}&\\ 
a^{1/3}bc^{2/3}d& =\sqrt[3]{\eta( \tau)^{7}\eta(2\tau)^{4}\eta(3\tau)^{11}\eta(6\tau)^{-4}}   
= q^{ \frac{1}{3}} - \frac{7}{3}q^{ \frac{4}{3}} - \frac{19}{9}q^{ \frac{7}{3}} 
+ \frac{193}{81}q^{ \frac{10}{3}} + \frac{2306}{243}q^{ \frac{13}{3}}+\cdots\\
a^{2/3}bc^{1/3}d& =\sqrt[3]{\eta( \tau)^{11}\eta(2\tau)^{-4}\eta(3\tau)^{7}\eta(6\tau)^{4}} 
= q^{ \frac{2}{3}} - \frac{11}{3}q^{ \frac{5}{3}} + \frac{23}{9}q^{ \frac{8}{3}} 
- \frac{13}{81}q^{\frac{11}{3}} + \frac{2495}{243}q^{ \frac{14}{3}}+\cdots\\ 
\hline
\hline {{\Gamma_{18.3^4.2^3}}}&\\ 
a^{1/3}bcd^{2/3}& =\sqrt[3]{\eta( \tau)^{-2}\eta(2\tau)^{13}\eta(3\tau)^{14}\eta(6\tau)^{-7}} 
= q^{ \frac{1}{3}} + \frac{2}{3}q^{ \frac{4}{3}} - \frac{28}{9}q^{ \frac{7}{3}} 
- \frac{482}{81}q^{ \frac{10}{3}} - \frac{736}{243}q^{ \frac{13}{3}} +\cdots\\
a^{2/3}bcd^{1/3}& =\sqrt[3]{\eta( \tau)^{-7}\eta(2\tau)^{14}\eta(3\tau)^{13}\eta(6\tau)^{-2}} 
= q^{ \frac{2}{3}} + \frac{7}{3}q^{ \frac{5}{3}} + \frac{14}{9}q^{ \frac{8}{3}} 
- \frac{148}{81}q^{ \frac{11}{3}} - \frac{1708}{243}q^{ \frac{14}{3}}                     
+\cdots\\
\hline
\end{array}
$$
\caption{Basis of weight three cusp forms for 
some index $3$ subgroups of $\Gamma_1(6)$. $a,b,c,d$ are eta products as in
Table~\ref{tab:cuspvaluesforlevel6}.}
\label{tab:level6q-expansions}
\end{table}

\clearpage

\section{Traces and Point Counting}
\label{S:traces}
As described by Scholl, 
corresponding to each of these families, we have a representation
on parabolic cohomology:
\begin{equation}
\rho   = \rho_l:
\Gal(\overline \Q/\Q)\rightarrow H^1(X(\Gamma),
j_* R^1 f_* \Q _l).
\label{eqn:notationforgalrep}
\end{equation}
Here 
\[
 E^{\circ} (\Gamma ) \overset {f}{\longrightarrow} Y(\Gamma)
 \overset {j}{\hookrightarrow}X(\Gamma),
\]
with 
\[
 Y(\Gamma) = \Gamma \backslash \mathbf{H}, \quad
 X(\Gamma) = (\Gamma \backslash \mathbf{H})^*
\]
$E^{\circ} (\Gamma )$ be a family of elliptic curves over 
$Y(\Gamma)$. 
We let $\msr{F} = j_* R^1 f_* \Q _l$, an $l$-adic sheaf for the 
\'etale topology on $X(\Gamma)$.
We computed the traces of the Frobenius elements of 
this representation via point counting, as in \cite{LLY} and \cite{ALL}.
\subsection{Equations for elliptic surfaces associated with
the noncongruence subgroups}
\label{subsection:ellipticsurfacesnoncong}
As in section Section~\ref{SS:beauv},
associated to $\Gamma_0(8)\cap\Gamma_1(4)$ and $\Gamma_1(6)$, 
we have families of elliptic curves $E_8(t)$ and $E_6(t)$
as given in Table~\ref{tab:weierstrassforbeauville}:
\begin{eqnarray}
E_8(t): &&y^2 + 4xy + 4t^2y =x^3 + t^2x^2
\label{eqn:E8}\\
E_6(t): &&y^2 + (t+1)xy + (t-t^2)y =x^3 + (t-t^2)x^2.
\label{eqn:E6}
\end{eqnarray}
Thus we have elliptic surfaces $E_8$ and $E_6$, with fibrations
$$f_8:E_8\rightarrow X(\Gamma_0(8)\cap\Gamma_1(4))$$
and
$$f_6:E_6\rightarrow X(\Gamma_1(6)),$$
with fibres given by
$f_8^{-1}(t)=E_8(t)$ and $f_6^{-1}(t)=E_6(t)$.

By composing the covering maps given in
Table~\ref{tab:coveringmaps} with the fibrations $f_8$ or $f_6$,
associated with our noncongruence subgroups we have the
families of elliptic curves given in Table~\ref{tab:equationsofnonconfamilies}.
Our notation is explained by example:  The elliptic surface 
$E(\Gamma_{8^3.2^3.3^3})$
corresponding to
$\Gamma_{8^3.2^3.3^2}$ has a fibration
$$f: E(\Gamma_{8^3.2^3.3^3})
\rightarrow X(\Gamma_{8^3.2^3.3^3}),$$
with fiber $f^{-1}(r)$ having
an equation
$$y^2 + 4xy + 4\left(\frac{r^3-1}{r^3+1}\right)^2y
=x^3 + 4\left(\frac{r^3-1}{r^3+1}\right)^2x^2,
$$
i.e., the $t$ in (\ref{eqn:E8}) 
is replaced by $m^{-1}(r^3)=\frac{r^3-1}{r^3+1},$ where
$m(t)=\frac{1+t}{1-t}$.
This family of elliptic curve is denoted by $E_8\left(\frac{r^3-1}{r^3+1}\right)$.
The other families are constructed and denoted in a similar way.
\begin{table}
$
\renewcommand\arraystretch{1.5}
\begin{array}{|l|l|l|l|l|}
\cline{1-2}\cline{4-5}
\text{group} & \text{family of curves}&&
\text{group} & \text{family of curves}\\
\cline{1-2}\cline{4-5}
\Gamma_{24.6.1^6}     &E_8(r^3) &&
\Gamma_{18.6.3^3.1^3} &E_6({9r^3}) \\
\Gamma_{8^32^33^2}    &E_8\left(\frac{r^3-1}{r^3+1}\right) &&
\Gamma_{9.6^3.3.2^3}  &E_6\left(\frac{1-3r^3}{9-3r^3}\right) \\
\Gamma_{8^36.3.1^3}   &E_8(4r^3-1)&&
\Gamma_{9.6^4.1^3}    &E_6\left(1-\frac{8}{9r^3}\right) \\
\Gamma_{24.3.2^3.1^3} &E_8\left(\frac{2}{r^3-2}\right) &&
\Gamma_{18.3^4.2^3}   &E_6\left(\frac{1}{9(8r^3+1)}\right)\\
\cline{1-2}\cline{4-5}
\end{array}
$
\caption{Families of elliptic curves $E_n(m^{-1}(r^3))$
corresponding to certain noncongruence
subgroups.}
\label{tab:equationsofnonconfamilies}
\end{table}

We computed the traces of Frobenius by summing local terms 
using: 

\begin{thm}
\[
\Tr (\Frob _q | H^1 (X(\Gamma ), \msr{F}))
= - \sum_{x \in X(\fq)} \Tr (\Frob _q |\msr{F}_x ).  
\]
\end{thm}
\begin{proof}
This follows from 
Grothendieck-Lefschetz trace formula because the other terms
 $H^i (X(\Gamma ), \msr{F})), \ i \ne 1$ are zero. 
\end{proof}

The following is also well known:

\begin{thm}
\label{thm:Grothendieck-Lefschetz}
$\Tr (\Frob _q |\msr{F}_x )$ may be computed according to
the following:
\begin{enumerate}
\item If the fiber $E_x$ is smooth, then 
\[
 \Tr ({\Frob} _q |\msr{F}_x ) =  
 \Tr ({\Frob} _q |H^1 (E_x, \Q _l) ) =
q +1 - \#E_x (\fq). 
\]
\item
If  the fiber $E_x$ is singular, then Tate's algorithm tells us that
\[
 \Tr ({\Frob} _q |\msr{F}_x ) = 
\begin{cases}
1\text{  if the fiber is split multiplicative.}\\
-1\text{  if the fiber is nonsplit multiplicative.}\\
0\text{  if the fiber is additive.}\\
\end{cases}
\]
\item
If $E$ is a singular curve over a field with characteristic not $2$ or $3$,
given by an equation
$$E: y^2 = x^3 + ax + b,$$ then the reduction type of $E$ is determined as follows:
\[
\left.
\begin{array}{l}
\text{additive }\\
\text{split multiplicative }\\
\text{nonsplit multiplicative}
\end{array}
\right\}
\text{ if } -2ab \text{ is}
\left\{
\begin{array}{l}
\text { $0$ in $k$}\\
\text { a nonzero square in $k$}\\
\text { not a square in $k$}
\end{array}
\right.
\]
\end{enumerate}
\end{thm}
In order to apply part (3) of the above result, we need to transform $E_8(t)$ and
$E_6(t)$ in to the simplified Weierstrass form $y^2=x^3 + ax +b$.  
We obtain the following curves,
isomorphic to the originals, over any field of characteristic not $2$ or $3$.
\begin{eqnarray}
\widetilde E_8 :&& y^2=x^3 - 27(t^4 - 16t^2 + 16)x + 54(t^2-2)(t^4 + 32t^2 - 32)\\
\widetilde E_6 :&& y^3=x^3 -2^43^3(3t - 1)(3t^3 - 3t^2 + 9t - 1)x\\
&& \hspace{1.5cm}
-2^73^3(3t^2 + 6t - 1)(9t^4 - 36t^3 + 30t^2 - 12t + 1)
\nonumber
\end{eqnarray}
Thus one may compute values of the trace by
using the above result, for example with {\sc Magma}. 
The results for a range of values of $p$ and various covers of $E_8$ and $E_6$
are given in Table~\ref{traces}.

\begin{table}[h]
$$
\hspace{-1in}
\renewcommand\arraystretch{1.5}
\begin{array}{|c|c|c|c|c|c|c|c|c|c|c|c|c|c|c|c|c|}
\hline
\text{Group}&\text{Equation}&p & 5& 7& 11& 13& 17& 19& 23 & 73 \\ \hline
\Gamma_{24.6.1^6}     &E_8(r^3)&\Tr _p & 0 &4 & 0 & -44 & 0 & 52 & 0 & -92 \\
&        &\Tr_{p^2} &100 & -188 &    484 &    292 &  1156 &    -92 &   2116 & -17084\\   
\hline
\Gamma_{8^32^33^2}   
 &E_8\left(\frac{r^3-1}{r^3+1}\right)&\Tr _p & 0 &-4 & 0 & -44 & 0 & -52 & 0 & -92 \\
&        &\Tr_{p^2} &100 & -188 &    484 &    292 &  1156 &    -92 &   2116 & -17084\\   
\hline
\Gamma_{8^36.3.1^3}   &E_8(r^3-1)&\Tr _p &0& -3& 0& 13& 0 & 33& 0 & -71 \\
&&\Tr_{p^2}&
-44 &-95 &52 &169 &1012 & -359 &-1772 & 5617\\
\cline{2-11}
 &E_8(2r^3-1)&\Tr _p &0& 3& 0& 13& 0 & -33& 0 & -71 \\
&&\Tr_{p^2}&
-44 &-95 &52 &169 &1012 & -359 &-1772 & 5617\\
\cline{2-11}
  &E_8(4r^3-1)&\Tr _p &0& 0& 0& -26& 0 & 0& 0 & 142  \\
&& \Tr_{p^2}& -44 & 190 &52 & -338 &1012 &718 &-1772 & -11234 \\
\hline
\Gamma_{24.3.2^3.1^3} &E_8\left(\frac{2}{r^3-2}\right)&\Tr _p &0& 0& 0& -26& 0 & 0& 0 & 
142  \\
&& \Tr_{p^2}& -44 & 190 &52 & -338 &1012 &718 &-1772 & -11234 \\
\hline
\hline
\Gamma_{18.6.3^3.1^3} &E_6(3r^3)&\Tr _p&0 &-11 & 0& -5& 0 & 19 & 0& 76\\
&&\Tr_{p^2} & 28& -23& 196& 313& 508& 361& 316& -18428 \\
\hline
\Gamma_{18.6.3^3.1^3} &E_6(9r^3)&\Tr _p&0 &22 & 0& 10& 0 & -38 & 0& 76\\
&&\Tr_{p^2} & 28& 46& 196& -626& 508& -722& 316& -18428 \\
\hline
\Gamma_{9.6^3.3.2^3}  &E_6\left(\frac{1-3r^3}{9-3r^3}\right)&\Tr _p&0 &22 &0 &10&0 &-38&0 &76 \\
&&\Tr_{p^2} & 28& 46& 196& -626& 508& -722& 316& -18428 \\
\hline
\Gamma_{9.6^4.1^3}    &E_6\left(1-\frac{24}{r^3}\right)
&\Tr _p& 0& 7& 0& -5& 0& -17& 0& -248\\
& & \Tr_{p^2} &
64& 49& 448& 313& -140& 433& 1972& 9436
\\
\cline{2-11}
&E_6\left(1-\frac{8}{3r^3}\right)
&\Tr _p&0 &-14& 0 & 10& 0 & 34& 0 & -248\\
&&\Tr_{p^2} & 64& -98& 448& -626& -140& -866& 1972& 9436\\
\hline
\Gamma_{18.3^4.2^3}   &E_6\left(\frac{1}{24r^3+9}\right)&
\Tr _p&
0 &-14& 0 & 10& 0 & 34& 0 & -248\\
&&\Tr_{p^2} & 64& -98& 448& -626& -140& -866& 1972& 9436\\
\hline
\end{array}
$$
\caption{Table of $\Tr \, \rho^* ({\Frob}_p)$.}
\label{traces}
\end{table}

\clearpage

\section{Involutions and Isogenies}
\label{S:II}

\subsection{Involutions}
\label{SS:In}
The
four dimensional representations on 
$H^1(X(\Gamma),\msr F_\Gamma)$ in fact split into two $2$-dimensional
Galois representations.  We can achieve this splitting by using an
involution on $\Gamma\setminus \HH$ which extends to either an
automorphism or isogeny on the elliptic surface.

For each family given in Table~\ref{tab:equationsofnonconfamilies}
 by an equation $E_n(r)$, 
corresponding
to a covering $r^3=m(t)$,
we have involutions $\iotaA$ and $\iota$ of $t$ and $r$,
given in Table~\ref{table:involutionsonbase},
such that the following diagram commutes.
$$
\xymatrix{
{\mathbb P^1}\ar[r]^{r\mapsto \iota(r)}
\ar[d]_{r\mapsto r^3=m(t)}
&{\mathbb P^1}\ar[d]^{r\mapsto r^3=m(t)}\\
{\mathbb P^1}\ar[r]_{t\mapsto \iotaA(t)}
&{\mathbb P^1}
}
$$
Furthermore, if $c_1, c_2$ are the 
ramified cusps of the map $r\mapsto r^3=m(t)$, 
and $c_3, c_4$ are the unramified cusps, then
$\iotaA$ fixes the sets $\{c_1,c_2\}$ and $\{c_3,c_4\}$.
This means that the involution
$\iotaA$
 lifts to an involution
$\iota$ of $r$, as indicated in Table~\ref{table:involutionsonbase}.
To check these are the correct maps, one just needs to verify that
$(\iota(\sqrt[3]{m(t)}))^3=m(\iotaA(t))$, which is simple algebra.

\begin{table}
$$
\renewcommand\arraystretch{1.5}
\begin{array}{|l|ll|l|cc|}
\hline
\multicolumn{6}{|c|}{\text{
Involutions  $\iotaA$ 
of $X(\Gamma_0(8)\cap\Gamma_1(4))$, and $\iota$ of
$X(\Gamma)$, for $\Gamma\subset\Gamma_0(8)\cap\Gamma_1(4)$
}}
\\
\hline
\text{subgroup}&\multicolumn{2}{|l|}{\text{values of $\tau$ and $t$ where}}
&r^3= & \multicolumn{2}{c|}{\text{involutions of $t$ and $r$}}
\\
\Gamma&
\multicolumn{2}{c|}{\text{cover ramifies}}
&m(t)
&\iotaA: t\mapsto & \iota: r\mapsto \\
\hline
& \tau\hspace{2cm} & t(\tau)\hspace{0cm} &&&\\
\hline
\Gamma_{24.6.1^6}    & 1/2, 0   &\infty, 0 & t & -t & -r\\
\Gamma_{8^2.2^3.3^2}& \infty, 1/4  &1, -1&\frac{t+1}{1-t} & 1/t & -r\\
\Gamma_{8^36.3.1^3}  & 1/2, 1/4 &\infty, -1& \frac{t+1}{4}& \frac{1-t}{1+t}& \frac{1}{2r}\\
\Gamma_{24.3.2^3.1^3}  & 0, 1/4   &0, -1      &\frac{2(1+t)}{t}& \frac{t+1}{t-1}& \frac{2}{r}\\
\hline
\hline
\multicolumn{6}{|c|}
{\text
{
Involutions  $\iotaA$ 
of $X(\Gamma_1(6))$, and $\iota$ of
$X(\Gamma)$ for } \Gamma\subset\Gamma_1(6)
} 
\\
\hline
\text{subgroup}&\multicolumn{2}{l|}{\text{values of $\tau$ and $t$ where}}
&r^3= & \multicolumn{2}{c|}{\text{involutions of $t$ and $r$}}
\\
\Gamma&
\multicolumn{2}{c|}{\text{cover ramifies}}
&m(t)
&\iotaA: t\mapsto & \iota: r\mapsto \\
\hline
& \tau & t(\tau) &&&\\
\hline
\Gamma_{18.6.3^3.1^3}   &1/3, 0 &\infty,0 &t/9 & \frac{1}{9t} & \frac{1}{9r}\\
\Gamma_{9.6^3.3.2^3}   &\infty, 1/2 &\frac{1}{9}, 1&\frac{1-9t}{3(1-t)}&\frac{1}{9t} & \frac{1}{r}\\
\Gamma_{9.6^4.1^3}     &1/2, 1/3 &1, \infty&\frac{8}{3(1-t)} & \frac{1-9t}{9-9t}& \frac{2}{r}\\
\Gamma_{18.3^4.2^3}  &\infty, 0 &\frac{1}{9}, 0&\frac{1-9t}{24t} & \frac{1-9t}{9-9t}& \frac{1}{2r}
\\
\hline
\end{array}
$$
\caption{Involutions of modular curves $\Gamma\setminus \HH$.
For $\Gamma_0(8)\cap \Gamma_1(4)$, 
$t(\tau)=\frac{\eta(z)^8\eta(4z)^4}{\eta(2z)^{12}}$,
and for $\Gamma_1(6)$, $t(\tau)=\frac{1}{9}
\frac{\eta(6\tau)^{4}\eta(\tau)^{8}}  
{\eta(3\tau)^{8}\eta(2\tau)^{4}}$, as in
Tables~\ref{tab:weierstrassforbeauville},
\ref{tab:cuspvaluesG8}, and
\ref{tab:cuspvaluesforlevel6}.
}
\label{table:involutionsonbase}
\end{table}

\subsection{Isogenies}
\label{SS:Iso}

The involutions  $\iotaA$ of
 modular curves given in Table~\ref{table:involutionsonbase}
lift to maps
\begin{eqnarray}
\nonumber
\tilde\iotaA: E_n &\rightarrow& E_n\\
\tilde\iotaA: (t,x,y) \in E_n(t) &\mapsto& (\iotaA(t),
\
\iotaA_x(t,x,y),
\
\iotaA_y(t,x,y)),
\end{eqnarray}
where $n=8$ or $6$, which restrict to isogenies 
between the fibres
of the corresponding family of elliptic curves
(given by (\ref{eqn:E8}) and (\ref{eqn:E6})).
From the isogenies of the families
$E_6(t)$, $E_8(t)$, 
one can obtain the isogenies on the families
$E_6(m^{-1}(r^3))$, $E_8(m^{-1}(r^3))$, lifting  $\iota$ to $\tilde\iota$.
These isogenies will give rise to involutions on the level of cohomology.

To show that two curves $E(t)$ and $E(\iotaA(t))$
are isogenous by an isogeny of degree $d$, it suffices to show that
$\Phi_d(j(E(t)),j(E(\iotaA(t))))=0$,
where $\Phi_d$ is the $d$th modular polynomial.
The isogeny can be explicitly determined by Velu's methodfrom a subgroup of
order $d$ on $E(t)$.  
Although the algorithms involved are well known and not difficult theoretically, 
in practice they should be
carried out with the help of a computer program, such as 
{\sc Magma} \cite{magma}, because of the large number of
of terms in the polynomials involved.
For example, $\Phi_8$ is a polynomial in two variables of degree
$20$ with $141$ terms; 
$\Phi_n$ can be found in a {\sc Magma} database using the
command {\tt ClassicalModularPolynomial(n)}  for $1\le n\le 17$.

Although it's not important to know the isogeny exactly, we do
need to know the field over which the map is defined.  This information
was computed with the assistance of {\sc Magma}, and is given in 
Table~\ref{table:isogeniesfieldofdefinition}.  The polynomials given in this table
are such that their roots are the $x$-coordinates of points in the kernel of the
isogeny.  

\begin{table}
$$
\renewcommand\arraystretch{1.5}
\begin{array}{llllll}
\hline
\text{subgroup}&
 \iotaA(t)  && d  & \text{\small polynomial 
defining}
&\tilde\iota\text{'s field of}\\
&&&&\text{kernel of isogeny}&\text{ definition}\\
\hline
\text{Level $8$ cases}
\\
\hline
\Gamma_{24.6.1^6}    &-t & &1&- & \Q\\
\Gamma_{8^2.2^3.3^2}& 1/t & &4&(x+t^2)x&\Q\\
 \Gamma_{8^36.3.1^3}  &\frac{1-t}{1+t}& &8&(x^2 - 4tx - 4t^3)(x+t^2)x & \Q[\sqrt{-1}]\\
\Gamma_{24.3.2^3.1^3} &\frac{t+1}{t-1}&& 8&(x^2 + 4tx + 4t^3)(x+t^2)x & \Q[\sqrt{-1}]
 \\
\hline
\text{
Level $6$ cases
}
\\
\hline
\Gamma_{18.6.3^3.1^3}, \Gamma_{9.6^4.1^3}   &\frac{1}{9t} && 3 &  x - t^2 + t
&\Q[\sqrt{-3}]\\
\Gamma_{9.6^3.1^3}, \Gamma_{18.3^4.2^3}& \frac{1-9t}{9-9t}&&6&(x - t^2 + t)x(x+t)
 &\Q[\sqrt{-3}]\\
 \hline
\end{array}
$$
\caption{Data concerning involutions $\iotaA$ and $\iota$
of Table~\ref{table:involutionsonbase}, 
lifted to maps $\tilde\iota$ of families of curves, defining isogenies of
degree $d$ on fibres.  In particular, 
$\Phi_d(j(E_n(\iotaA(t))),j(E_n(t)))=0$ where $n$ is the level, and
$\Phi_d$ is the $d$th modular polynomial.}
\label{table:isogeniesfieldofdefinition}
\end{table}

\subsection{Isogenous relationships between families}
\label{subsec:isog_between_families}

In the previous section we showed how
involutions give rise to isogenies on the fibres,  which will
resulting in involutions on the cohomology of each family.
There are also isogenous maps between families, which explain our groupings into
pairs of cases,
which was originally based on the relationships between traces
seen in Table~\ref{traces}.  
Combining the relations between curves we already have, we find that 
\begin{eqnarray*}
\Phi_8\left(
j\left(
E_6\left(
\frac{t-1}{t+1}
\right)
\right),
j\left(
E_8\left(
\phi_1(t)
\right)
\right)
\right)&=&0
\\
\Phi_8\left(j(E_8(4t-1)), j\left(E_8\left(\frac{2}{\phi_2(t)-2}\right)\right)\right)&=&0\\
\\
\Phi_6\left(
j\left(
E_6\left(
\frac{1-3t}{9-3t}
\right)
\right),
j\left(
E_6\left(
9\phi_3(t)
\right)
\right)
\right)&=&0
\\
\Phi_3\left(
j\left(E_6\left(1-\frac{8}{3t}\right)\right), 
j\left(
E_6\left(
  \frac{1}{9-24\phi_4(t)}\right)\right)\right)&=&0,
\end{eqnarray*}
where $\phi_1(t)=\phi(2)=1/t, \phi_3(t)=t/3, \phi_4(t)=-1/t$.
This may also be checked directly with {\sc Magma}. 
Thus the maps $\phi_i$ between the bases lift to isogenies on the fibres between
families.  Replacing $t$ by $r^3$ in these equations does not change the relationships,
so this also holds for the covers, and these maps induce isomorphisms on the level
of cohomology.  Refer to Table~\ref{traces} for which cover corresponds to which group.

\clearpage

\section{Experimental data for the ASwD congruences}
\label{S:aswddata}
The strategy for finding an ASwD basis is the following: For our noncongruence
subgroup $\Gamma$, we have found a basis $h_1$, $h_2$ for 
$S_3 (\Gamma)$. We have also found a Hecke eigenform $f\in S_3 (\Gamma _0, \chi)$
for some congruence subgroup $\Gamma _0$. Let $a_n$ and $b_n$ respectively
be the expansion coefficients of $h_1$ and $h_2$. Let $A_n$ be the expansion 
coefficients of $f$.
We consider two possible situations.

\subsection{Case 1}
\label{ss:case1of_S:aswddata}

In the simplest case, $h_1, h_2$ is already an ASwD basis.
This case occurs in section \ref{SS:48}.  So
for good primes $p$ and integers $n$ with $p\not| n$
\begin{eqnarray}
a_{pn} \equiv A_pa_{n}\mr{\ mod\ }p^2 
\;\;\;\text{ and }\;\;\;
b_{pn} \equiv A_p b_{n}\mr{\ mod\ }p^2,
\end{eqnarray}
which implies, for $p$ fixed and $n$ varying with
$a_n\not=0$ and $b_n\not=0$,
\begin{eqnarray}
a_{pn}/a_{n} \equiv \text{constant}\mr{\ mod\ }p^2 
\;\;\;\text{ and }\;\;\;
b_{pn}/b_n \equiv \text{constant}\mr{\ mod\ }p^2.
\label{eqn:ASwD_test1}
\end{eqnarray}
So, our test for whether $h_1, h_2$ is an ASwD basis is 
to check whether 
$a_{pn}/a_n$ and
$b_{pn}/b_n$ take constant values for fixed $p$ and varying $n$,
with $np$ less than some fixed bound.
If this holds, then we also consider this to be evidence that
$h_1, h_2$ is an ASwD basis.  We can make this conclusion regardless of
whether $f$ is known.

In the case $n=1$, since $a_1 = b_1 = 1$, (\ref{eqn:ASwD_test1}) 
implies that
\begin{eqnarray}
a_{p} \equiv A_p\mr{\ mod\ }p^2 
\;\;\;\text{ and }\;\;\;
b_{p}\equiv A_p\mr{\ mod\ }p^2.
\label{eqn:ASwD_test2}
\end{eqnarray}

In order to determine the associated congruence modular form, we 
test whether (\ref{eqn:ASwD_test2}) holds for small primes
for the candidate form $f$.
This is what happens in subsection \ref{SSS:24}.

In some cases, to get congruences,
$f$ needs to be replaced by $f\otimes\chi$ for some character $\chi$.
Then $A_p$ will be replaced by $A_p\chi(p)$ in (\ref{eqn:ASwD_test2}),
so this phenomena can be recognized by checking whether
$A_p/a_p$ and $A_p/b_p$ are roots of unity.  
This happens in subsection \ref{SSS:82}.
However, we have not worked out what the character $\chi$ is.

\subsection{Case 2}
\label{ss:case2of_S:aswddata}
In most of our examples examples, it turns out that 
the ASwD basis depends on the congruence class of the prime $p$
modulo some small integer.  It turns out that for some primes,
(\ref{eqn:ASwD_test1}) holds for the values tested, 
in which case $h_1, h_2$ is assumed to be
the ASwD basis, but for other primes, this does not hold.

If (\ref{eqn:ASwD_test1}) does not hold for some prime $p$, then we will
assume that for this prime, an ASwD basis consists of
linear combinations of the form
$h_1 + \alpha h_2$, where $\alpha$ is an algebraic number of small degree,
such that for integers $n$ with $p\not| n$. 
the expansion coefficients satisfy 
\begin{eqnarray}
a_{pn} + \alpha b_{pn} \equiv A_p (a_{n} + \alpha b_{n} )\mr{\ mod\ }p^2.
\label{eqn:case2_eq1}
 \end{eqnarray}
A priori, $\alpha$ depends on $p$, though we will see that 
in the examples we are considering, evidence
suggests that it only depends on the congruence class of $p$ modulo 
a small integer.

For (\ref{eqn:case2_eq1}) to hold, it is sufficient, but not necessary,
that
\begin{eqnarray}
a_{pn} \equiv  A_p\alpha b_{n}\mod p^2,
\text{ and }\alpha b_{pn} \equiv A_p a_{n} \mod p^2,
\label{eqn:case2_eq2}
\end{eqnarray}
which, assuming all the terms are non-zero, implies that
$a_{pn}/b_n = A_p\alpha_p$ and $b_{pn}/a_n = A_p/\alpha_p$,
So if (\ref{eqn:ASwD_test1}) does not hold as $n$ varies,
we test whether  
\begin{eqnarray}
\frac {a_{np}}{ b_n}\equiv\text{constant}\mod p^2
\;\;\;\text{ and }\;\;\;
\frac {b_{np}}{ a_n}\equiv\text{constant}\mod p^2.
\label{eqn:case2_eqn4}
\end{eqnarray}
If this holds,
the values of $\alpha$ and $A_p$
mod $p^2$, up to sign, are determined by
\begin{eqnarray}
\alpha^2 \equiv 
\frac {a_{np}}{ b_n}\Big/\frac{ b_{np}}{ a_n}\mod p^2,
\;\;\;\text{ and }\;\;\;
A_p^2\equiv
\frac {a_{np}}{ b_n}\frac{ b_{np}}{ a_n}\mod p^2.
\label{eqn:case2_eqn3}
\end{eqnarray}
For $p$ for which (\ref{eqn:case2_eqn4}) holds,
there are two solutions to (\ref{eqn:case2_eqn3})
for $\alpha$, and the ASwD basis has the form
$h_1 + \alpha h_2, h_1 - \alpha h_2$.
We expect that $\alpha$ only depends on $p$ modulo some small integer.
Since $\alpha$ is expected to be an algebraic integer, but not an integer,
it may be difficult to guess the value of $\alpha$, from $\alpha\mod p^2$.
So we also look at powers of $\alpha\mod p^2$, and if for some small power
these are constant as $p$ varies, then we deduce a value of $\alpha$.
Once $\alpha$ is determined, $A_p\mod p^2$ is determined, if this agrees
with the coefficients of our congruence modular form, then we take this
as evidence that  
$h_1 + \alpha h_2, h_1 - \alpha h_2$ is an ASwD basis with $f$ the
associated new form.  As for case 1, we will also test whether
the $A_p$ must be multiplied some root of unity, presumably the value
$\chi(p)$ for some character $\chi$, though again, we have not determined
the character in question.

\subsection{Examples associated with 
newform in $S_3 (\Gamma _0 (48), \chi)$}
\label{SS:48}
For $\Gamma_{24.6.1^6}$ and $\Gamma_{8^3.2^3.3^3}$,
evidence suggests that the associated
congruence form is as follows, with the first few $A_p$ as in 
Table~\ref{table_coefs1}.
\begin{eqnarray}
f(z) &=&\displaystyle{
\frac{\eta(4z)^9\eta(12z)^9}
{\eta(2z)^3\eta(6z)^3\eta(8z)^3\eta(24z)^3}}
\label{eqn:examples:f_in_case1}
\\
&=&
q + 3q^3 - 2q^7 + 9q^9 - 22q^{13} - 26q^{19} - 6q^{21} + 25q^{25} +\ldots
\nonumber
\end{eqnarray}

\begin{table}[h]
$
\begin{array}{|c|ccccccccccccccccc|}
\hline
p & 5 & 7 & 11 & 13 & 17 & 19 & 23 & 29 & 31 & 37 & 41 & 43 & 47 & 53 & 59 & 61 & 67 \\ 
\hline
a_p &
  0 & -2 & 0 & -22 & 0 & -26 & 0 & 0 & 46 & 26 & 0 & 22 & 0 & 0 &  0 &74&-122\\ 
\hline
\end{array}
$
\caption{First few coefficients $A_p$ for newform for 
$S_3 (\Gamma _0 (48), \chi)$.}
\label{table_coefs1}
\end{table}

\subsubsection{Atkin Swinnerton-Dyer congruences for
$\Gamma_{24.6.1^6}$}
\label{SSS:24}

We have shown previously that $S_3(\Gamma_{24.6.1^6})$ has a 
basis
\begin{eqnarray}
\label{h1_for_24.6.1^6}
h_1(z) &=&\displaystyle{\sqrt[3]{\frac{\eta(z)^4\eta(4z)^{20}}{\eta(2z)^6}}=
 q - \frac{4}{3}
q^2 + \frac{8}{9}q^3 - 
\frac{176}{81}q^4 - \frac{850}{243}q^5  
\cdots}
\\
\label{h2_for_24.6.1^6}
h_2(z) &=&\displaystyle{
\sqrt[3]{\frac{\eta(4z)^{16}\eta(2z)^6}{\eta(z)^4}}
=
 q + \frac{4}{3}q^2 + 
\frac{8}{9}q^3 + 
\frac{176}{81}q^4 - 
\frac{850}{243}q^5\cdots}
\end{eqnarray}
\comment{
created in pari with:

ha1=(eta(q)^4*eta(q^4)^20/eta(q^2)^6)^(1/3)*q;
ha2=(eta(q^4)^16*eta(q^2)^6/eta(q)^4)^(1/3)*q;

}
The first few prime coefficients of these forms are:
$$
\begin{array}{ccccccccccc}
p& 2 & 3 & 5 & 7 & 11 & 13 & 17 & 19\\
a_p 
 & -\frac{4}{3}
 & \frac{8}{9}
 & -\frac{850}{243}
 & -\frac{5968}{6561}
 & -\frac{35104520}{4782969}
 & \frac{952141694}{129140163}
 & -\frac{206256733102}{31381059609}
 & \frac{60201506159720}{2541865828329}\\
b_p \rule{0ex}{3ex}
  & \frac{4}{3}
 & \frac{8}{9}
 & -\frac{850}{243}
 & -\frac{5968}{6561}
 & -\frac{35104520}{4782969}
 & \frac{952141694}{129140163}
 & -\frac{206256733102}{31381059609}
 & \frac{60201506159720}{2541865828329}
\end{array}
$$

\begin{table}[h]
$
\begin{array}{c|ccccccccccccccccccccccc}
p  & 5 & 7 & 11 & 13 & 17 & 19 & 23 & 29 & 31 & 37 & 41 & 43 & 47\\
\hline
a_{np}/a_n\mod p^2
& 0 & 47 & 0 & 147 & 0 & 335 & 0 & 0 & 46 & 26 & 0 & 22 & 0\\
b_{np}/b_n\mod p^2
& 0 & 47 & 0 & 147 & 0 & 335 & 0 & 0 & 46 & 26 & 0 & 22 & 0
\end{array}
$
\caption{values of 
$\frac {a_{np}}{ a_n}$ and $\frac  {b_{np}}{ b_n}$
for primes $p\ge 5$ and integers $n$, with $pn\le 500$.
These agree mod $p^2$ with values in Table~\ref{table_coefs1}.}
\label{table_ratios1}
\end{table}
Since the ratios $a_{np}/a_n$ and $b_{np}/b_n$, given
in Table~\ref{table_ratios1}
appear to be constant, and the
numbers in Tables~\ref{table_coefs1} and \ref{table_ratios1} 
agree modulo $p^2$,
we conclude that the ASwD basis 
of $S_3(\Gamma_{24.6.1^6})$ 
is $h_1, h_2$, as given by
(\ref{h1_for_24.6.1^6}) and (\ref{h2_for_24.6.1^6}) for all primes,
with $f$ in (\ref{eqn:examples:f_in_case1}) 
being the associated congruence form.


\subsubsection{Atkin Swinnerton-Dyer congruences for
$\Gamma_{8^3.2^3.3^3}$}
\label{SSS:82}

Basis of $S_3(\Gamma_{8^3.2^3.3^3})$,
written in terms of $r=q^{1/3}$ and $s=q^{2/3}$.
\begin{eqnarray*}
h_1(z) &=&\displaystyle{
\sqrt[3]{
\frac{\eta(2\tau)^{20}\eta(8\tau)^{4}}{\eta(4\tau)^{6}}} =
\sum_{n\ge1}a_ns^n=
s - \frac{20}{3}s^4 + \frac{128}{9}s^7 - \frac{400}{81}s^{10} + \cdots}\\
h_2(z) &=&
\displaystyle{
\sqrt[3]{\frac{\eta(2\tau)^{16}\eta(4\tau)^{6}}{\eta(8\tau)^{4}}} =
\sum_{n\ge1}b_nr^n=
r - \frac{16}{3}r^{7} + \frac{38}{9}r^{13} + \frac{1696}{81}r^{19}
 + \cdots}
\end{eqnarray*}

\comment{
created in pari with:

\\ note, r=q^{1/3} and s=q^{2/3}

he1=(eta(s^3)^20*eta(s^6)^(-6)*eta(s^12)^4)^(1/3)*s;
he2=(eta(r^6)^16*eta(r^12)^6*eta(r^24)^(-4))^(1/3)*r;

}

First few prime coefficients:
$$
\begin{array}{ccccccccccc}
p& 2 & 3 & 5 & 7 & 11 & 13 & 17 & 19\\
a_p 
 & 0
 & 0
 & 0
 & \frac{128}{9}
 & 0
 & -\frac{3454}{243}
 & 0
 & -\frac{38656}{6561}
\\
b_p \rule{0ex}{3ex}
 & 0
 & 0
 & 0
 & -\frac{16}{3}
 & 0
 & \frac{38}{9}
 & 0
 & \frac{1696}{81}
\end{array}
$$

Our computations show that the ratios
$\frac {a_{np}}{ a_n}$ and $\frac  {b_{np}}{ b_n}$
remain constant for fixed $p$, for values of $pn$ up to 500.  
We can write these ratios in terms of $\omega$, a 
sixth root of $1$ mod $p^2$, as in Table~\ref{table_ratios2}.
In this table we also
tabulate $\omega$, and the order of $\omega$ as an element of
$(\Z/p^2\Z)^\times$.

\begin{table}[h]
$
\begin{array}{|l|ll|ll|ll|}
\hline
p &\multicolumn{2}{|c|}{
\frac {a_{np}}{ a_n}\mod p^2} 
& \multicolumn{2}{|c|}{\frac  {b_{np}}{ b_n}
\mod p^2}
&\omega&o(\omega)\\
\hline
 7  &36   &=-2\omega& 11   &=-2\omega^{-1}    &31&6\\ 
 11 & 0   && 0    &&&\\
 13 & 168 &=-22\omega& 23   &=-22\omega^{-1}    &146&3\\ 
 17 & 0   && 0    &&&\\
 19 & 11  &=-26\omega& 324  &=-26\omega^{-1}  &69&6\\ 
 23 & 0   && 0    &&&\\
 29 & 0   && 0    &&&\\
 31 & 915 &=46\omega     & 915  &=46\omega^{-1}  &-1&2\\ 
 37 & 47  &=26\omega& 1296 &=26\omega^{-1}    &581&3\\ 
 41 & 0   && 0    &&&\\
 43 & 1827&=22\omega& 1827 &=22\omega^{-1}       &{-1}&2\\ 
 47 & 0   && 0    &&&\\
\hline
\end{array}
$
\caption{values of 
$\frac {a_{np}}{ a_n}$ and $\frac  {b_{np}}{ b_n}$ for $\Gamma_{8^3.2^3.3^3}$,
for primes $p\ge 5$ and integers $n$, with $pn\le 500$,
in terms of a 6th root of unity, $\omega$, with order $o(\omega)$.
Compare with values in Table~\ref{table_coefs1}.}
\label{table_ratios2}
\end{table}
Since the values of 
$a_{np}/a_n$ and
$b_{np}/b_n$ are constant over the ranges computed,
we conjecture that
$h_1, h_2$ is an ASwD basis for all primes.
Comparing these values with the coefficients of $f$, we conjecture
that the associated congruence form is $f\otimes \chi$ where
$\chi$ is a certain Hecke character.



\subsection{Examples associated with 
newform in $S_3 (\Gamma _0 (432), \chi)$}
\label{SS:432}

For $\Gamma_{8^3.6.3.1^3}$
and
$\Gamma_{24.3.2^3.1^3}$
evidence suggests that the associated
congruence form is
\begin{eqnarray}
\label{eqn:newform2}
f(z)&=&
q + 6\sqrt{2}q^5 + \sqrt{-3}q^7 + 6\sqrt{-6}q^{11}+  13q^{13} - 6\sqrt{2}q^{17} +\\
&& 11\sqrt{-3}q^{19} - 18\sqrt{-6}q^{23}+
  47q^{25} - 24\sqrt{2}q^{29}
+ \cdots
\nonumber
\end{eqnarray}

The first few $A_p$ are given in 
Table~\ref{table_coefs2}, where they are 
divided by either $1$, $\sqrt{2}$,
$\sqrt{3}$, or $\sqrt{-6}$, for easy readability
\begin{table}[h]
$
\begin{array}{|c|ccccccccccccc|}
\hline
p            & 5 & 7 & 11& 13& 17& 19& 23 & 29 & 31 & 37 & 41& 43 & 47\\
\hline
A_p          &   &   &   &13 &   &   &    &   &   &35&    &   &\\
A_p/\sqrt{2} & 6 &   &   &   &-6 &   &    &-24&   &  & 0  &   &\\
A_p/\sqrt{-3}&   &1  &   &   &   &11 &    &   &24 &  &    &-24&  \\
A_p/\sqrt{-6}&   &   &6  &   &   &   & -18&   &   &  &    &    &6  \\
\hline
\end{array}
$
\caption{Coefficients of $f$ in (\ref{eqn:newform2}) 
and (\ref{eqn:f_432_as_eta_and_einstein}).}
\label{table_coefs2}.
\end{table}

The form $f$ can be given in terms of eta products and an
Eisenstein series as follows:
\begin{equation}
f(z) = f_1(12z) + 6\sqrt{2} f_5(12z) + \sqrt{-3}f_7(12z) + 6\sqrt{-6}f_{11}(12z),
\label{eqn:f_432_as_eta_and_einstein}
\end{equation}
where
\begin{eqnarray}
f_1(z)&=& \frac{\eta(2z)^3\eta(3z)}{\eta(6z)\eta(z)}E_6(z)
\label{eqn:defoff1}
\\
f_5(z)&=&
\frac{\eta(z)\eta(2z)^3\eta(3z)^3}{\eta(6z)}
\\
f_7(z)&=&\frac{\eta(6z)^3\eta(z)}{\eta(2z)\eta(3z)}E_6(z)
\\
f_{11}(z)&=&
\frac{\eta(3z)\eta(z)^3\eta(6z)^3}{\eta(2z)}
\\
\text {where }E_6(z)&=&
1 + 12\sum_{n\ge 1}(\sigma(3n) -3\sigma(n))q^n,
\end{eqnarray}
and $\sigma(n)=\sum_{d|n} d$.

\subsubsection{Atkin Swinnerton-Dyer congruences for
$\Gamma_{8^3.6.3.1^3}$}
\label{SSS:83}

We have seen that a basis of $S_3(\Gamma_{8^3.6.3.1^3})$ can be given by:
\begin{eqnarray*}
h_1(z)
&=&
\displaystyle{
\sqrt[3]{\frac{\eta(z)^{4}\eta(2z)^{10}\eta(8z)^{8}}{\eta(4z)^4}}}
=
\sum_{n\ge 1}a_nq^n
=
q - \frac 4 3q^2 - \frac{40}{9}q^3 + \frac{400}{81}q^4 
+ \frac{1454}{243}q^5 + \cdots
\\
h_2(z)
&=&
\displaystyle{
\sqrt[3]{\frac{\eta(z)^{8}\eta(4z)^{10}\eta(8z)^{4}}{\eta(2z)^4}}}
=
\sum_{n\ge 1}b_nq^n
=
 q - \frac 8 3q^2 + \frac 8 9q^3 + \frac {32} {81}q^4 
- \frac{82}{243}q^5 + \ldots
\end{eqnarray*}
\comment{
created in pari with:

hb1=(eta(q)^4*eta(q^2)^10*eta(q^8)^8/eta(q^4)^4)^(1/3)*q;
hb2=(eta(q)^8*eta(q^4)^10*eta(q^8)^4/eta(q^2)^4)^(1/3)*q;

}
The first few prime coefficients of $h_1$ and $h_2$ are as follows:
$$
\begin{array}{ccccccccccc}
p& 2 & 3 & 5 & 7 & 11 & 13 & 17 & 19\\
a_p 
 & -\frac{4}{3}
 & -\frac{40}{9}
 & \frac{1454}{243}
 & -\frac{13168}{6561}
 & \frac{38671144}{4782969}
 & -\frac{2230795138}{129140163}
 & -\frac{418720079278}{31381059609}
 & \frac{30660416258552}{2541865828329}\\
b_p \rule{0ex}{3ex}
  & -\frac{8}{3}
 & \frac{8}{9}
 & -\frac{82}{243}
 & -\frac{24400}{6561}
 & \frac{16345336}{4782969}
 & \frac{1236747902}{129140163}
 & \frac{842483994194}{31381059609}
 & -\frac{34758650729368}{2541865828329}
\end{array}
$$

For $p\equiv 1\mod 3$, our data suggests that
$a_{pn}/a_p$ and $b_{pn}/b_n$ remain constant as $n$ varies, with
values as in Table~\ref{table_ratios3}.  This means we are in case 1, described
in subsection~\ref{ss:case1of_S:aswddata}.
Experimentally, we noted that for these $p$ we always have
$\left(\frac{a_{pn}}{a_p}\big/\frac{b_{pn}}{b_n}\right)^6\equiv 1\mod p^2$ 
(excluding the case $p=13$, when $a_{pn}\equiv b_{pn}\equiv 0\mod 13$).
We also checked that
$\frac{a_{pn}}{a_p}\times \frac{b_{pn}}{b_n}
\equiv A_p^2\mod p^2$ where
the $A_p$ are as in Table~\ref{table_coefs2}.
The first observation indicates that these two forms correspond to
congruence forms which are
twists of each other by an order $6$ character, and the second observation
indicates that the congruence form is the $f$ given by (\ref{eqn:newform2}).
Using these two observations, we write the ratios $a_{np}/a_n$ and
$b_{np}/b_n$ in the factored forms in Table~\ref{table_ratios3}.  The values of
$\omega$, a
sixth root of $1$, and the values used for
$\sqrt{3}\mod p^2$ are also tabulated.

\begin{table}
$
\begin{array}{|l|ll|ll|ll|}
\hline
p &
\multicolumn{2}{|c|}
{\frac{a_{np}}{ a_n}\mod p^2}   & 
\multicolumn{2}{|c|}
{\frac{b_{np}}{b_n}\mod p^2}
&
\sqrt{-3}
&
\omega
\\
\hline
7&   17	&= \omega^{-4} \sqrt{-3}  &  29  &= \omega^{-2} \sqrt{-3}  &37&\sqrt[4]{-18}\\
13&  52	&= \omega^{-2} 13         &  130 &= \omega^2 13         &&
\sqrt{23}\\ 		
19&  48	&= \omega^{-2} 11\sqrt{-3}&  346 &= \omega^{-4} 11\sqrt{-3}&137
&\sqrt{69}\\ 	
31&  915&= \omega^{6} 24\sqrt{-3}&  46  &=  24\sqrt{-3}&82&\sqrt[6]{-1}\\ 
37&  165&= \omega^{-4} 35         &  1169&= \omega^4 35         &&\sqrt[4]{581}\\
43&  11	&=-\omega^{6} 24\sqrt{-3}&  1838&=- 24\sqrt{-3}&1002&
\sqrt[6]{-1}\\ 
\hline
\end{array}
$
\caption{Values of $a_{np}/a_n$ and $b_{np}/b_n$ for $p\equiv 1\mod 3$,
for $h_1$ and $h_2$ for $\Gamma_{8^3.6.3.1^3}$,
in terms of $A_p$ in Table~\ref{table_coefs2}.
}
\label{table_ratios3}
\end{table}

Based on these experiments, we conjecture that
the Atkin Swinnerton-Dyer basis 
of $S_3(\Gamma_{8^3.6.3.1^3})$
when $p\equiv 1\mod 3$ is $h_1, h_2$,
and the associated congruence forms are 
$f\otimes\chi$ and
$f\otimes\chi^{-1}$ for a certain Hecke character.

\begin{table}
$
\begin{array}{|l|ll|cr|}
\hline
p &
\frac {a_{np}}{ b_n}  &  \frac{ b_{np}}{ a_n}\mod p^2 
&
(\frac {a_{np}}{ b_n}/ \frac{ b_{np}}{ a_n})^6\equiv \alpha^3
&
\frac {a_{np}}{ b_n}\frac{ b_{np}}{ a_n}\equiv A_p^2
\\
\hline
5&   3   &  1    &4& -2\cdot6^2   \\
11& 84   &  32   &4& -6\cdot6^2   \\
17& 278  &  243  &4& -2\cdot6^2  \\
23& 335  &  130  &4& -6\cdot18^2   \\
29& 272  &  441  &4& -2\cdot24^2  \\
41& 0    &  0    & &\\
47& 302  &  760  &4& -6\cdot6^2  \\
\hline
\end{array}
$
\caption{Values of $a_{np}/b_n$ and $b_{np}/a_n$ for $p\equiv 2\mod 3$,
for $h_1$ and $h_2$ for $\Gamma_{8^3.6.3.1^3}$, with $\alpha$ as in 
(\ref{eqn:case2_eqn3}), and $A_p$ (experimentally) as in 
Table~\ref{table_coefs2}.}
\label{table_ratios4}
\end{table}

From the data in Table~\ref{table_ratios4},
following the explanation of Section~\ref{ss:case2of_S:aswddata}, the
Atkin Swinnerton-Dyer basis 
of $S_3(\Gamma_{8^3.6.3.1^3})$
when $p\equiv 1\mod 3$ should be $h_1, h_2$,
and when $p\equiv 2\mod 3$, it should consist of forms of the form
$h_1 + \alpha h_2$ with $\alpha^3=4$.


\subsubsection{Atkin Swinnerton-Dyer congruences for
$\Gamma_{24.3.2^3.1^3}$}
\label{SSS:24.3}
Basis of $S_3(\Gamma_{24.3.2^3.1^3})$:
\begin{eqnarray*}
h_1(z) &=&\displaystyle{
\sqrt[3]{
\frac{\eta(2\tau)^{22}\eta(8\tau)^{8}}{\eta( \tau)^{4}\eta(4\tau)^{8}}}
=
q
 + \frac{4}{3} q^2
 - \frac{40}{9} q^3
 - \frac{400}{81} q^4
 + \frac{1454}{243} q^5
 + \frac{1888}{729} q^6
 - \frac{13168}{6561} q^7+ \cdots}
\\
h_2(z) &=&
\displaystyle{
\sqrt[3]{
\frac{\eta(2\tau)^{20}\eta(4\tau)^{2}\eta(8\tau)^{4}}{\eta( \tau)^{8}}}
=
q
 + \frac{8}{3} q^2
 + \frac{8}{9} q^3
 - \frac{32}{81} q^4
 - \frac{82}{243} q^5
 - \frac{5440}{729} q^6
 - \frac{24400}{6561} q^7+ \cdots}
\end{eqnarray*}

\comment{
created in pari with:
hf1=(eta(q)^(-4)*eta(q^2)^(22)*eta(q^4)^(-8)*eta(q^8)^8)^(1/3)*q;
hf2=(eta(q)^(-8)*eta(q^2)^(20)*eta(q^4)^(2)*eta(q^8)^4)^(1/3)*q;

printtex(hf1 + O(q^8))
printtex(hf2 + O(q^8))
}
First few prime coefficients:
$$
\begin{array}{ccccccccccc}
p& 2 & 3 & 5 & 7 & 11 & 13 & 17 & 19\\
a_p 
  & \frac{4}{3}
 & -\frac{40}{9}
 & \frac{1454}{243}
 & -\frac{13168}{6561}
 & \frac{38671144}{4782969}
 & -\frac{2230795138}{129140163}
 & -\frac{418720079278}{31381059609}
 & \frac{30660416258552}{2541865828329}
\\
b_p \rule{0ex}{3ex}
 & \frac{8}{3}
 & \frac{8}{9}
 & -\frac{82}{243}
 & -\frac{24400}{6561}
 & \frac{16345336}{4782969}
 & \frac{1236747902}{129140163}
 & \frac{842483994194}{31381059609}
 & -\frac{34758650729368}{2541865828329}
\end{array}
$$
Note that up to sign these are identical to the coefficients of the
forms given for the $\Gamma_{8^3.6.3.1^3}$ case, and so the 
ASwD basis is expected to be the same as in the 
$\Gamma_{8^3.6.3.1^3}$ case, namely $h_1, h_2$ when
$p\equiv 1\mod 3$ and 
$h_1 + \alpha h_2$ with $\alpha^3=4$ when $p\equiv 2\mod 3$.

\begin{table}
$
\begin{array}{|l|ll|}
\hline
p &\frac {a_{np}}{ a_n} & \frac  {b_{np}}{ b_n}  \\
\hline
 7  & 17 & 29        \\ 
 13 & 52 & 130       \\ 
 19 & 48 & 346       \\ 
 31 & 915& 46        \\ 
 37 & 165& 1169      \\ 
 43 & 11 & 1838      \\ 
\hline
\end{array}
$
\caption{Values of $a_{np}/a_n$ and $b_{np}/b_n$ for $p\equiv 1\mod 3$,
for $h_1$ and $h_2$ for 
$S_3(\Gamma_{24.3.2^3.1^3})$.
These values are the same as those in Table~\ref{table_ratios3}.
}
\label{table_ratios5}
\end{table}

\begin{table}
$
\begin{array}{|l|ll|}
\hline
p &
\frac {a_{np}}{ b_n}  &  \frac{ b_{np}}{ a_n}\mod p^2 
\\
\hline
5        & 3   & 1     \\ 
 11      & 84  & 32    \\ 
 17      & 278 & 243   \\ 
 23      & 335 & 130   \\ 
 29      & 272 & 441   \\ 
 41      & 0   &   0   \\ 
 47      & 302 & 760   \\ 
\hline
\end{array}
$
\caption{Values of $a_{np}/b_n$ and $b_{np}/a_n$ for $p\equiv 2\mod 3$,
for $h_1$ and $h_2$ for 
$S_3(\Gamma_{24.3.2^3.1^3})$.
These values are the same as those in Table~\ref{table_ratios4}.
}
\label{table_ratios6}
\end{table}

\subsubsection{Atkin Swinnerton-Dyer congruences for
$\Gamma_{24.3.2^3.1^3 B}$}
\label{SSS:836}
This is a conjugate of the $S_3(\Gamma_{24.3.2^3.1^3})$ example
by the involution
\[
W_8 = \begin{pmatrix}
       0 & -1\\8 & 0
      \end{pmatrix}. 
\]

Basis of $S_3(\Gamma_{24.3.2^3.1^3 B})$
in terms of $r=q^{1/3}$.
\begin{eqnarray*}
h_1(z)
&=&
\displaystyle{
\sqrt[3]{\frac{\eta(z)^{8}\eta(4z)^{22}}{\eta(8z)^{4}\eta(8z)^8}}}
=
\sum_{n\ge 1}a_nr^n
=r^2
 - \frac{8}{3} r^5
 + \frac{20}{9} r^8
 - \frac{256}{81} r^{11}
 - \frac{64}{243} r^{14}
 + \cdots
\\
h_2(z)
&=&
\displaystyle{
\sqrt[3]{\frac{\eta(z)^{4}\eta(2z)^{2}\eta(4z)^{20}}{\eta(8z)^8}}}
=
\sum_{n\ge 1}b_nr^n
=
r
 - \frac{4}{3} r^4
 - \frac{16}{9} r^7
 + \frac{112}{81} r^{10}
+ \ldots
\end{eqnarray*}

\comment{
created in pari with:
hi1=(eta(r^12)^22*eta(r^3)^8/eta(r^6)^8/eta(r^24)^4)^(1/3)*r^2;
hi2=(eta(r^3)^4*eta(r^6)^2*eta(r^12)^20/eta(r^24)^8)^(1/3)*r;
}

First few prime coefficients:
$$
\begin{array}{cccccccccccc}
p& 
2 & 3 
& 5 & 7 & 11 & 13 & 17 & 19 & 23 & 29 & 31
\\
a_p 
  & 1
 & 0
 & -\frac{8}{3}
 & 0
 & -\frac{256}{81}
 & 0
 & \frac{7984}{729}
 & 0
 & \frac{172544}{19683}
 & -\frac{18907736}{1594323}
 & 0
\\
b_p \rule{0ex}{3ex}
 & 0
 & 0
 & 0
 & -\frac{16}{9}
 & 0
 & -\frac{1534}{243}
 & 0
 & \frac{78560}{6561}
 & 0
 & 0
 & -\frac{126424784}{4782969}
\end{array}
$$
Ratios when terms are non-zero:
\begin{table}
$
\hspace{-1.6in}
{\small
\begin{array}{|r|rr|rr|rr|rr|rrr|}
\hline
p &
\multicolumn{2}{|c|}{{a_{np}}/{ a_n}}&
\multicolumn{2}{|c|}{   {b_{np}}/{ b_n}}&
\multicolumn{2}{|c|}{ {a_{np}}/{ b_n}}&  
\multicolumn{2}{|c|}{{ b_{np}}/{ a_n}}
&
\omega
&
i
&
\sqrt[3]{2}
\\
\hline
7 & \multicolumn{2}{|r|}{32 =-\sqrt{-3}\cdot\omega^2}
  & 20 &=\sqrt{-3}\cdot\omega  &   &  &   &  & 18  & \sqrt{-3}=12 &\\
13 & 52 &=-13\cdot\omega  & 130 &=-13\cdot\omega^2  &   &  &   &  & 22  & \sqrt{-3}=45 &  \\
19 & 313 &=11\sqrt{-3}\cdot\omega  & \multicolumn{2}{|r|}{15 =-11\sqrt{-3}\cdot\omega^2}
  &   &  &   &  & 68 & \sqrt{-3}=137 &  \\
31 & 46 &=24\sqrt{-3}  & 915 &=-24\sqrt{-3}  &   &  &   &  & 439 & \sqrt{-3}= 82 & \\
37 & 165 &=35\cdot\omega^2  & 1169 &=35\cdot\omega  &   &  &   &  & 581& &    \\
43 & \multicolumn{2}{|r|}{1838 =24\sqrt{-3}} & 11 &=-24\sqrt{-3}  &   &  &   &  & 423 &\sqrt{-3}=847 &  \\
\hline
\end{array}
}
$
\caption{Values of $a_{np}/a_n$ and $b_{np}/b_n$ for $p\equiv 1\mod 3$,
for $h_1$ and $h_2$ for 
$S_3(\Gamma_{24.3.2^3.1^3 B})$.}
\label{table_ratios7}
\end{table}

\begin{table}
$
\hspace{-1.6in}
{\small
\begin{array}{|r|rr|rr|rr|rr|rrr|}
\hline
p &
\multicolumn{2}{|c|}{{a_{np}}/{ a_n}}&
\multicolumn{2}{|c|}{   {b_{np}}/{ b_n}}&
\multicolumn{2}{|c|}{ {a_{np}}/{ b_n}}&  
\multicolumn{2}{|c|}{{ b_{np}}/{ a_n}}
&
\omega
&
i
&
\sqrt[3]{2}
\\
\hline
5 &&&&& 14 & =6\sqrt{-2}\cdot \frac{\sqrt{2}}{2\sqrt[3]2}  & 2 & =6\sqrt{-2}
\cdot\frac{2\sqrt[3]2}{\sqrt{2}}
 &  &i=7   &  3\\
7 & \multicolumn{2}{|r|}{32 =-\sqrt{-3}\cdot\omega^2}
  & 20 &=\sqrt{-3}\cdot\omega  &   &  &   &  & 18  & \sqrt{-3}=12 &\\
11 &   & &   & & 79 & =6\sqrt{-6}\cdot \frac{\sqrt{-2}}{2\sqrt[3]2}  & 57 & ={6\sqrt{-6}}\cdot
\frac{2\sqrt[3]2}{\sqrt{-2}}
  &    & \sqrt{3}=27 & 73\\
13 & 52 &=-13\cdot\omega  & 130 &=-13\cdot\omega^2  &   &  &   &  & 22  & \sqrt{-3}=45 &  \\
17 &   & &   & & 139 & =6\sqrt{-2}\cdot \frac{\sqrt{2}}{2\sqrt[3]2}  & 197 & ={6\sqrt{-2}}
\cdot\frac{2\sqrt[3]2}{\sqrt{2}}&&i=38 &  195\\
19 & 313 &=11\sqrt{-3}\cdot\omega  & \multicolumn{2}{|r|}{15 =-11\sqrt{-3}\cdot\omega^2}
  &   &  &   &  & 68 & \sqrt{-3}=137 &  \\
23 &   & &   & & 97 & =-18\sqrt{-6}\cdot \frac{\sqrt{-2}}{2\sqrt[3]{2}}  & 269
 & = {-18\sqrt{-6}}\cdot\frac{2\sqrt[3]{2}}{\sqrt{-2}}  &  
 &\sqrt{3}=223      & 384\\
29 &   & &   & & 136 & =-24\sqrt{-2}\cdot \frac{\sqrt{2}}{2\sqrt[3]2}  & 41 & ={-24\sqrt{-2}}
\cdot\frac{2\sqrt[3]2}{\sqrt{2}}
  && i=800 & 403\\
31 & 46 &=24\sqrt{-3}  & 915 &=-24\sqrt{-3}  &   &  &   &  & 439 & \sqrt{-3}= 82 & \\
37 & 165 &=35\cdot\omega^2  & 1169 &=35\cdot\omega  &   &  &   &  & 581& &    \\
41 &   & &   & & 0 &   & 0 &  &   &  &\\
43 & \multicolumn{2}{|r|}{1838 =24\sqrt{-3}} & 11 &=-24\sqrt{-3}  &   &  &   &  & 423 &\sqrt{-3}=847 &  \\
47 &   & &   & & 2058 & =6\sqrt{-6}\cdot \frac{\sqrt{-2}}{2\sqrt[3]{2}}  & 689 & 
6\sqrt{-6}\cdot\frac{2\sqrt[3]{2}}{\sqrt{-2}}  &   &\sqrt{3}=270    & 1854 \\
\hline
\end{array}
}
$
\caption{Values of $a_{np}/b_n$ and $b_{np}/a_n$ for $p\equiv 2\mod 3$,
for $h_1$ and $h_2$ for 
$S_3(\Gamma_{24.3.2^3.1^3 B})$.}
\label{table_ratios8}
\end{table}

Atkin Swinnerton-Dyer basis:
\begin{eqnarray*}
\text{if } p\equiv 1\mod 3   &\text{ basis is }& h_1, h_2\\
\text{if }p\equiv 5\mod 12  &\text{ basis is }& h_1\pm \frac{\sqrt{2}}{2\sqrt[3]{2}}h_2\\
\text{if }p\equiv 11\mod 12 &\text{ basis is }& h_1\pm \frac{\sqrt{-2}}{2\sqrt[3]{2}}h_2\\
\end{eqnarray*}


\subsection{Examples associated with 
newform in $S_3 (\Gamma _0 (243), \chi)$}
$$f(z) =
q + 3iq^2 - 5q^4 + 6iq^5 + 11q^7 - 3iq^8 - 18q^{10} + 12iq^{11}+\cdots
$$
where $i$ is a root of $x^2+1=0$.
Note, the corresponding Galois representation
is a twist of the representation corresponding  to
$E_6(3r^3)$.

The first few prime coefficients $\tilde A_p$ of this form are as follows:

$$
\begin{array}{c|cccccccccccccc}
p  & 5 & 7 & 11 & 13 & 17 & 19 & 23 & 29 & 31 & 37\\
\hline 
\tilde A_p&    6i&    11&    12i&    5&    -18i&    -19&    -30i&    48i&    -13&    17\\
\end{array}
$$

\subsubsection{Atkin Swinnerton-Dyer congruences for
$\Gamma_{18.6.3^3.1^3}$}
\label{SSS:186}
Basis of $S_3(\Gamma_{18.6.3^3.1^3})$

\begin{eqnarray*}
h_1(z)
&=&
\displaystyle{
\sqrt[3]{\frac{\eta(z)^{4}\eta(2z)^7\eta(6z)^{11}}{\eta(3z)^4}}}
=
\sum_{n\ge 1}a_nq^n
=
q - \frac 4 3q^2 - \frac{31}{9}q^3 + \frac{400}{81}q^4 
+ \frac{104}{243}q^5 + \cdots
\\
h_2(z)
&=&
\displaystyle{
\sqrt[3]{\frac{\eta(3z)^{4}\eta(6z)^7\eta(2z)^{11}}{\eta(z)^4}}}
=
\sum_{n\ge 1}b_nq^n
=
 q + \frac 4 3q^2 - \frac 7 9q^3 - \frac {112} {81}q^4 
- \frac{616}{243}q^5 + \ldots
\end{eqnarray*}

\comment{
created in pari with:

hc1=(eta(q)^4*eta(q^2)^7*eta(q^6)^11/eta(q^3)^4)^(1/3)*q;
hc2=(eta(q^3)^4*eta(q^6)^7*eta(q^2)^11/eta(q)^4)^(1/3)*q;

}

First few prime coefficients:
$$
\begin{array}{ccccccccccc}
p& 2 & 3 & 5 & 7 & 11 & 13 & 17 & 19\\
a_p & -\frac{4}{3}
 & -\frac{31}{9}
 & \frac{104}{243}
 & \frac{44018}{6561}
 & -\frac{38654696}{4782969}
 & -\frac{1857609346}{129140163}
 & \frac{362933655200}{31381059609}
 & -\frac{33243449873158}{2541865828329}\\
b_p \rule{0ex}{3ex}& \frac{4}{3}
 & -\frac{7}{9}
 & -\frac{616}{243}
 & -\frac{15886}{6561}
 & \frac{43656424}{4782969}
 & -\frac{343807618}{129140163}
 & -\frac{100695940768}{31381059609}
 & \frac{19258418018042}{2541865828329}
\end{array}
$$

with $a_n$ and $b_n$ the coefficients of the non-congruence forms given
above. 
The following ratios, all computed mod $p^2$, appear to be constant as $n$ varies, for the given $p$s.
The table shows the constants; if no entry is shown, this means the 
ratio is not constant in this case.
$$
\begin{array}{|l|llll|}
\hline
p &\frac {a_{np}}{ a_n}
  & \frac  {b_{np}}{ b_n}   & \frac {a_{np}}{ b_n}  &  \frac{ b_{np}}{ a_n}\mod p^2 \\
\hline
5  & 	  &        &   3  &  13         \\
7  &    36&      2 &      &       	\\
11 &      &        &   13 &        82  	\\
13 &   54 &     110&   	  &	        \\
17 &      &        &   279&  148  	\\
19 &   228&     152&   	  &	        \\
23 &      &        &   130&  400  	\\
29 &      &        &   296&  515  	\\
31 &   915&     59 &      &             \\                                          
37 &   1058&   294 &	  &             \\
\hline
\end{array}
$$

\subparagraph{Case I: $p\equiv 1\mod 3$}
\label{SP:case1}
These ratios are a special case of the Atkin-Swinnerton-Dyer type relation,
e.g., $a_{7n}/a_n\equiv 36\mod 7^2$ can be written as
$$a_{7n} - 36 a_n + 7^2 a_{n/p}\equiv 0\mod 7^2. $$
So, for $p\equiv 1\mod3$, it looks like
$h_1$ and $h_2$ form an Atkin Swinnerton-Dyer basis.

Note that for $p$ in the above table with $p\equiv 1\mod p$, except 
for the case $p=19$,
we have $(\frac {a_{np}}{ a_n}/\frac  {b_{np}}{ b_n})^3\equiv 1\mod p^2$.

It's not surprising that this relation holds,
since the ratios ought to be the
values of $A_p$ given above, which we can see should always be
$\omega$ or $\omega^2$ in these cases, including for $p=19$.

The reason the congruence does not hold for $p=19$ is that
in this case we have $\omega,\omega^2\equiv 68,292\mod 19^2$, and
$\alpha_1=-19\omega, \alpha_2=-19\omega^2\equiv 152,228\mod 19^2$, so we only have that
$\alpha_1/19\equiv \omega\mod 19$,
$\alpha_2/19\equiv \omega^2\mod 19$,
i.e., the ratio satisfies 
$(\frac {a_{19n}}{ a_n}/\frac  {b_{19n}}{ b_n})^3\equiv 1\mod 19$,
which we can check is true.

\subparagraph{Case II: $p\equiv 2\mod 3$}
\label{SP:case2}
Observation: when $p\equiv 2\mod 3$ we always have
$(\frac{a_{np}}{ b_n}/\frac{ b_{np}}{ a_n})^3 \equiv -9\mod p^2$.

\comment{
\\ check that all these have value 9:
(-   3/13 )^3
(-  13/82 )^3
(- 279/148)^3
(- 130/400)^3
(- 296/515)^3

}

Suppose that the Atkin Swinnerton-Dyer basis is $h_1 + \alpha h_2$, 
then (writing $\alpha_p=\alpha\mod p^2$) we would have
$$a_{pn} + \alpha_p b_{pn} \equiv A_p (a_{n} + \alpha_p b_{n})\mod p^2,$$
and suppose we in fact have
$$a_{pn} \equiv  A_p\alpha_p b_{n}\mod p^2,
\text{ and }\alpha_p b_{pn} \equiv A_p a_{n} \mod p^2,$$
then this implies that
$a_{pn}/b_n = A_p\alpha_p$ and $b_{pn}/a_n = A_p/\alpha_p$,
so $\alpha_p^2 \equiv 
\frac {a_{np}}{ b_n}/\frac{ b_{np}}{ a_n}$,
so from the above observation we expect
$\alpha^6\equiv -9\mod p^2$,
i.e., 
$\alpha\equiv \sqrt[3]{3}i\mod p^2$, so it seems that
for $p\equiv 2\mod 3$ we should have Atkin Swinnerton-Dyer basis
consisting of forms of the form
$h_1 + \alpha h_2$, where $\alpha^6=-9$.

The value of $A_p$ is given by $A_p\equiv
\pm\sqrt{\frac {a_{np}}{ b_n}\frac{ b_{np}}{ a_n}}\mod p^2$, whereas the values for 
$p\equiv 1\mod 3$ are those already in the table above.
From the values in the above table, we compute the following table of $A_p$s,
with no particular order given to the two possible values.
In this table, we write e.g., $A_p\equiv 6i\mod 25$ to mean that
$A_p^2\equiv -36\mod 25$, etc, and $\omega$ means
$\omega^2 + \omega + 1\equiv 0\mod p^2$.
$$
\begin{array}{c|cccccccccccccc}
p  & 5 & 7 & 11 & 13 & 17 & 19 & 23 & 29 & 31 & 37\\
\hline
A_p&6i &11\omega &12i &5\omega &18i &-19\omega&30i &48i &-13\omega&17\omega\\
\mod p^2&-6i&11\omega^2  &-12i&5\omega^2&-18i&-19\omega^2&-30i&-48i&-13\omega^2 &17\omega^2\\
\end{array}
$$


\subsubsection{Atkin Swinnerton-Dyer congruences for
$\Gamma_{9.6^3.3.2^3}$}
\label{SS:96}
 
Basis of $S_3(\Gamma_{9.6^3.3.2^3})$ in terms of $r=q^{1/3}$.
\begin{eqnarray*}
h_1(z)
&=&
\displaystyle{
\sqrt[3]{
\frac{
\eta( \tau)^{7}\eta(2\tau)^{4}\eta(3\tau)^{11}}{\eta(6\tau)^{4}}}
=
\sum_{n\ge 1}a_nr^n
=
r
 - \frac{7}{3} r^4
 - \frac{19}{9} r^7
 + \frac{193}{81} r^{10}
 + \frac{2306}{243} r^{13}+ \cdots}\\
\\
h_2(z)
&=&
\displaystyle{
\sqrt[3]{
\frac{\eta( \tau)^{11}\eta(3\tau)^{7}\eta(6\tau)^{4}}{\eta(2\tau)^{4}}}
=
\sum_{n\ge 1}b_nr^n
= 
r^2
 - \frac{11}{3} r^5
 + \frac{23}{9} r^8
 - \frac{13}{81} r^{11}
+\cdots
}
\end{eqnarray*}

\comment{
created in pari with:

\\ r=q^{1/3}, s=q^{2/3}

hg1=(eta(r^3)^7*eta(r^6)^4*eta(r^9)^11/eta(r^18)^4)^(1/3)*r;
hg2=(eta(r^3)^11*eta(r^9)^7*eta(r^18)^4/eta(r^6)^4)^(1/3)*r^2;

printtex(hg1 + O(r^14))
printtex(hg2 + O(r^14))

}

First few prime coefficients:
\newline
$
\begin{array}{|ccccccccc|}
\hline
p& 2 & 3 & 5 & 7 & 11 & 13 & 17 & 19\\
a_p  & 0
 & 0
 & 0
 & -\frac{19}{9}
 & 0
 & \frac{2306}{243}
 & 0
 & -\frac{151696}{6561}
\\
b_p \rule{0ex}{3ex}
 & 1
 & 0
 & -\frac{11}{3}
 & 0
 & -\frac{13}{81}
 & 0
 & -\frac{7130}{729}
 & 0\\
\hline
\end{array}
$

First few prime coefficients mod $p^2$.
Notice that these are either zero or the same as
in the $\Gamma_{18.6.3^3.1^3}$ case.

Ratios of coefficients, (when all terms are non-zero), all numbers
given mod $p^2$.  When 
$p\equiv 2\mod 3$, there is a unique cube root
mod $p^2$ of any integer, so the given value of $\sqrt[3]{3}$
is unique.  $i$ means the square root of $-1$.
$$
{\small
\begin{array}{|r|rr|rr|rr|rr|rrr|}
\hline
p &
\multicolumn{2}{|c|}{{a_{np}}/{ a_n}}&
\multicolumn{2}{|c|}{   {b_{np}}/{ b_n}}&
\multicolumn{2}{|c|}{ {a_{np}}/{ b_n}}&  
\multicolumn{2}{|c|}{{ b_{np}}/{ a_n}}
&
\omega
&
\omega^2
&
\sqrt[3]{3}
\\
\hline
5 &   & &   & & 3 & =6i\cdot i\sqrt[3]{3}  & 13 & =6i/ i\sqrt[3]{3}  &   &   & 12 \\
7 & 36 &=11\cdot\omega^2  & 2 &=11\cdot\omega  &   &  &   &  & 18 & 30 &   \\
11 &   & &   & & 13 & =12i\cdot i\sqrt[3]{3}  & 82 & =12i/i\sqrt[3]{3}  &   &   & 9 \\
13 & 54 &=5\cdot\omega^2  & 110 &=5\cdot\omega  &   &  &   &  & 22 & 146 &   \\
17 &   & &   & & 279 & =-18i\cdot i\sqrt[3]{3}  & 148 & =-18i/i\sqrt[3]{3}  &   &   & 160 \\
19 & 228 &=-19\cdot\omega^2  & 152 &=-19\cdot\omega  &   &  &   &  & 68 & 292 &   \\
23 &   & &   & & 130 & =-30i\cdot i\sqrt[3]{3}  & 400 & =-30i/i\sqrt[3]{3}  &   &   & 357 \\
29 &   & &   & & 296 & =48i\cdot \sqrt[3]{3}  & 515 & =48i/i\sqrt[3]{3}  &   &   & 134 \\
31 & 915 &=-13\cdot\omega^2  & 59 &=-13\cdot\omega  &   &  &   &  & 439 & 521 &   \\
37 & 1058 &=17\cdot\omega^2  & 294 &=17\cdot\omega  &   &  &   &  & 581 & 787 &   \\
41 &   & &   & & 1384 & =-30i\cdot i\sqrt[3]{3}  & 869 & =-30i/i\sqrt[3]{3}  &   &   & 1503 \\
43 & 1173 &=29\cdot\omega^2  & 647 &=29\cdot\omega  &   &  &   &  & 1425 & 423 &   \\
47 &   & &   & & 155 & =-24i\cdot i\sqrt[3]{3}  & 1906 & =-24i/i\sqrt[3]{3}  &   &   & 1203 \\
\hline
\end{array}
}
$$

The above table indicates that when $p\equiv1\mod 3$, we have
$$a_{np}-A_p\omega^2a_n\equiv 0\mod p^2
\;\;\;\;\;
\text{ and }
\;\;\;\;\;
b_{np}-A_p\omega b_n\equiv 0\mod p^2$$
for certain $A_p$, indicating $h_1, h_2$ is an ASWD-basis in this case.

Note that this relation only hold when terms are non zero.
E.g., $b_{1}=0$, so we can't have $b_{p}+A_pb_{1}\equiv 0\mod p$ for any $p$ with
$b_p\not=0$.

For $p\equiv 2\mod 3$, the above table indicates that we have
\begin{align*}
\left(a_{np} + i\sqrt[3]{3}b_{np}\right)
+iA_p\left(a_{n} + i\sqrt[3]{3}b_{n}\right)
&\equiv 0\mod p^2\\
\left(a_{np} - i\sqrt[3]{3}b_{np}\right)
-iA_p\left(a_{n} - i\sqrt[3]{3}b_{n}\right)
&\equiv 0\mod p^2,
\end{align*}
so $h_1+i\sqrt[3]{3}h_2$, $h_1-i\sqrt[3]{3}h_2$ 
should be the ASWD-basis in this case.
(shouldn't make any difference which cube root of three is taken)


\subsection{Examples associated with newform in $S_3(\Gamma _0 (48), \chi).$}
\label{SS:gam48}
\begin{eqnarray*}
f(z)&=&
q - \sqrt{-2}q^{2} - 2q^{4} + 3\sqrt{-2}q^{5} - 7q^{7} + 2\sqrt{-2}q^{8} +
 6q^{10} - 3\sqrt{-2}q^{11} + 5q^{13}\\
&& + 7\sqrt{-2}q^{14} + 4q^{16} - 18\sqrt{-2}q^{17} + 
  17q^{19} - 6\sqrt{-2}q^{20} - 6q^{22} - 6\sqrt{-2}q^{23}\\
&& + 7q^{25} - 5\sqrt{-2}q^{26} 
+ 14q^{28} - 39\sqrt{-2}q^{29} + 59q^{31} - 4\sqrt{-2}q^{32} - 
    36q^{34} +\cdots
\end{eqnarray*}

First few coefficients $a_p$,
First few prime coefficients, divided by either $1$
or 3$\sqrt{-2}$, for easy readability
$$
\begin{array}{|c|ccccccccccccccccc|}
\hline
p & 5 & 7 & 11 & 13 & 17 & 19 & 23 & 29 & 31 & 37 & 41 & 43 & 47 & 53 & 59 & 61 & 67 \\ 
\hline
a_p & &-7& &5& &17& & &59&-19& &47& & & &-4&-46 \\
\frac{a_p}{3\sqrt{-2}}&1& &-1& &-6& &-2&-13& & &13& &-19&9&-5& &\\
\hline
\end{array}
$$

\subsubsection{Atkin Swinnerton-Dyer congruences for
$\Gamma_{9.6^4.1^3}$}
\label{SS:961}

Basis of $S_3(\Gamma_{9.6^4.1^3})$
\begin{eqnarray*}
h_1(z)
&=&
\displaystyle{
\sqrt[3]{\frac{\eta(z)^{13}\eta(6z)^{14}}{\eta(2z)^{2}\eta(3z)^7}}}
=
\sum_{n\ge 1}a_nq^n
=
q - \frac {13} 3q^2 + \frac{32}{9}q^3 + \frac{670}{81}q^4 
- \frac{3577}{243}q^5 + \cdots
\\
h_2(z)
&=&
\displaystyle{
\sqrt[3]{\frac{\eta(z)^{14}\eta(6z)^{13}}{\eta(2z)^{7}\eta(3z)^2}}}
=
\sum_{n\ge 1}b_nq^n
=
 q - \frac {14} 3q^2 + \frac {56} 9q^3 - \frac {58} {81}q^4 
+ \frac{266}{243}q^5 + \ldots
\end{eqnarray*}

\comment{
created in pari with:

hd1=(eta(q)^13*eta(q^6)^14/eta(q^2)^2/eta(q^3)^7)^(1/3)*q;
hd2=(eta(q)^14*eta(q^6)^13/eta(q^2)^7/eta(q^3)^2)^(1/3)*q;

}

First few prime coefficients:
$$
\begin{array}{ccccccccccc}
p& 2 & 3 & 5 & 7 & 11 & 13 & 17 & 19\\
a_p 
 & -\frac{13}{3}
 & \frac{32}{9}
 & -\frac{3577}{243}
 & \frac{38780}{6561}
 & \frac{97488844}{4782969}
 & -\frac{198000616}{129140163}
 & \frac{1030071452831}{31381059609}
 & -\frac{91038813695632}{2541865828329}
\\
b_p \rule{0ex}{3ex}

& -\frac{14}{3}
 & \frac{56}{9}
 & \frac{266}{243}
 & -\frac{1036}{6561}
 & \frac{24235144}{4782969}
 & -\frac{2216727472}{129140163}
 & -\frac{894269035558}{31381059609}
 & \frac{97467805305080}{2541865828329}
\end{array}
$$

$$
\begin{array}{|l|llll|crcr|}
\hline
p &\frac {a_{np}}{ a_n} & \frac  {b_{np}}{ b_n}   & 
\frac {a_{np}}{ b_n}  &  \frac{ b_{np}}{ a_n}\mod p^2 
&(\frac {a_{np}}{ a_n}/\frac  {b_{np}}{ b_n})^3
&\frac {a_{np}}{ a_n}\frac  {b_{np}}{ b_n}&
(\frac {a_{np}}{ b_n}/ \frac{ b_{np}}{ a_n})^6
&
\frac {a_{np}}{ b_n}\frac{ b_{np}}{ a_n}
\\
\hline
5 &   &   & 11 & 12 &   &   & 4 & -18\\ 
 7 & 35 & 21 &   &   & 1 & 0 &   &  \\ 
 11 &   &   & 94 & 41 &   &   & 75 & -18\\ 
 13 & 54 & 110 &   &   & 1 & 5^2 &   &  \\ 
 17 &   &   & 10 & 282 &   &   & 69 & -18\cdot6^2\\ 
 19 & 271 & 73 &   &   & 1 & 17^2 &   &  \\ 
 23 &   &   & 503 & 369 &   &   & 522 & -18\cdot2^2\\ 
 29 &   &   & 661 & 101 &   &   & 724 & -18\cdot13^2\\ 
 31 & 948 & 915 &   &   & 1 & 59^2 &   &  \\ 
 37 & 106 & 1282 &   &   & 1 & 19^2 &   &  \\ 
 41 &   &   & 1463 & 1587 &   &   & 1656 & -18\cdot13^2\\ 
 43 & 1391 & 411 &   &   & 1 & 47^2 &   &  \\ 
 47 &   &   & 2117 & 887 &   &   & 519 & -18\cdot19^2\\ 
\hline
\end{array}
$$


\subsubsection{Atkin Swinnerton-Dyer congruences for
$\Gamma_{18.3^4.2^3}$}
\label{SS:183}

Basis of $S_3(\Gamma_{18.3^4.2^3})$,
in terms of $r=q^{1/3}$:
\begin{eqnarray*}
h_1(z)
&=&
\displaystyle{
\sqrt[3]{
\frac{\eta(2\tau)^{13}\eta(3\tau)^{14}}{\eta(6\tau)^{7}\eta( \tau)^{2}}}
=
\sum_{n\ge 1}a_nr^n
= r + \frac{2}{3}r^4 - \frac{28}{9}r^7
- \frac{482}{81}r^{10} - \frac{736}{243}r^{13}
+\cdots
}
\\
h_2(z)
&=&
\displaystyle{
\sqrt[3]{
\frac{\eta(2\tau)^{14}\eta(3\tau)^{13}}{\eta(6\tau)^{2}\eta( \tau)^{7}}} 
=
\sum_{n\ge 1}b_nq^n
= r^2 + 
\frac{7}{3}r^5 + \frac{14}{9}r^8
- \frac{148}{81}r^{11}
- \frac{1708}{243}r^{14}
+\cdots}
\end{eqnarray*}

\comment{
created in pari with:

\\ use r=q^{1/3}

hh1=(eta(r^6)^13*eta(r^9)^14/eta(r^3)^2/eta(r^18)^7)^(1/3)*r;
hh2=(eta(r^6)^14*eta(r^9)^13/eta(r^3)^7/eta(r^18)^2)^(1/3)*r^2;

}

First few prime coefficients:
$$
\begin{array}{ccccccccccc}
p& 2 & 3 & 5 & 7 & 11 & 13 & 17 & 19\\
a_p 
 & 0
 & 0
 & 0
 & -\frac{28}{9}
 & 0
 & -\frac{736}{243}
 & 0
 & \frac{120680}{6561}
\\
b_p \rule{0ex}{3ex}
  & 1
 & 0
 & \frac{7}{3}
 & 0
 & -\frac{148}{81}
 & 0
 & -\frac{4529}{729}
 & 0
\end{array}
$$

$$
{\small
\begin{array}{|r|rr|rr|rr|rr|rrr|}
\hline
p &
\multicolumn{2}{|c|}{{a_{np}}/{ a_n}}&
\multicolumn{2}{|c|}{   {b_{np}}/{ b_n}}&
\multicolumn{2}{|c|}{ {a_{np}}/{ b_n}}&  
\multicolumn{2}{|c|}{{ b_{np}}/{ a_n}}
&
\omega
&
\omega^2
&
\sqrt[3]{3}
\\
\hline
5 &   & &   & & 3 & =-1\cdot6\sqrt[3]{3}  & 19 & =1\cdot3/\sqrt[3]{3}  &   &   & 12\\
7 & 35 &=-7\cdot\omega^2  & 21 &=-7\cdot\omega  &   &  &   &  & 18 & 30 &  \\
11 &   & &   & & 54 & =1\cdot6\sqrt[3]{3}  & 40 & =-1\cdot3/\sqrt[3]{3}  &   &   & 9\\
13 & 54 &=5\cdot\omega^2  & 110 &=5\cdot\omega  &   &  &   &  & 22 & 146 &  \\
17 &   & &   & & 269 & =6\cdot6\sqrt[3]{3}  & 148 & =-6\cdot3/\sqrt[3]{3}  &   &   & 160\\
19 & 271 &=17\cdot\omega^2  & 73 &=17\cdot\omega  &   &  &   &  & 68 & 292 &  \\
23 &   & &   & & 52 & =2\cdot6\sqrt[3]{3}  & 80 & =-2\cdot3/\sqrt[3]{3}  &   &   & 357\\
29 &   & &   & & 360 & =13\cdot6\sqrt[3]{3}  & 370 & =-13\cdot3/\sqrt[3]{3}  &   &   & 134\\
31 & 948 &=59\cdot\omega^2  & 915 &=59\cdot\omega  &   &  &   &  & 439 & 521 &  \\
37 & 106 &=-19\cdot\omega^2  & 1282 &=-19\cdot\omega  &   &  &   &  & 581 & 787 &  \\
41 &   & &   & & 436 & =-13\cdot6\sqrt[3]{3}  & 47 & =13\cdot3/\sqrt[3]{3}  &   &   & 1503\\
43 & 1391 &=47\cdot\omega^2  & 411 &=47\cdot\omega  &   &  &   &  & 1425 & 423 &  \\
47 &   & &   & & 184 & =19\cdot6\sqrt[3]{3}  & 661 & =-19\cdot3/\sqrt[3]{3}  &   &   & 1203\\
\hline
\end{array}
}
$$

When $p\equiv 1\mod 3$, we see the ASWD-basis should be $h_1, h_2$.

For $p\equiv 2\mod 3$, the congruences (which only hold when all terms are non-zero)
$$a_{np}/b_p\equiv-\alpha_p\cdot 6\sqrt[3]{3}\;\;\;\;\;\text{ and } \;\;\;\;\;
b_{np}/a_p\equiv\alpha_p\cdot 3/\sqrt[3]{3}
$$
should be rewritten in terms of $u$, where $u^2=-2$, writing $-6=3u\cdot u$, so we have
$$a_{np}/b_p\equiv\alpha_p3u\cdot u\sqrt[3]{3}\;\;\;\;\;\text{ and } \;\;\;\;\;
b_{np}/a_p\equiv\alpha_p3u/u\sqrt[3]{3}.
$$
These imply that 
$a_{np}\equiv\alpha_p3u\cdot u\sqrt[3]{3} b_p\;\;\;\;\;\text{ and } \;\;\;\;\;
u\sqrt[3]{3}b_{np}\equiv\alpha_p3ua_p.
$
so
$$a_{np} + u\sqrt[3]{3}b_{np}
\equiv\alpha_p3u(\cdot u\sqrt[3]{3} b_p+a_p),
$$
which holds for $u$ replaced with $-u$, so the ASWD-basis should be
$h_1 \pm \sqrt{-2}\sqrt[2]{3} h_2$.


\end{document}

%% file: beauvillecurves.pstex_t
\begin{picture}(0,0)%
\includegraphics{beauvillecurves.pstex}%
\end{picture}%
\setlength{\unitlength}{1865sp}%
\begingroup\makeatletter\ifx\SetFigFont\undefined%
\gdef\SetFigFont#1#2#3#4#5{%
  \reset@font\fontsize{#1}{#2pt}%
  \fontfamily{#3}\fontseries{#4}\fontshape{#5}%
  \selectfont}%
\fi\endgroup%
\begin{picture}(12849,12102)(214,-7207)
\put(7471,4529){\makebox(0,0)[lb]{\smash{{\SetFigFont{7}{8.4}{\rmdefault}{\mddefault}{\updefault}{\color[rgb]{0,0,0}$\Gamma_1(4)\cap\Gamma_0(8)$}%
}}}}
\put(9091,3989){\makebox(0,0)[lb]{\smash{{\SetFigFont{7}{8.4}{\rmdefault}{\mddefault}{\updefault}{\color[rgb]{0,0,0}$\abcd{1}{1}{0}{1}$}%
}}}}
\put(8236,2999){\makebox(0,0)[lb]{\smash{{\SetFigFont{7}{8.4}{\rmdefault}{\mddefault}{\updefault}{\color[rgb]{0,0,0}$\abcd{1}{0}{8}{1}$}%
}}}}
\put(11071,2999){\makebox(0,0)[lb]{\smash{{\SetFigFont{7}{8.4}{\rmdefault}{\mddefault}{\updefault}{\color[rgb]{0,0,0}$\abcd{5}{-2}{8}{-3}$}%
}}}}
\put(12736,1469){\makebox(0,0)[lb]{\smash{{\SetFigFont{7}{8.4}{\rmdefault}{\mddefault}{\updefault}{\color[rgb]{0,0,0}$\frac{3}{4}$}%
}}}}
\put(2311,3959){\makebox(0,0)[lb]{\smash{{\SetFigFont{7}{8.4}{\rmdefault}{\mddefault}{\updefault}{\color[rgb]{0,0,0}$\abcd{1}{1}{0}{1}$}%
}}}}
\put(10036,1469){\makebox(0,0)[lb]{\smash{{\SetFigFont{7}{8.4}{\rmdefault}{\mddefault}{\updefault}{\color[rgb]{0,0,0}$\frac{1}{4}$}%
}}}}
\put(11386,1469){\makebox(0,0)[lb]{\smash{{\SetFigFont{7}{8.4}{\rmdefault}{\mddefault}{\updefault}{\color[rgb]{0,0,0}$\frac{1}{2}$}%
}}}}
\put(11386,2189){\makebox(0,0)[lb]{\smash{{\SetFigFont{7}{8.4}{\rmdefault}{\mddefault}{\updefault}{\color[rgb]{0,0,0}$2$}%
}}}}
\put(8686,2189){\makebox(0,0)[lb]{\smash{{\SetFigFont{7}{8.4}{\rmdefault}{\mddefault}{\updefault}{\color[rgb]{0,0,0}$8$}%
}}}}
\put(2161,1469){\makebox(0,0)[lb]{\smash{{\SetFigFont{7}{8.4}{\rmdefault}{\mddefault}{\updefault}{\color[rgb]{0,0,0}$0$}%
}}}}
\put(5716,1469){\makebox(0,0)[lb]{\smash{{\SetFigFont{7}{8.4}{\rmdefault}{\mddefault}{\updefault}{\color[rgb]{0,0,0}$\frac{2}{3}$}%
}}}}
\put(3961,1469){\makebox(0,0)[lb]{\smash{{\SetFigFont{7}{8.4}{\rmdefault}{\mddefault}{\updefault}{\color[rgb]{0,0,0}$\frac{1}{3}$}%
}}}}
\put(3061,1469){\makebox(0,0)[lb]{\smash{{\SetFigFont{7}{8.4}{\rmdefault}{\mddefault}{\updefault}{\color[rgb]{0,0,0}$\frac{1}{6}$}%
}}}}
\put(361,1469){\makebox(0,0)[lb]{\smash{{\SetFigFont{7}{8.4}{\rmdefault}{\mddefault}{\updefault}{\color[rgb]{0,0,0}$-\frac{1}{3}$}%
}}}}
\put(3916,2144){\makebox(0,0)[lb]{\smash{{\SetFigFont{7}{8.4}{\rmdefault}{\mddefault}{\updefault}{\color[rgb]{0,0,0}$1$}%
}}}}
\put(586,4664){\makebox(0,0)[lb]{\smash{{\SetFigFont{7}{8.4}{\rmdefault}{\mddefault}{\updefault}{\color[rgb]{0,0,0}$\Gamma_0(9)\cap\Gamma_1(3)$}%
}}}}
\put(1531,3134){\makebox(0,0)[lb]{\smash{{\SetFigFont{7}{8.4}{\rmdefault}{\mddefault}{\updefault}{\color[rgb]{0,0,0}$\abcd{1}{0}{9}{1}$}%
}}}}
\put(4366,3044){\makebox(0,0)[lb]{\smash{{\SetFigFont{7}{8.4}{\rmdefault}{\mddefault}{\updefault}{\color[rgb]{0,0,0}$\abcd{-2}{1}{-9}{4}$}%
}}}}
\put(2206,2189){\makebox(0,0)[lb]{\smash{{\SetFigFont{7}{8.4}{\rmdefault}{\mddefault}{\updefault}{\color[rgb]{0,0,0}$9$}%
}}}}
\put(631,434){\makebox(0,0)[lb]{\smash{{\SetFigFont{7}{8.4}{\rmdefault}{\mddefault}{\updefault}{\color[rgb]{0,0,0}$\Gamma(3)$}%
}}}}
\put(8236,-4231){\makebox(0,0)[lb]{\smash{{\SetFigFont{7}{8.4}{\rmdefault}{\mddefault}{\updefault}{\color[rgb]{0,0,0}$\Gamma_1(5)$}%
}}}}
\put(586,-4366){\makebox(0,0)[lb]{\smash{{\SetFigFont{7}{8.4}{\rmdefault}{\mddefault}{\updefault}{\color[rgb]{0,0,0}$\Gamma_1(6)$}%
}}}}
\put(7516,494){\makebox(0,0)[lb]{\smash{{\SetFigFont{7}{8.4}{\rmdefault}{\mddefault}{\updefault}{\color[rgb]{0,0,0}$\Gamma(2)\cap\Gamma_0(4)$}%
}}}}
\put(9091, -1){\makebox(0,0)[lb]{\smash{{\SetFigFont{7}{8.4}{\rmdefault}{\mddefault}{\updefault}{\color[rgb]{0,0,0}$\abcd{1}{2}{0}{1}$}%
}}}}
\put(8416,-1081){\makebox(0,0)[lb]{\smash{{\SetFigFont{7}{8.4}{\rmdefault}{\mddefault}{\updefault}{\color[rgb]{0,0,0}$\abcd{1}{0}{4}{1}$}%
}}}}
\put(11026,-1126){\makebox(0,0)[lb]{\smash{{\SetFigFont{7}{8.4}{\rmdefault}{\mddefault}{\updefault}{\color[rgb]{0,0,0}$\abcd{5}{-4}{4}{-3}$}%
}}}}
\put(8821,-5446){\makebox(0,0)[lb]{\smash{{\SetFigFont{7}{8.4}{\rmdefault}{\mddefault}{\updefault}{\color[rgb]{0,0,0}$\abcd{1}{0}{5}{1}$}%
}}}}
\put(11791,-6031){\makebox(0,0)[lb]{\smash{{\SetFigFont{7}{8.4}{\rmdefault}{\mddefault}{\updefault}{\color[rgb]{0,0,0}$\abcd{11}{-5}{20}{-9}$}%
}}}}
\put(10936,-4186){\makebox(0,0)[lb]{\smash{{\SetFigFont{7}{8.4}{\rmdefault}{\mddefault}{\updefault}{\color[rgb]{0,0,0}$\abcd{1}{1}{0}{1}$}%
}}}}
\put(1576,-5626){\makebox(0,0)[lb]{\smash{{\SetFigFont{7}{8.4}{\rmdefault}{\mddefault}{\updefault}{\color[rgb]{0,0,0}$\abcd{1}{0}{6}{1}$}%
}}}}
\put(12736,-2566){\makebox(0,0)[lb]{\smash{{\SetFigFont{7}{8.4}{\rmdefault}{\mddefault}{\updefault}{\color[rgb]{0,0,0}$\frac{3}{2}$}%
}}}}
\put(10036,-2566){\makebox(0,0)[lb]{\smash{{\SetFigFont{7}{8.4}{\rmdefault}{\mddefault}{\updefault}{\color[rgb]{0,0,0}$\frac{1}{2}$}%
}}}}
\put(11386,-2566){\makebox(0,0)[lb]{\smash{{\SetFigFont{7}{8.4}{\rmdefault}{\mddefault}{\updefault}{\color[rgb]{0,0,0}$1$}%
}}}}
\put(7291,-2566){\makebox(0,0)[lb]{\smash{{\SetFigFont{7}{8.4}{\rmdefault}{\mddefault}{\updefault}{\color[rgb]{0,0,0}$-\frac{1}{2}$}%
}}}}
\put(5716,-7111){\makebox(0,0)[lb]{\smash{{\SetFigFont{7}{8.4}{\rmdefault}{\mddefault}{\updefault}{\color[rgb]{0,0,0}$\frac{2}{3}$}%
}}}}
\put(316,-7111){\makebox(0,0)[lb]{\smash{{\SetFigFont{7}{8.4}{\rmdefault}{\mddefault}{\updefault}{\color[rgb]{0,0,0}$-\frac{1}{3}$}%
}}}}
\put(2161,-7111){\makebox(0,0)[lb]{\smash{{\SetFigFont{7}{8.4}{\rmdefault}{\mddefault}{\updefault}{\color[rgb]{0,0,0}$0$}%
}}}}
\put(4861,-7111){\makebox(0,0)[lb]{\smash{{\SetFigFont{7}{8.4}{\rmdefault}{\mddefault}{\updefault}{\color[rgb]{0,0,0}$\frac{1}{2}$}%
}}}}
\put(3961,-7111){\makebox(0,0)[lb]{\smash{{\SetFigFont{7}{8.4}{\rmdefault}{\mddefault}{\updefault}{\color[rgb]{0,0,0}$\frac{1}{3}$}%
}}}}
\put(7291,-7111){\makebox(0,0)[lb]{\smash{{\SetFigFont{7}{8.4}{\rmdefault}{\mddefault}{\updefault}{\color[rgb]{0,0,0}$-\frac{2}{5}$}%
}}}}
\put(11656,-7111){\makebox(0,0)[lb]{\smash{{\SetFigFont{7}{8.4}{\rmdefault}{\mddefault}{\updefault}{\color[rgb]{0,0,0}$\frac{2}{5}$}%
}}}}
\put(12736,-7111){\makebox(0,0)[lb]{\smash{{\SetFigFont{7}{8.4}{\rmdefault}{\mddefault}{\updefault}{\color[rgb]{0,0,0}$\frac{3}{5}$}%
}}}}
\put(12196,-7111){\makebox(0,0)[lb]{\smash{{\SetFigFont{7}{8.4}{\rmdefault}{\mddefault}{\updefault}{\color[rgb]{0,0,0}$\frac{1}{2}$}%
}}}}
\put(3196,-4501){\makebox(0,0)[lb]{\smash{{\SetFigFont{7}{8.4}{\rmdefault}{\mddefault}{\updefault}{\color[rgb]{0,0,0}$\abcd{1}{1}{0}{1}$}%
}}}}
\put(4186,-5941){\makebox(0,0)[lb]{\smash{{\SetFigFont{7}{8.4}{\rmdefault}{\mddefault}{\updefault}{\color[rgb]{0,0,0}$\abcd{7}{-3}{12}{-5}$}%
}}}}
\put(1486,-1186){\makebox(0,0)[lb]{\smash{{\SetFigFont{7}{8.4}{\rmdefault}{\mddefault}{\updefault}{\color[rgb]{0,0,0}$\abcd{1}{0}{3}{1}$}%
}}}}
\put(3961,-1231){\makebox(0,0)[lb]{\smash{{\SetFigFont{7}{8.4}{\rmdefault}{\mddefault}{\updefault}{\color[rgb]{0,0,0}$\abcd{-2}{3}{-3}{4}$}%
}}}}
\put(2701,-286){\makebox(0,0)[lb]{\smash{{\SetFigFont{7}{8.4}{\rmdefault}{\mddefault}{\updefault}{\color[rgb]{0,0,0}$\abcd{1}{3}{0}{1}$}%
}}}}
\put(2116,-2581){\makebox(0,0)[lb]{\smash{{\SetFigFont{7}{8.4}{\rmdefault}{\mddefault}{\updefault}{\color[rgb]{0,0,0}$0$}%
}}}}
\put(316,-2581){\makebox(0,0)[lb]{\smash{{\SetFigFont{7}{8.4}{\rmdefault}{\mddefault}{\updefault}{\color[rgb]{0,0,0}$-1$}%
}}}}
\put(3016,-2581){\makebox(0,0)[lb]{\smash{{\SetFigFont{7}{8.4}{\rmdefault}{\mddefault}{\updefault}{\color[rgb]{0,0,0}$\frac{1}{2}$}%
}}}}
\put(3961,-2581){\makebox(0,0)[lb]{\smash{{\SetFigFont{7}{8.4}{\rmdefault}{\mddefault}{\updefault}{\color[rgb]{0,0,0}$1$}%
}}}}
\put(5716,-2581){\makebox(0,0)[lb]{\smash{{\SetFigFont{7}{8.4}{\rmdefault}{\mddefault}{\updefault}{\color[rgb]{0,0,0}$2$}%
}}}}
\put(8641,-2566){\makebox(0,0)[lb]{\smash{{\SetFigFont{7}{8.4}{\rmdefault}{\mddefault}{\updefault}{\color[rgb]{0,0,0}$0$}%
}}}}
\put(9496,-7111){\makebox(0,0)[lb]{\smash{{\SetFigFont{7}{8.4}{\rmdefault}{\mddefault}{\updefault}{\color[rgb]{0,0,0}$0$}%
}}}}
\put(8596,1469){\makebox(0,0)[lb]{\smash{{\SetFigFont{7}{8.4}{\rmdefault}{\mddefault}{\updefault}{\color[rgb]{0,0,0}$0$}%
}}}}
\put(7291,1469){\makebox(0,0)[lb]{\smash{{\SetFigFont{7}{8.4}{\rmdefault}{\mddefault}{\updefault}{\color[rgb]{0,0,0}$-\frac{1}{4}$}%
}}}}
\put(2204,-1840){\makebox(0,0)[lb]{\smash{{\SetFigFont{7}{8.4}{\rmdefault}{\mddefault}{\updefault}{\color[rgb]{0,0,0}$3$}%
}}}}
\put(3884,-1840){\makebox(0,0)[lb]{\smash{{\SetFigFont{7}{8.4}{\rmdefault}{\mddefault}{\updefault}{\color[rgb]{0,0,0}$3$}%
}}}}
\put(2114,-6222){\makebox(0,0)[lb]{\smash{{\SetFigFont{7}{8.4}{\rmdefault}{\mddefault}{\updefault}{\color[rgb]{0,0,0}$6$}%
}}}}
\put(9443,-6092){\makebox(0,0)[lb]{\smash{{\SetFigFont{7}{8.4}{\rmdefault}{\mddefault}{\updefault}{\color[rgb]{0,0,0}$5$}%
}}}}
\put(4020,-6439){\makebox(0,0)[lb]{\smash{{\SetFigFont{7}{8.4}{\rmdefault}{\mddefault}{\updefault}{\color[rgb]{0,0,0}$1$}%
}}}}
\put(3063,-1921){\makebox(0,0)[lb]{\smash{{\SetFigFont{7}{8.4}{\rmdefault}{\mddefault}{\updefault}{\color[rgb]{0,0,0}$1$}%
}}}}
\put(12195,-6608){\makebox(0,0)[lb]{\smash{{\SetFigFont{7}{8.4}{\rmdefault}{\mddefault}{\updefault}{\color[rgb]{0,0,0}$5$}%
}}}}
\put(10001,-1771){\makebox(0,0)[lb]{\smash{{\SetFigFont{7}{8.4}{\rmdefault}{\mddefault}{\updefault}{\color[rgb]{0,0,0}$1$}%
}}}}
\put(8660,-1818){\makebox(0,0)[lb]{\smash{{\SetFigFont{7}{8.4}{\rmdefault}{\mddefault}{\updefault}{\color[rgb]{0,0,0}$4$}%
}}}}
\put(4882,-6478){\makebox(0,0)[lb]{\smash{{\SetFigFont{7}{8.4}{\rmdefault}{\mddefault}{\updefault}{\color[rgb]{0,0,0}$3$}%
}}}}
\put(5595,2298){\makebox(0,0)[lb]{\smash{{\SetFigFont{7}{8.4}{\rmdefault}{\mddefault}{\updefault}{\color[rgb]{0,0,0}$\frac{1}{3}$}%
}}}}
\put(495,2486){\makebox(0,0)[lb]{\smash{{\SetFigFont{7}{8.4}{\rmdefault}{\mddefault}{\updefault}{\color[rgb]{0,0,0}$\frac{1}{3}$}%
}}}}
\put(3026,2186){\makebox(0,0)[lb]{\smash{{\SetFigFont{7}{8.4}{\rmdefault}{\mddefault}{\updefault}{\color[rgb]{0,0,0}$\frac{1}{3}$}%
}}}}
\put(476,-6140){\makebox(0,0)[lb]{\smash{{\SetFigFont{7}{8.4}{\rmdefault}{\mddefault}{\updefault}{\color[rgb]{0,0,0}$\frac{1}{3}$}%
}}}}
\put(10001,2167){\makebox(0,0)[lb]{\smash{{\SetFigFont{7}{8.4}{\rmdefault}{\mddefault}{\updefault}{\color[rgb]{0,0,0}$\frac{1}{2}$}%
}}}}
\put(7451,2354){\makebox(0,0)[lb]{\smash{{\SetFigFont{7}{8.4}{\rmdefault}{\mddefault}{\updefault}{\color[rgb]{0,0,0}$\frac{1}{4}$}%
}}}}
\put(7451,-1714){\makebox(0,0)[lb]{\smash{{\SetFigFont{7}{8.4}{\rmdefault}{\mddefault}{\updefault}{\color[rgb]{0,0,0}$\frac{1}{2}$}%
}}}}
\put(476,-1584){\makebox(0,0)[lb]{\smash{{\SetFigFont{7}{8.4}{\rmdefault}{\mddefault}{\updefault}{\color[rgb]{0,0,0}$1$}%
}}}}
\put(5557,-1603){\makebox(0,0)[lb]{\smash{{\SetFigFont{7}{8.4}{\rmdefault}{\mddefault}{\updefault}{\color[rgb]{0,0,0}$1$}%
}}}}
\put(11332,-1883){\makebox(0,0)[lb]{\smash{{\SetFigFont{7}{8.4}{\rmdefault}{\mddefault}{\updefault}{\color[rgb]{0,0,0}$4$}%
}}}}
\put(5613,-6308){\makebox(0,0)[lb]{\smash{{\SetFigFont{7}{8.4}{\rmdefault}{\mddefault}{\updefault}{\color[rgb]{0,0,0}$\frac{2}{3}$}%
}}}}
\put(11670,-6364){\makebox(0,0)[lb]{\smash{{\SetFigFont{7}{8.4}{\rmdefault}{\mddefault}{\updefault}{\color[rgb]{0,0,0}$\frac{1}{2}$}%
}}}}
\put(7469,-5895){\makebox(0,0)[lb]{\smash{{\SetFigFont{7}{8.4}{\rmdefault}{\mddefault}{\updefault}{\color[rgb]{0,0,0}$\frac{1}{10}$}%
}}}}
\put(12533,-1658){\makebox(0,0)[lb]{\smash{{\SetFigFont{7}{8.4}{\rmdefault}{\mddefault}{\updefault}{\color[rgb]{0,0,0}$\frac{1}{2}$}%
}}}}
\put(12551,2354){\makebox(0,0)[lb]{\smash{{\SetFigFont{7}{8.4}{\rmdefault}{\mddefault}{\updefault}{\color[rgb]{0,0,0}$\frac{1}{4}$}%
}}}}
\put(12644,-6401){\makebox(0,0)[lb]{\smash{{\SetFigFont{7}{8.4}{\rmdefault}{\mddefault}{\updefault}{\color[rgb]{0,0,0}$\frac{2}{5}$}%
}}}}
\end{picture}%

%% file: level8subgroups.pstex_t
\begin{picture}(0,0)%
\includegraphics{level8subgroups.pstex}%
\end{picture}%
\setlength{\unitlength}{2072sp}%
\begingroup\makeatletter\ifx\SetFigFont\undefined%
\gdef\SetFigFont#1#2#3#4#5{%
  \reset@font\fontsize{#1}{#2pt}%
  \fontfamily{#3}\fontseries{#4}\fontshape{#5}%
  \selectfont}%
\fi\endgroup%
\begin{picture}(11614,16990)(439,-17264)
\put(1498,-17206){\makebox(0,0)[lb]{\smash{{\SetFigFont{8}{9.6}{\rmdefault}{\mddefault}{\updefault}{\color[rgb]{0,0,0}$-10$}%
}}}}
\put(10081,-2851){\makebox(0,0)[lb]{\smash{{\SetFigFont{10}{12.0}{\rmdefault}{\mddefault}{\updefault}{\color[rgb]{0,0,0}$g^{3}tg^{-3}$}%
}}}}
\put(5311,-3931){\makebox(0,0)[lb]{\smash{{\SetFigFont{8}{9.6}{\rmdefault}{\mddefault}{\updefault}{\color[rgb]{0,0,0}$0$}%
}}}}
\put(1126,-511){\makebox(0,0)[lb]{\smash{{\SetFigFont{10}{12.0}{\rmdefault}{\mddefault}{\updefault}{\color[rgb]{0,0,0}$t=\abcd{1}{0}{1}{1}$}%
}}}}
\put(4321,-3931){\makebox(0,0)[lb]{\smash{{\SetFigFont{8}{9.6}{\rmdefault}{\mddefault}{\updefault}{\color[rgb]{0,0,0}$-2$}%
}}}}
\put(3421,-3931){\makebox(0,0)[lb]{\smash{{\SetFigFont{8}{9.6}{\rmdefault}{\mddefault}{\updefault}{\color[rgb]{0,0,0}$-4$}%
}}}}
\put(6211,-3931){\makebox(0,0)[lb]{\smash{{\SetFigFont{8}{9.6}{\rmdefault}{\mddefault}{\updefault}{\color[rgb]{0,0,0}$2$}%
}}}}
\put(7111,-3931){\makebox(0,0)[lb]{\smash{{\SetFigFont{8}{9.6}{\rmdefault}{\mddefault}{\updefault}{\color[rgb]{0,0,0}$4$}%
}}}}
\put(8011,-3931){\makebox(0,0)[lb]{\smash{{\SetFigFont{8}{9.6}{\rmdefault}{\mddefault}{\updefault}{\color[rgb]{0,0,0}$6$}%
}}}}
\put(8911,-3931){\makebox(0,0)[lb]{\smash{{\SetFigFont{8}{9.6}{\rmdefault}{\mddefault}{\updefault}{\color[rgb]{0,0,0}$8$}%
}}}}
\put(9721,-3931){\makebox(0,0)[lb]{\smash{{\SetFigFont{8}{9.6}{\rmdefault}{\mddefault}{\updefault}{\color[rgb]{0,0,0}$10$}%
}}}}
\put(11521,-3931){\makebox(0,0)[lb]{\smash{{\SetFigFont{8}{9.6}{\rmdefault}{\mddefault}{\updefault}{\color[rgb]{0,0,0}$14$}%
}}}}
\put(2521,-3931){\makebox(0,0)[lb]{\smash{{\SetFigFont{8}{9.6}{\rmdefault}{\mddefault}{\updefault}{\color[rgb]{0,0,0}$-6$}%
}}}}
\put(1621,-3931){\makebox(0,0)[lb]{\smash{{\SetFigFont{8}{9.6}{\rmdefault}{\mddefault}{\updefault}{\color[rgb]{0,0,0}$-8$}%
}}}}
\put(5131,-2851){\makebox(0,0)[lb]{\smash{{\SetFigFont{10}{12.0}{\rmdefault}{\mddefault}{\updefault}{\color[rgb]{0,0,0}$t$}%
}}}}
\put(5761,-1951){\makebox(0,0)[lb]{\smash{{\SetFigFont{10}{12.0}{\rmdefault}{\mddefault}{\updefault}{\color[rgb]{0,0,0}$g^6$}%
}}}}
\put(1126,-961){\makebox(0,0)[lb]{\smash{{\SetFigFont{10}{12.0}{\rmdefault}{\mddefault}{\updefault}{\color[rgb]{0,0,0}$g=\abcd{1}{4}{0}{1}$}%
}}}}
\put(3241,-2851){\makebox(0,0)[lb]{\smash{{\SetFigFont{10}{12.0}{\rmdefault}{\mddefault}{\updefault}{\color[rgb]{0,0,0}$g^{-1}tg$}%
}}}}
\put(8281,-2851){\makebox(0,0)[lb]{\smash{{\SetFigFont{10}{12.0}{\rmdefault}{\mddefault}{\updefault}{\color[rgb]{0,0,0}$g^2tg^{-2}$}%
}}}}
\put(1441,-2851){\makebox(0,0)[lb]{\smash{{\SetFigFont{10}{12.0}{\rmdefault}{\mddefault}{\updefault}{\color[rgb]{0,0,0}$g^{-2}tg^2$}%
}}}}
\put(5761,-2761){\makebox(0,0)[lb]{\smash{{\SetFigFont{10}{12.0}{\rmdefault}{\mddefault}{\updefault}{\color[rgb]{0,0,0}$\abcd{5}{-16}{1}{-3}=gtg^{-1}$}%
}}}}
\put(1126,-1501){\makebox(0,0)[lb]{\smash{{\SetFigFont{10}{12.0}{\rmdefault}{\mddefault}{\updefault}{\color[rgb]{0,0,0}$s=\abcd{0}{1}{-1}{0}$}%
}}}}
\put(8139,-1051){\makebox(0,0)[lb]{\smash{{\SetFigFont{9}{10.8}{\rmdefault}{\mddefault}{\updefault}{\color[rgb]{0,0,0}Fundamental domain for}%
}}}}
\put(5324,-3391){\makebox(0,0)[lb]{\smash{{\SetFigFont{6}{7.2}{\familydefault}{\mddefault}{\updefault}{\color[rgb]{0,0,0}$1$}%
}}}}
\put(7112,-3391){\makebox(0,0)[lb]{\smash{{\SetFigFont{6}{7.2}{\familydefault}{\mddefault}{\updefault}{\color[rgb]{0,0,0}$1$}%
}}}}
\put(6211,-3311){\makebox(0,0)[lb]{\smash{{\SetFigFont{6}{7.2}{\familydefault}{\mddefault}{\updefault}{\color[rgb]{0,0,0}$\frac{1}{2}$}%
}}}}
\put(8551,-1381){\makebox(0,0)[lb]{\smash{{\SetFigFont{9}{10.8}{\rmdefault}{\mddefault}{\updefault}{\color[rgb]{0,0,0}$s\Gamma_{24.6.1^6} s^{-1}$}%
}}}}
\put(721,-3931){\makebox(0,0)[lb]{\smash{{\SetFigFont{8}{9.6}{\rmdefault}{\mddefault}{\updefault}{\color[rgb]{0,0,0}$-10$}%
}}}}
\put(10621,-3931){\makebox(0,0)[lb]{\smash{{\SetFigFont{8}{9.6}{\rmdefault}{\mddefault}{\updefault}{\color[rgb]{0,0,0}$12$}%
}}}}
\put(1126,-9511){\makebox(0,0)[lb]{\smash{{\SetFigFont{10}{12.0}{\rmdefault}{\mddefault}{\updefault}{\color[rgb]{0,0,0}$t=\abcd{1}{1}{0}{1}$}%
}}}}
\put(1126,-9961){\makebox(0,0)[lb]{\smash{{\SetFigFont{10}{12.0}{\rmdefault}{\mddefault}{\updefault}{\color[rgb]{0,0,0}$g=\abcd{1}{0}{8}{1}$}%
}}}}
\put(5761,-10951){\makebox(0,0)[lb]{\smash{{\SetFigFont{10}{12.0}{\rmdefault}{\mddefault}{\updefault}{\color[rgb]{0,0,0}$t^6$}%
}}}}
\put(4321,-12931){\makebox(0,0)[lb]{\smash{{\SetFigFont{8}{9.6}{\rmdefault}{\mddefault}{\updefault}{\color[rgb]{0,0,0}$-\frac{3}{2}$}%
}}}}
\put(8011,-12931){\makebox(0,0)[lb]{\smash{{\SetFigFont{8}{9.6}{\rmdefault}{\mddefault}{\updefault}{\color[rgb]{0,0,0}$\frac{1}{2}$}%
}}}}
\put(6976,-11986){\makebox(0,0)[lb]{\smash{{\SetFigFont{10}{12.0}{\rmdefault}{\mddefault}{\updefault}{\color[rgb]{0,0,0}$g$}%
}}}}
\put(7169,-12931){\makebox(0,0)[lb]{\smash{{\SetFigFont{8}{9.6}{\rmdefault}{\mddefault}{\updefault}{\color[rgb]{0,0,0}$0$}%
}}}}
\put(10399,-11966){\makebox(0,0)[lb]{\smash{{\SetFigFont{10}{12.0}{\rmdefault}{\mddefault}{\updefault}{\color[rgb]{0,0,0}$t^2gt^{-2}$}%
}}}}
\put(8792,-11679){\makebox(0,0)[lb]{\smash{{\SetFigFont{10}{12.0}{\rmdefault}{\mddefault}{\updefault}{\color[rgb]{0,0,0}$t^{2}st^{-1}s^{-1}t^{-2}$}%
}}}}
\put(8116,-10051){\makebox(0,0)[lb]{\smash{{\SetFigFont{9}{10.8}{\rmdefault}{\mddefault}{\updefault}{\color[rgb]{0,0,0}Fundamental domain for}%
}}}}
\put(6256,-12345){\makebox(0,0)[lb]{\smash{{\SetFigFont{6}{7.2}{\familydefault}{\mddefault}{\updefault}{\color[rgb]{0,0,0}$1$}%
}}}}
\put(7066,-12481){\makebox(0,0)[lb]{\smash{{\SetFigFont{6}{7.2}{\familydefault}{\mddefault}{\updefault}{\color[rgb]{0,0,0}$8$}%
}}}}
\put(1956,-11713){\makebox(0,0)[lb]{\smash{{\SetFigFont{10}{12.0}{\rmdefault}{\mddefault}{\updefault}{\color[rgb]{0,0,0}$t^{-2}st^{-1}s^{-1}t^2$}%
}}}}
\put(3493,-11977){\makebox(0,0)[lb]{\smash{{\SetFigFont{10}{12.0}{\rmdefault}{\mddefault}{\updefault}{\color[rgb]{0,0,0}$t^{-2}gt^2$}%
}}}}
\put(6701,-12886){\makebox(0,0)[lb]{\smash{{\SetFigFont{5}{6.0}{\rmdefault}{\mddefault}{\updefault}{\color[rgb]{0,0,0}$-\frac{1}{6}$}%
}}}}
\put(6108,-12931){\makebox(0,0)[lb]{\smash{{\SetFigFont{8}{9.6}{\rmdefault}{\mddefault}{\updefault}{\color[rgb]{0,0,0}$-\frac{1}{2}$}%
}}}}
\put(3421,-12931){\makebox(0,0)[lb]{\smash{{\SetFigFont{8}{9.6}{\rmdefault}{\mddefault}{\updefault}{\color[rgb]{0,0,0}$-2$}%
}}}}
\put(2521,-12931){\makebox(0,0)[lb]{\smash{{\SetFigFont{8}{9.6}{\rmdefault}{\mddefault}{\updefault}{\color[rgb]{0,0,0}$-\frac{5}{2}$}%
}}}}
\put(721,-12931){\makebox(0,0)[lb]{\smash{{\SetFigFont{8}{9.6}{\rmdefault}{\mddefault}{\updefault}{\color[rgb]{0,0,0}$-\frac{7}{2}$}%
}}}}
\put(5886,-11793){\makebox(0,0)[lb]{\smash{{\SetFigFont{10}{12.0}{\rmdefault}{\mddefault}{\updefault}{\color[rgb]{0,0,0}$st^{-1}s^{-1}$}%
}}}}
\put(11521,-12931){\makebox(0,0)[lb]{\smash{{\SetFigFont{8}{9.6}{\rmdefault}{\mddefault}{\updefault}{\color[rgb]{0,0,0}$\frac{5}{2}$}%
}}}}
\put(1126,-10501){\makebox(0,0)[lb]{\smash{{\SetFigFont{10}{12.0}{\rmdefault}{\mddefault}{\updefault}{\color[rgb]{0,0,0}$s=\abcd{1}{0}{-2}{1}$}%
}}}}
\put(9721,-12931){\makebox(0,0)[lb]{\smash{{\SetFigFont{8}{9.6}{\rmdefault}{\mddefault}{\updefault}{\color[rgb]{0,0,0}$\frac{3}{2}$}%
}}}}
\put(10738,-12942){\makebox(0,0)[lb]{\smash{{\SetFigFont{8}{9.6}{\rmdefault}{\mddefault}{\updefault}{\color[rgb]{0,0,0}$2$}%
}}}}
\put(8551,-10381){\makebox(0,0)[lb]{\smash{{\SetFigFont{9}{10.8}{\rmdefault}{\mddefault}{\updefault}{\color[rgb]{0,0,0}$s\Gamma_{8^36.3.1^3}s^{-1}$}%
}}}}
\put(10621,-8431){\makebox(0,0)[lb]{\smash{{\SetFigFont{8}{9.6}{\rmdefault}{\mddefault}{\updefault}{\color[rgb]{0,0,0}$1$}%
}}}}
\put(8551,-5881){\makebox(0,0)[lb]{\smash{{\SetFigFont{9}{10.8}{\rmdefault}{\mddefault}{\updefault}{\color[rgb]{0,0,0}$\Gamma_{8^32^33^2} $}%
}}}}
\put(1126,-5461){\makebox(0,0)[lb]{\smash{{\SetFigFont{10}{12.0}{\rmdefault}{\mddefault}{\updefault}{\color[rgb]{0,0,0}$g=\abcd{1}{0}{8}{1}$}%
}}}}
\put(5131,-7351){\makebox(0,0)[lb]{\smash{{\SetFigFont{10}{12.0}{\rmdefault}{\mddefault}{\updefault}{\color[rgb]{0,0,0}$g$}%
}}}}
\put(4321,-8431){\makebox(0,0)[lb]{\smash{{\SetFigFont{8}{9.6}{\rmdefault}{\mddefault}{\updefault}{\color[rgb]{0,0,0}$-\frac{1}{4}$}%
}}}}
\put(5311,-8431){\makebox(0,0)[lb]{\smash{{\SetFigFont{8}{9.6}{\rmdefault}{\mddefault}{\updefault}{\color[rgb]{0,0,0}$0$}%
}}}}
\put(6211,-8431){\makebox(0,0)[lb]{\smash{{\SetFigFont{8}{9.6}{\rmdefault}{\mddefault}{\updefault}{\color[rgb]{0,0,0}$\frac{1}{4}$}%
}}}}
\put(7111,-8431){\makebox(0,0)[lb]{\smash{{\SetFigFont{8}{9.6}{\rmdefault}{\mddefault}{\updefault}{\color[rgb]{0,0,0}$\frac{1}{2}$}%
}}}}
\put(8011,-8431){\makebox(0,0)[lb]{\smash{{\SetFigFont{8}{9.6}{\rmdefault}{\mddefault}{\updefault}{\color[rgb]{0,0,0}$\frac{3}{4}$}%
}}}}
\put(8911,-8431){\makebox(0,0)[lb]{\smash{{\SetFigFont{8}{9.6}{\rmdefault}{\mddefault}{\updefault}{\color[rgb]{0,0,0}$1$}%
}}}}
\put(9721,-8431){\makebox(0,0)[lb]{\smash{{\SetFigFont{8}{9.6}{\rmdefault}{\mddefault}{\updefault}{\color[rgb]{0,0,0}$\frac{5}{4}$}%
}}}}
\put(3421,-8431){\makebox(0,0)[lb]{\smash{{\SetFigFont{8}{9.6}{\rmdefault}{\mddefault}{\updefault}{\color[rgb]{0,0,0}$-\frac{1}{2}$}%
}}}}
\put(2521,-8431){\makebox(0,0)[lb]{\smash{{\SetFigFont{8}{9.6}{\rmdefault}{\mddefault}{\updefault}{\color[rgb]{0,0,0}$-\frac{3}{4}$}%
}}}}
\put(1621,-8431){\makebox(0,0)[lb]{\smash{{\SetFigFont{8}{9.6}{\rmdefault}{\mddefault}{\updefault}{\color[rgb]{0,0,0}$-1$}%
}}}}
\put(721,-8431){\makebox(0,0)[lb]{\smash{{\SetFigFont{8}{9.6}{\rmdefault}{\mddefault}{\updefault}{\color[rgb]{0,0,0}$-\frac{5}{4}$}%
}}}}
\put(1441,-7351){\makebox(0,0)[lb]{\smash{{\SetFigFont{10}{12.0}{\rmdefault}{\mddefault}{\updefault}{\color[rgb]{0,0,0}$t^{-2}gt^2$}%
}}}}
\put(3241,-7351){\makebox(0,0)[lb]{\smash{{\SetFigFont{10}{12.0}{\rmdefault}{\mddefault}{\updefault}{\color[rgb]{0,0,0}$t^{-1}gt$}%
}}}}
\put(10081,-7351){\makebox(0,0)[lb]{\smash{{\SetFigFont{10}{12.0}{\rmdefault}{\mddefault}{\updefault}{\color[rgb]{0,0,0}$t^{3}gt^{-3}$}%
}}}}
\put(5761,-7261){\makebox(0,0)[lb]{\smash{{\SetFigFont{10}{12.0}{\rmdefault}{\mddefault}{\updefault}{\color[rgb]{0,0,0}$\abcd{5}{-2}{8}{-3}=tgt^{-1}$}%
}}}}
\put(8281,-7351){\makebox(0,0)[lb]{\smash{{\SetFigFont{10}{12.0}{\rmdefault}{\mddefault}{\updefault}{\color[rgb]{0,0,0}$t^2gt^{-2}$}%
}}}}
\put(5761,-6451){\makebox(0,0)[lb]{\smash{{\SetFigFont{10}{12.0}{\rmdefault}{\mddefault}{\updefault}{\color[rgb]{0,0,0}$t^6$}%
}}}}
\put(1126,-5011){\makebox(0,0)[lb]{\smash{{\SetFigFont{10}{12.0}{\rmdefault}{\mddefault}{\updefault}{\color[rgb]{0,0,0}$t=\abcd{1}{1/2}{0}{1}$}%
}}}}
\put(7201,-5551){\makebox(0,0)[lb]{\smash{{\SetFigFont{9}{10.8}{\rmdefault}{\mddefault}{\updefault}{\color[rgb]{0,0,0}Fundamental domain for}%
}}}}
\put(5266,-7801){\makebox(0,0)[lb]{\smash{{\SetFigFont{6}{7.2}{\familydefault}{\mddefault}{\updefault}{\color[rgb]{0,0,0}$8$}%
}}}}
\put(7111,-7846){\makebox(0,0)[lb]{\smash{{\SetFigFont{6}{7.2}{\familydefault}{\mddefault}{\updefault}{\color[rgb]{0,0,0}$2$}%
}}}}
\put(6200,-7846){\makebox(0,0)[lb]{\smash{{\SetFigFont{6}{7.2}{\familydefault}{\mddefault}{\updefault}{\color[rgb]{0,0,0}$\frac{1}{2}$}%
}}}}
\put(4321,-17206){\makebox(0,0)[lb]{\smash{{\SetFigFont{8}{9.6}{\rmdefault}{\mddefault}{\updefault}{\color[rgb]{0,0,0}$-4$}%
}}}}
\put(8011,-17206){\makebox(0,0)[lb]{\smash{{\SetFigFont{8}{9.6}{\rmdefault}{\mddefault}{\updefault}{\color[rgb]{0,0,0}$4$}%
}}}}
\put(6211,-17206){\makebox(0,0)[lb]{\smash{{\SetFigFont{8}{9.6}{\rmdefault}{\mddefault}{\updefault}{\color[rgb]{0,0,0}$0$}%
}}}}
\put(1126,-14956){\makebox(0,0)[lb]{\smash{{\SetFigFont{10}{12.0}{\rmdefault}{\mddefault}{\updefault}{\color[rgb]{0,0,0}$s=\abcd{0}{1}{-1}{0}$}%
}}}}
\put(1126,-14056){\makebox(0,0)[lb]{\smash{{\SetFigFont{10}{12.0}{\rmdefault}{\mddefault}{\updefault}{\color[rgb]{0,0,0}$t=\abcd{1}{0}{1}{1}$}%
}}}}
\put(11555,-17206){\makebox(0,0)[lb]{\smash{{\SetFigFont{8}{9.6}{\rmdefault}{\mddefault}{\updefault}{\color[rgb]{0,0,0}$12$}%
}}}}
\put(9792,-17206){\makebox(0,0)[lb]{\smash{{\SetFigFont{8}{9.6}{\rmdefault}{\mddefault}{\updefault}{\color[rgb]{0,0,0}$8$}%
}}}}
\put(736,-17206){\makebox(0,0)[lb]{\smash{{\SetFigFont{8}{9.6}{\rmdefault}{\mddefault}{\updefault}{\color[rgb]{0,0,0}$-12$}%
}}}}
\put(2517,-17206){\makebox(0,0)[lb]{\smash{{\SetFigFont{8}{9.6}{\rmdefault}{\mddefault}{\updefault}{\color[rgb]{0,0,0}$-8$}%
}}}}
\put(6125,-16107){\makebox(0,0)[lb]{\smash{{\SetFigFont{10}{12.0}{\rmdefault}{\mddefault}{\updefault}{\color[rgb]{0,0,0}$t$}%
}}}}
\put(5611,-17180){\makebox(0,0)[lb]{\smash{{\SetFigFont{5}{6.0}{\rmdefault}{\mddefault}{\updefault}{\color[rgb]{0,0,0}$-\frac{4}{3}$}%
}}}}
\put(8116,-14133){\makebox(0,0)[lb]{\smash{{\SetFigFont{9}{10.8}{\rmdefault}{\mddefault}{\updefault}{\color[rgb]{0,0,0}Fundamental domain for}%
}}}}
\put(8551,-14553){\makebox(0,0)[lb]{\smash{{\SetFigFont{9}{10.8}{\rmdefault}{\mddefault}{\updefault}{\color[rgb]{0,0,0}$s\Gamma_{24.3.2^3.1^3}s^{-1}$}%
}}}}
\put(1126,-14506){\makebox(0,0)[lb]{\smash{{\SetFigFont{10}{12.0}{\rmdefault}{\mddefault}{\updefault}{\color[rgb]{0,0,0}$g=\abcd{1}{2}{0}{1}$}%
}}}}
\put(5761,-14644){\makebox(0,0)[lb]{\smash{{\SetFigFont{10}{12.0}{\rmdefault}{\mddefault}{\updefault}{\color[rgb]{0,0,0}$g^{12}$}%
}}}}
\put(2455,-15957){\makebox(0,0)[lb]{\smash{{\SetFigFont{10}{12.0}{\rmdefault}{\mddefault}{\updefault}{\color[rgb]{0,0,0}$g^{-4}tg^4$}%
}}}}
\put(9493,-15962){\makebox(0,0)[lb]{\smash{{\SetFigFont{10}{12.0}{\rmdefault}{\mddefault}{\updefault}{\color[rgb]{0,0,0}$g^4tg^{-4}$}%
}}}}
\put(4609,-16357){\makebox(0,0)[lb]{\smash{{\SetFigFont{10}{12.0}{\rmdefault}{\mddefault}{\updefault}{\color[rgb]{0,0,0}$g^{-1}t^2g$}%
}}}}
\put(1162,-16357){\makebox(0,0)[lb]{\smash{{\SetFigFont{10}{12.0}{\rmdefault}{\mddefault}{\updefault}{\color[rgb]{0,0,0}$g^{-5}t^2g^5$}%
}}}}
\put(8349,-16357){\makebox(0,0)[lb]{\smash{{\SetFigFont{10}{12.0}{\rmdefault}{\mddefault}{\updefault}{\color[rgb]{0,0,0}$g^{3}t^2g^{-3}$}%
}}}}
\put(8894,-17206){\makebox(0,0)[lb]{\smash{{\SetFigFont{8}{9.6}{\rmdefault}{\mddefault}{\updefault}{\color[rgb]{0,0,0}$6$}%
}}}}
\put(5197,-17206){\makebox(0,0)[lb]{\smash{{\SetFigFont{8}{9.6}{\rmdefault}{\mddefault}{\updefault}{\color[rgb]{0,0,0}$-2$}%
}}}}
\put(6149,-16652){\makebox(0,0)[lb]{\smash{{\SetFigFont{6}{7.2}{\familydefault}{\mddefault}{\updefault}{\color[rgb]{0,0,0}$1$}%
}}}}
\put(5344,-16776){\makebox(0,0)[lb]{\smash{{\SetFigFont{6}{7.2}{\familydefault}{\mddefault}{\updefault}{\color[rgb]{0,0,0}$2$}%
}}}}
\end{picture}%

%% file: level6subgps.pstex_t
\begin{picture}(0,0)%
\includegraphics{level6subgps.pstex}%
\end{picture}%
\setlength{\unitlength}{1450sp}%
\begingroup\makeatletter\ifx\SetFigFont\undefined%
\gdef\SetFigFont#1#2#3#4#5{%
  \reset@font\fontsize{#1}{#2pt}%
  \fontfamily{#3}\fontseries{#4}\fontshape{#5}%
  \selectfont}%
\fi\endgroup%
\begin{picture}(16716,17752)(214,-11693)
\put(676,-4651){\makebox(0,0)[lb]{\smash{{\SetFigFont{6}{7.2}{\rmdefault}{\mddefault}{\updefault}{\color[rgb]{0,0,0}$g_1=\abcd{-5}{-12}{3}{7}$}%
}}}}
\put(8461,-2476){\makebox(0,0)[lb]{\smash{{\SetFigFont{6}{7.2}{\rmdefault}{\mddefault}{\updefault}{\color[rgb]{0,0,0}$0$}%
}}}}
\put(10261,-2446){\makebox(0,0)[lb]{\smash{{\SetFigFont{6}{7.2}{\rmdefault}{\mddefault}{\updefault}{\color[rgb]{0,0,0}$\frac{1}{3}$}%
}}}}
\put(5671,-2476){\makebox(0,0)[lb]{\smash{{\SetFigFont{6}{7.2}{\rmdefault}{\mddefault}{\updefault}{\color[rgb]{0,0,0}$-\frac{1}{2}$}%
}}}}
\put(11071,-2476){\makebox(0,0)[lb]{\smash{{\SetFigFont{6}{7.2}{\rmdefault}{\mddefault}{\updefault}{\color[rgb]{0,0,0}$\frac{1}{2}$}%
}}}}
\put(9766,-2476){\makebox(0,0)[lb]{\smash{{\SetFigFont{6}{7.2}{\rmdefault}{\mddefault}{\updefault}{\color[rgb]{0,0,0}$\frac{1}{4}$}%
}}}}
\put(316,-2476){\makebox(0,0)[lb]{\smash{{\SetFigFont{6}{7.2}{\rmdefault}{\mddefault}{\updefault}{\color[rgb]{0,0,0}$-\frac{3}{2}$}%
}}}}
\put(13771,-2476){\makebox(0,0)[lb]{\smash{{\SetFigFont{6}{7.2}{\rmdefault}{\mddefault}{\updefault}{\color[rgb]{0,0,0}$1$}%
}}}}
\put(2971,-2476){\makebox(0,0)[lb]{\smash{{\SetFigFont{6}{7.2}{\rmdefault}{\mddefault}{\updefault}{\color[rgb]{0,0,0}$-1$}%
}}}}
\put(10171,-1951){\makebox(0,0)[lb]{\smash{{\SetFigFont{6}{7.2}{\rmdefault}{\mddefault}{\updefault}{\color[rgb]{0,0,0}$2$}%
}}}}
\put(16471,-2476){\makebox(0,0)[lb]{\smash{{\SetFigFont{6}{7.2}{\rmdefault}{\mddefault}{\updefault}{\color[rgb]{0,0,0}$\frac{3}{2}$}%
}}}}
\put(15661,-2491){\makebox(0,0)[lb]{\smash{{\SetFigFont{6}{7.2}{\rmdefault}{\mddefault}{\updefault}{\color[rgb]{0,0,0}$\frac{5}{4}$}%
}}}}
\put(15256,-2491){\makebox(0,0)[lb]{\smash{{\SetFigFont{6}{7.2}{\rmdefault}{\mddefault}{\updefault}{\color[rgb]{0,0,0}$\frac{4}{3}$}%
}}}}
\put(4771,-2491){\makebox(0,0)[lb]{\smash{{\SetFigFont{6}{7.2}{\rmdefault}{\mddefault}{\updefault}{\color[rgb]{0,0,0}$-\frac{2}{3}$}%
}}}}
\put(8416,-781){\makebox(0,0)[lb]{\smash{{\SetFigFont{6}{7.2}{\rmdefault}{\mddefault}{\updefault}{\color[rgb]{0,0,0}$g_6$}%
}}}}
\put(9991,-1366){\makebox(0,0)[lb]{\smash{{\SetFigFont{6}{7.2}{\rmdefault}{\mddefault}{\updefault}{\color[rgb]{0,0,0}$g_2$}%
}}}}
\put(3106,-781){\makebox(0,0)[lb]{\smash{{\SetFigFont{6}{7.2}{\rmdefault}{\mddefault}{\updefault}{\color[rgb]{0,0,0}$t^{-1}g_6t$}%
}}}}
\put(4411,-1366){\makebox(0,0)[lb]{\smash{{\SetFigFont{6}{7.2}{\rmdefault}{\mddefault}{\updefault}{\color[rgb]{0,0,0}$t^{-1}g_2t$}%
}}}}
\put(13276,-691){\makebox(0,0)[lb]{\smash{{\SetFigFont{6}{7.2}{\rmdefault}{\mddefault}{\updefault}{\color[rgb]{0,0,0}$tg_6t^{-1}$}%
}}}}
\put(15301,-1366){\makebox(0,0)[lb]{\smash{{\SetFigFont{6}{7.2}{\rmdefault}{\mddefault}{\updefault}{\color[rgb]{0,0,0}$tg_2t^{-1}$}%
}}}}
\put(676,974){\makebox(0,0)[lb]{\smash{{\SetFigFont{6}{7.2}{\rmdefault}{\mddefault}{\updefault}{\color[rgb]{0,0,0}$t=\abcd{1}{1}{0}{1}$}%
}}}}
\put(676,524){\makebox(0,0)[lb]{\smash{{\SetFigFont{6}{7.2}{\rmdefault}{\mddefault}{\updefault}{\color[rgb]{0,0,0}$g_6=\abcd{1}{0}{6}{1}$}%
}}}}
\put(676, 74){\makebox(0,0)[lb]{\smash{{\SetFigFont{6}{7.2}{\rmdefault}{\mddefault}{\updefault}{\color[rgb]{0,0,0}$g_2=\abcd{7}{-2}{18}{-5}$}%
}}}}
\put(8461,-1636){\makebox(0,0)[lb]{\smash{{\SetFigFont{6}{7.2}{\rmdefault}{\mddefault}{\updefault}{\color[rgb]{0,0,0}$6$}%
}}}}
\put(9586,884){\makebox(0,0)[lb]{\smash{{\SetFigFont{6}{7.2}{\familydefault}{\mddefault}{\updefault}{\color[rgb]{0,0,0}Fundamental domain for $\Gamma_{9.6^3.3.2^3}$}%
}}}}
\put(7066,299){\makebox(0,0)[lb]{\smash{{\SetFigFont{6}{7.2}{\rmdefault}{\mddefault}{\updefault}{\color[rgb]{0,0,0}$t^3$}%
}}}}
\put(8461,2024){\makebox(0,0)[lb]{\smash{{\SetFigFont{6}{7.2}{\rmdefault}{\mddefault}{\updefault}{\color[rgb]{0,0,0}$0$}%
}}}}
\put(316,2024){\makebox(0,0)[lb]{\smash{{\SetFigFont{6}{7.2}{\rmdefault}{\mddefault}{\updefault}{\color[rgb]{0,0,0}$-\frac{3}{2}$}%
}}}}
\put(15301,3134){\makebox(0,0)[lb]{\smash{{\SetFigFont{6}{7.2}{\rmdefault}{\mddefault}{\updefault}{\color[rgb]{0,0,0}$tg_2t^{-1}$}%
}}}}
\put(2971,2024){\makebox(0,0)[lb]{\smash{{\SetFigFont{6}{7.2}{\rmdefault}{\mddefault}{\updefault}{\color[rgb]{0,0,0}$-1$}%
}}}}
\put(5716,2024){\makebox(0,0)[lb]{\smash{{\SetFigFont{6}{7.2}{\rmdefault}{\mddefault}{\updefault}{\color[rgb]{0,0,0}$-3$}%
}}}}
\put(6571,2024){\makebox(0,0)[lb]{\smash{{\SetFigFont{6}{7.2}{\rmdefault}{\mddefault}{\updefault}{\color[rgb]{0,0,0}$-2$}%
}}}}
\put(7111,2054){\makebox(0,0)[lb]{\smash{{\SetFigFont{6}{7.2}{\rmdefault}{\mddefault}{\updefault}{\color[rgb]{0,0,0}$-\frac{3}{2}$}%
}}}}
\put(11161,2024){\makebox(0,0)[lb]{\smash{{\SetFigFont{6}{7.2}{\rmdefault}{\mddefault}{\updefault}{\color[rgb]{0,0,0}$3$}%
}}}}
\put(13771,2024){\makebox(0,0)[lb]{\smash{{\SetFigFont{6}{7.2}{\rmdefault}{\mddefault}{\updefault}{\color[rgb]{0,0,0}$6$}%
}}}}
\put(16561,2024){\makebox(0,0)[lb]{\smash{{\SetFigFont{6}{7.2}{\rmdefault}{\mddefault}{\updefault}{\color[rgb]{0,0,0}$9$}%
}}}}
\put(12016,2009){\makebox(0,0)[lb]{\smash{{\SetFigFont{6}{7.2}{\rmdefault}{\mddefault}{\updefault}{\color[rgb]{0,0,0}$4$}%
}}}}
\put(676,5474){\makebox(0,0)[lb]{\smash{{\SetFigFont{6}{7.2}{\rmdefault}{\mddefault}{\updefault}{\color[rgb]{0,0,0}$t=\abcd{1}{0}{1}{1}$}%
}}}}
\put(676,5024){\makebox(0,0)[lb]{\smash{{\SetFigFont{6}{7.2}{\rmdefault}{\mddefault}{\updefault}{\color[rgb]{0,0,0}$g_6=\abcd{1}{6}{0}{1}$}%
}}}}
\put(8236,3809){\makebox(0,0)[lb]{\smash{{\SetFigFont{6}{7.2}{\rmdefault}{\mddefault}{\updefault}{\color[rgb]{0,0,0}$t$}%
}}}}
\put(6571,2999){\makebox(0,0)[lb]{\smash{{\SetFigFont{6}{7.2}{\rmdefault}{\mddefault}{\updefault}{\color[rgb]{0,0,0}$g_3$}%
}}}}
\put(2791,3764){\makebox(0,0)[lb]{\smash{{\SetFigFont{6}{7.2}{\rmdefault}{\mddefault}{\updefault}{\color[rgb]{0,0,0}$g_6^{-1}tg_6$}%
}}}}
\put(721,2999){\makebox(0,0)[lb]{\smash{{\SetFigFont{6}{7.2}{\rmdefault}{\mddefault}{\updefault}{\color[rgb]{0,0,0}$g_6^{-1}g_3g_6$}%
}}}}
\put(6661,2594){\makebox(0,0)[lb]{\smash{{\SetFigFont{6}{7.2}{\rmdefault}{\mddefault}{\updefault}{\color[rgb]{0,0,0}$3$}%
}}}}
\put(8371,2819){\makebox(0,0)[lb]{\smash{{\SetFigFont{6}{7.2}{\rmdefault}{\mddefault}{\updefault}{\color[rgb]{0,0,0}$1$}%
}}}}
\put(13276,3764){\makebox(0,0)[lb]{\smash{{\SetFigFont{6}{7.2}{\rmdefault}{\mddefault}{\updefault}{\color[rgb]{0,0,0}$g_6tg_6^{-1}$}%
}}}}
\put(11656,2999){\makebox(0,0)[lb]{\smash{{\SetFigFont{6}{7.2}{\rmdefault}{\mddefault}{\updefault}{\color[rgb]{0,0,0}$g_6g_3g_6^{-1}$}%
}}}}
\put(7606,4754){\makebox(0,0)[lb]{\smash{{\SetFigFont{6}{7.2}{\rmdefault}{\mddefault}{\updefault}{\color[rgb]{0,0,0}$g_6^3$}%
}}}}
\put(676,5879){\makebox(0,0)[lb]{\smash{{\SetFigFont{6}{7.2}{\rmdefault}{\mddefault}{\updefault}{\color[rgb]{0,0,0}$s=\abcd{0}{1}{-1}{0}$}%
}}}}
\put(676,4574){\makebox(0,0)[lb]{\smash{{\SetFigFont{6}{7.2}{\rmdefault}{\mddefault}{\updefault}{\color[rgb]{0,0,0}$g_3=\abcd{}{}{}{}$}%
}}}}
\put(9586,5384){\makebox(0,0)[lb]{\smash{{\SetFigFont{6}{7.2}{\familydefault}{\mddefault}{\updefault}{\color[rgb]{0,0,0}Fundamental domain for $s\Gamma_{18.6.3^3.1^3}s^{-1}$}%
}}}}
\put(676,-8161){\makebox(0,0)[lb]{\smash{{\SetFigFont{6}{7.2}{\rmdefault}{\mddefault}{\updefault}{\color[rgb]{0,0,0}$t=\abcd{1}{1}{0}{1}$}%
}}}}
\put(676,-8611){\makebox(0,0)[lb]{\smash{{\SetFigFont{6}{7.2}{\rmdefault}{\mddefault}{\updefault}{\color[rgb]{0,0,0}$g_6=\abcd{1}{0}{6}{1}$}%
}}}}
\put(676,-9061){\makebox(0,0)[lb]{\smash{{\SetFigFont{6}{7.2}{\rmdefault}{\mddefault}{\updefault}{\color[rgb]{0,0,0}$g_2=\abcd{7}{-2}{18}{-5}$}%
}}}}
\put(5716,-11611){\makebox(0,0)[lb]{\smash{{\SetFigFont{6}{7.2}{\rmdefault}{\mddefault}{\updefault}{\color[rgb]{0,0,0}$0$}%
}}}}
\put(11116,-11611){\makebox(0,0)[lb]{\smash{{\SetFigFont{6}{7.2}{\rmdefault}{\mddefault}{\updefault}{\color[rgb]{0,0,0}$1$}%
}}}}
\put(316,-11581){\makebox(0,0)[lb]{\smash{{\SetFigFont{6}{7.2}{\rmdefault}{\mddefault}{\updefault}{\color[rgb]{0,0,0}$-1$}%
}}}}
\put(7921,-11626){\makebox(0,0)[lb]{\smash{{\SetFigFont{6}{7.2}{\rmdefault}{\mddefault}{\updefault}{\color[rgb]{0,0,0}$\frac{2}{5}$}%
}}}}
\put(8416,-11626){\makebox(0,0)[lb]{\smash{{\SetFigFont{6}{7.2}{\rmdefault}{\mddefault}{\updefault}{\color[rgb]{0,0,0}$\frac{1}{2}$}%
}}}}
\put(2971,-11581){\makebox(0,0)[lb]{\smash{{\SetFigFont{6}{7.2}{\rmdefault}{\mddefault}{\updefault}{\color[rgb]{0,0,0}$-\frac{1}{2}$}%
}}}}
\put(13861,-11581){\makebox(0,0)[lb]{\smash{{\SetFigFont{6}{7.2}{\rmdefault}{\mddefault}{\updefault}{\color[rgb]{0,0,0}$\frac{2}{3}$}%
}}}}
\put(16516,-11536){\makebox(0,0)[lb]{\smash{{\SetFigFont{6}{7.2}{\rmdefault}{\mddefault}{\updefault}{\color[rgb]{0,0,0}$2$}%
}}}}
\put(7516,-11626){\makebox(0,0)[lb]{\smash{{\SetFigFont{6}{7.2}{\rmdefault}{\mddefault}{\updefault}{\color[rgb]{0,0,0}$\frac{1}{3}$}%
}}}}
\put(7111,-10141){\makebox(0,0)[lb]{\smash{{\SetFigFont{6}{7.2}{\rmdefault}{\mddefault}{\updefault}{\color[rgb]{0,0,0}$g_2$}%
}}}}
\put(8551,-9556){\makebox(0,0)[lb]{\smash{{\SetFigFont{6}{7.2}{\rmdefault}{\mddefault}{\updefault}{\color[rgb]{0,0,0}$g_3$}%
}}}}
\put(8371,-11041){\makebox(0,0)[lb]{\smash{{\SetFigFont{6}{7.2}{\rmdefault}{\mddefault}{\updefault}{\color[rgb]{0,0,0}$3$}%
}}}}
\put(1306,-10186){\makebox(0,0)[lb]{\smash{{\SetFigFont{6}{7.2}{\rmdefault}{\mddefault}{\updefault}{\color[rgb]{0,0,0}$t^{-1}g_2t$}%
}}}}
\put(12961,-11626){\makebox(0,0)[lb]{\smash{{\SetFigFont{6}{7.2}{\rmdefault}{\mddefault}{\updefault}{\color[rgb]{0,0,0}$\frac{4}{3}$}%
}}}}
\put(12331,-10186){\makebox(0,0)[lb]{\smash{{\SetFigFont{6}{7.2}{\rmdefault}{\mddefault}{\updefault}{\color[rgb]{0,0,0}$tg_2t^{-1}$}%
}}}}
\put(13501,-9601){\makebox(0,0)[lb]{\smash{{\SetFigFont{6}{7.2}{\rmdefault}{\mddefault}{\updefault}{\color[rgb]{0,0,0}$tg_3t^{-1}$}%
}}}}
\put(3556,-9646){\makebox(0,0)[lb]{\smash{{\SetFigFont{6}{7.2}{\rmdefault}{\mddefault}{\updefault}{\color[rgb]{0,0,0}$t^{-1}g_3t$}%
}}}}
\put(7651,-11086){\makebox(0,0)[lb]{\smash{{\SetFigFont{6}{7.2}{\rmdefault}{\mddefault}{\updefault}{\color[rgb]{0,0,0}$2$}%
}}}}
\put(9586,-8251){\makebox(0,0)[lb]{\smash{{\SetFigFont{6}{7.2}{\familydefault}{\mddefault}{\updefault}{\color[rgb]{0,0,0}Fundamental domain for $\Gamma_{18.3^4.2^3}$}%
}}}}
\put(7066,-8836){\makebox(0,0)[lb]{\smash{{\SetFigFont{6}{7.2}{\rmdefault}{\mddefault}{\updefault}{\color[rgb]{0,0,0}$t^3$}%
}}}}
\put(8416,-5506){\makebox(0,0)[lb]{\smash{{\SetFigFont{6}{7.2}{\rmdefault}{\mddefault}{\updefault}{\color[rgb]{0,0,0}$g_6$}%
}}}}
\put(676,-4201){\makebox(0,0)[lb]{\smash{{\SetFigFont{6}{7.2}{\rmdefault}{\mddefault}{\updefault}{\color[rgb]{0,0,0}$g_6=\abcd{1}{0}{6}{1}$}%
}}}}
\put(11116,-7201){\makebox(0,0)[lb]{\smash{{\SetFigFont{6}{7.2}{\rmdefault}{\mddefault}{\updefault}{\color[rgb]{0,0,0}$1$}%
}}}}
\put(8461,-7201){\makebox(0,0)[lb]{\smash{{\SetFigFont{6}{7.2}{\rmdefault}{\mddefault}{\updefault}{\color[rgb]{0,0,0}$0$}%
}}}}
\put(5671,-7261){\makebox(0,0)[lb]{\smash{{\SetFigFont{6}{7.2}{\rmdefault}{\mddefault}{\updefault}{\color[rgb]{0,0,0}$-1$}%
}}}}
\put(676,-3346){\makebox(0,0)[lb]{\smash{{\SetFigFont{6}{7.2}{\rmdefault}{\mddefault}{\updefault}{\color[rgb]{0,0,0}$s=\abcd{1}{0}{-3}{1}$}%
}}}}
\put(7966,-7216){\makebox(0,0)[lb]{\smash{{\SetFigFont{6}{7.2}{\rmdefault}{\mddefault}{\updefault}{\color[rgb]{0,0,0}$\frac{1}{5}$}%
}}}}
\put(7471,-7216){\makebox(0,0)[lb]{\smash{{\SetFigFont{6}{7.2}{\rmdefault}{\mddefault}{\updefault}{\color[rgb]{0,0,0}$-\frac{1}{3}$}%
}}}}
\put(316,-7171){\makebox(0,0)[lb]{\smash{{\SetFigFont{6}{7.2}{\rmdefault}{\mddefault}{\updefault}{\color[rgb]{0,0,0}$-3$}%
}}}}
\put(3016,-7171){\makebox(0,0)[lb]{\smash{{\SetFigFont{6}{7.2}{\rmdefault}{\mddefault}{\updefault}{\color[rgb]{0,0,0}$-2$}%
}}}}
\put(12961,-7171){\makebox(0,0)[lb]{\smash{{\SetFigFont{6}{7.2}{\rmdefault}{\mddefault}{\updefault}{\color[rgb]{0,0,0}$\frac{5}{3}$}%
}}}}
\put(13321,-7171){\makebox(0,0)[lb]{\smash{{\SetFigFont{6}{7.2}{\rmdefault}{\mddefault}{\updefault}{\color[rgb]{0,0,0}$\frac{9}{5}$}%
}}}}
\put(16561,-7201){\makebox(0,0)[lb]{\smash{{\SetFigFont{6}{7.2}{\rmdefault}{\mddefault}{\updefault}{\color[rgb]{0,0,0}$3$}%
}}}}
\put(13816,-7171){\makebox(0,0)[lb]{\smash{{\SetFigFont{6}{7.2}{\rmdefault}{\mddefault}{\updefault}{\color[rgb]{0,0,0}$2$}%
}}}}
\put(8326,-6631){\makebox(0,0)[lb]{\smash{{\SetFigFont{6}{7.2}{\rmdefault}{\mddefault}{\updefault}{\color[rgb]{0,0,0}$6$}%
}}}}
\put(676,-3751){\makebox(0,0)[lb]{\smash{{\SetFigFont{6}{7.2}{\rmdefault}{\mddefault}{\updefault}{\color[rgb]{0,0,0}$g_2=\abcd{1}{2}{0}{1}$}%
}}}}
\put(8146,-4426){\makebox(0,0)[lb]{\smash{{\SetFigFont{6}{7.2}{\rmdefault}{\mddefault}{\updefault}{\color[rgb]{0,0,0}$g_2^3$}%
}}}}
\put(7111,-5686){\makebox(0,0)[lb]{\smash{{\SetFigFont{6}{7.2}{\rmdefault}{\mddefault}{\updefault}{\color[rgb]{0,0,0}$g_1$}%
}}}}
\put(631,-5731){\makebox(0,0)[lb]{\smash{{\SetFigFont{6}{7.2}{\rmdefault}{\mddefault}{\updefault}{\color[rgb]{0,0,0}$g_2^{-1}g_1g_2$}%
}}}}
\put(2836,-5551){\makebox(0,0)[lb]{\smash{{\SetFigFont{6}{7.2}{\rmdefault}{\mddefault}{\updefault}{\color[rgb]{0,0,0}$g_2^{-1}g_6g_2$}%
}}}}
\put(13681,-5506){\makebox(0,0)[lb]{\smash{{\SetFigFont{6}{7.2}{\rmdefault}{\mddefault}{\updefault}{\color[rgb]{0,0,0}$g_2^{-1}g_6g_2$}%
}}}}
\put(12241,-5731){\makebox(0,0)[lb]{\smash{{\SetFigFont{6}{7.2}{\rmdefault}{\mddefault}{\updefault}{\color[rgb]{0,0,0}$g_2^{-1}g_1g_2$}%
}}}}
\put(9586,-3841){\makebox(0,0)[lb]{\smash{{\SetFigFont{6}{7.2}{\familydefault}{\mddefault}{\updefault}{\color[rgb]{0,0,0}Fundamental domain for $s\Gamma_{9.6^4.1^3}s^{-1}$}%
}}}}
\put(7606,-6676){\makebox(0,0)[lb]{\smash{{\SetFigFont{6}{7.2}{\rmdefault}{\mddefault}{\updefault}{\color[rgb]{0,0,0}$1$}%
}}}}
\end{picture}%

%% file: aswdfinal.bbl
\begin{thebibliography}{99}



\bibitem[ALL05]{ALL}
Atkin, A. O. L.; Li, Wen-Ching Winnie; Long, Ling 
\emph{On Atkin and Swinnerton-Dyer congruence relations. II.}
  Math. Ann.  340  (2008),  no. 2, 335--358.

\bibitem[ASwD71]{ASwD}
A.~O.~L. Atkin and H.~P.~F. Swinnerton-Dyer, \emph{Modular forms
on non-congruence subgroups}, Combinatorics
(Proc. Sympos. Pure Math., Vol. XIX,
  Univ. California, Los Angeles, Calif., 1968), Amer. Math. Soc., Providence,
  R.I., 1971, pp.~1--25.

\bibitem[Beau82]{Beau82}
A. Beauville, \emph{Les familles stable de courbes elliptiques
sur $\mbf{P}^1$ admettant quatre fibres singuli\`eres},
C.R. Acad. Sci. {\bf 294} (1982), 657-660.



\bibitem [CN79]{CN} J. H. Conway and S. P. Norton, 
\emph{Monstrous moonshine.} Bull. London Math. Soc.
\textbf{11}
(1979), no. 3, 308--339.


































\bibitem[Kul91]{kulkarni} R. S. Kulkarni, \emph{An arithmetic-geometric method in the study of the
subgroups of the modular group}. Amer. J. Math. \textbf{113} (1991), no. 6, 1053--1133



\bibitem[LLY03]{LLY}
Li, Wen-Ching Winnie; Long, Ling; Yang, Zifeng 
\emph{On Atkin-Swinnerton-Dyer congruence relations.}
  J. Number Theory  113  (2005),  no. 1, 117--148. 


\bibitem[LLY05]{LLY2}
Li, Wen-Ching Winnie; Long, Ling; Yang, Zifeng 
\emph{Modular forms for noncongruence subgroups.}
  Q. J. Pure Appl. Math.  1  (2005),  no. 1, 205--221.





\bibitem[LL]{LL}
Long, Ling. \emph{On Atkin-Swinnerton-Dyer congruence relations(3)}, 
preprint, arXiv:math/0701310 

\bibitem[BCP97]{magma}
W.~Bosma, J.~Cannon, and C.~Playoust, \emph{The {M}agma algebra system {I}:
{T}he user language}, 1997, {\sf
http://www.maths.usyd.edu.au:8000/u/magma/}, pp.~235--265.

\bibitem[M96]{Martin}
Y.~Martin,
\emph{Multiplicative $\eta$-quotients},
Trans. Amer. Math. Soc. {\bf 348} (1996), no.~12, 4825--4856; MR1376550 (97d:11070)






\bibitem[Pari04]{PARI2}
 PARI/GP, version
{\tt 2.1.5}, Bordeaux, 2004, {\sf http://pari.math.u-bordeaux.fr/}.




\bibitem[Sch85i]{Sch85i}
A. J. Scholl,
\emph{A trace formula for $F$-crystals.} 
Invent. Math., {\bf 79} (1985), 31-48.

\bibitem[Sch85ii]{Sch85ii}{-----,}
\emph{Modular forms and de{R}ham cohomology;
Atkin-Swinnerton-Dyer congruences.} 
Invent. Math., {\bf 79} (1985), 49-77.

\bibitem[Sch87]{Sch87}{-----,}
\emph{Modular forms on noncongruence subgroups.}
S\'eminaire de Th\'eorie des Nombres, Paris 1985-86, 199-206,
Progr. Math. 71, Birkh\"auser, Boston, MA, 1987.

\bibitem[Sch88]{Sch88}{-----,} \emph{The $l$-adic representations attached to a
certain non-congruence subgroup}, J.\ Reine Angew.\ Math.
{\bf 392}
    (1988), 1--15.


\bibitem[Sch93]{Sch93}{-----,} \emph{The $l$-adic representations attached to
non-congruence subgroups II}, preprint, 1993.


\bibitem[Sch97]{Sch97}{-----,}
  \emph{On the Hecke algebra of a noncongruence subgroup},
Bull. London Math.Soc. {\bf 29} (1997), 395-399.




\bibitem[Seb01]{Seb01}{A.\ Sebbar,}
\emph{Classification of torsion-free genus zero
    congruence groups},  Proc.\ Amer.\ Math.\ Soc. {\bf 129} (2001), no. 9,
    2517--2527 (electronic).




\bibitem[Shi71]{shimu}
G.~Shimura, Introduction to the arithmetic theory of
automorphic
  functions, Publications of the Mathematical Society of Japan, No. 11.
  Iwanami Shoten, Publishers, Tokyo, 1971, Kan\^o Memorial Lectures, No. 1.



\bibitem[Stein]{stein}
W.~Stein \emph{The Modular Forms database}
http://modular.fas.harvard.edu/Tables






\end{thebibliography}
